\documentclass[12pt]{amsart}
\usepackage{polyhedra}
\usepackage{bbm}
\usepackage{tikz,tikz-3dplot}
\usepackage{tikz-cd}
\usepackage{todonotes}
\usetikzlibrary{matrix,arrows,decorations.pathmorphing,arrows.meta}
\usepackage{MnSymbol}
\usepackage{calligra,mathrsfs}
\usepackage{latexsym}
\usepackage{pgf}
\usepackage{mathtools}
\usepackage{stmaryrd}
\usepackage{atbegshi}
\usepackage{algpseudocode}
\usepackage{algorithm}
\usepackage{adjustbox}
\usepackage{enumitem}

\usepackage{hyperref}
\hypersetup{
    colorlinks,
}

\usepackage[margin=1.2in]{geometry}
\usepackage{pictex}
\usepackage{amsmath,amscd}
\usepackage{afterpage}
\usepackage{youngtab}
\usepackage{psfrag}
\usepackage[boxsize=.25 em]{ytableau}
\usepackage{verbatim}
\usepackage{graphics}
\usepackage{changepage}

\newtheorem{theorem}{Theorem}[section]
\newtheorem{proposition}[theorem]{Proposition} 
\newtheorem{definition}[theorem]{Definition}
\newtheorem{lemma}[theorem]{Lemma}
\newtheorem{corollary}[theorem]{Corollary} 
\newtheorem{example}[theorem]{Example}
\newtheorem{remark}[theorem]{Remark}

\newtheorem{maintheorem}{Theorem}

\makeatletter
\newcommand*{\da@rightarrow}{\mathchar"0\hexnumber@\symAMSa 4B }
\newcommand*{\da@leftarrow}{\mathchar"0\hexnumber@\symAMSa 4C }
\newcommand*{\xdashrightarrow}[2][]{%
  \mathrel{%
    \mathpalette{\da@xarrow{#1}{#2}{}\da@rightarrow{\,}{}}{}%
  }%
}
\newcommand{\xdashleftarrow}[2][]{%
  \mathrel{%
    \mathpalette{\da@xarrow{#1}{#2}\da@leftarrow{}{}{\,}}{}%
  }%
}
\newcommand*{\da@xarrow}[7]{%
  \sbox0{$\ifx#7\scriptstyle\scriptscriptstyle\else\scriptstyle\fi#5#1#6\m@th$}%
  \sbox2{$\ifx#7\scriptstyle\scriptscriptstyle\else\scriptstyle\fi#5#2#6\m@th$}%
  \sbox4{$#7\dabar@\m@th$}%
  \dimen@=\wd0 %
  \ifdim\wd2 >\dimen@
    \dimen@=\wd2 %
  \fi
  \count@=2 %
  \def\da@bars{\dabar@\dabar@}%
  \@whiledim\count@\wd4<\dimen@\do{%
    \advance\count@\@ne
    \expandafter\def\expandafter\da@bars\expandafter{%
      \da@bars
      \dabar@ 
    }%
  }%
  \mathrel{#3}%
  \mathrel{%
    \mathop{\da@bars}\limits
    \ifx\\#1\\%
    \else
      _{\copy0}%
    \fi
    \ifx\\#2\\%
    \else
      ^{\copy2}%
    \fi
  }%
  \mathrel{#4}%
}
\makeatother

\newcommand{\RR}{\mathbb{R}}

\newcommand{\ZZ}{\mathbb{Z}}
\newcommand{\NN}{\mathbb{N}}

\newcommand{\Def}{\operatorname{Def}}
\newcommand{\Perm}{\operatorname{Perm}}
\newcommand{\D}{\operatorname{D}^{\text{b}}}
\newcommand{\kk}{\mathbbm{k}}
\newcommand{\Poly}{\operatorname{Poly}}
\newcommand*{\sheafhom}{\mathscr{H}\kern -.5pt om}

\newcommand{\rk}{\operatorname{rk}}

\newcommand{\sat}{\operatorname{sat}}
\newcommand{\one}{\mathbbm{1}}
\newcommand{\Coh}{\operatorname{Coh}}

\newcommand{\BP}{\operatorname{BP}}
\newcommand{\IP}{\operatorname{IP}}
\newcommand{\FP}{\operatorname{FP}}
\newcommand{\Stell}{\operatorname{Stell}}
\newcommand{\Sch}{\operatorname{Sch}}
\newcommand{\McM}{\operatorname{McM}}
\newcommand{\Hom}{\operatorname{Hom}}

\newcommand{\IE}{\operatorname{IE}}

\title[Derived Categories of Permutahedral Varieties]{Derived Categories of Permutahedral and Stellahedral Varieties}

\author{Mario Sanchez}
\date{}
\thanks{}

\begin{document}
\begin{abstract}

We utilize the coherent-constructible correspondence to construct full strongly exceptional collections of nef line bundles in the derived category of a toric variety through the combinatorics of constructible sheaves built from polytopes. To show that sequences are full, we build exact complexes of line bundles that categorify the relations in the McMullen polytope algebra. We compute the homomorphisms between certain constructible sheaves on polytopes and use this to reduce the question of exceptionality to showing that certain set differences of polytopes are contractible.

As an application of our method, we construct full strongly exceptional collections of nef line bundles for the toric varieties associated to the permutahedron, stellahedron, and the type $B_n$ Coxeter permutahedron. The line bundles in our collections are indexed by base polytopes of loopless Schubert matroids, independence polytopes of all Schubert matroids, and feasible polytopes of loopless Schubert delta matroids, respectively.

Our collections satisfy a number of nice properties: First, the quiver with relations that encodes the endomorphism algebra of the tilting sheaf can be described matroid-theoretically as a slight extension of the notion of weak maps and inclusion of matroids; Second, our collections are invariant under the natural symmetries of the corresponding fans; Finally, the induced semi-orthogonal decomposition of the derived categories refines the cuspidal semi-orthogonal decomposition as studied by Castravet and Tevelev. This gives a full strongly exceptional collection of nef line bundles for the cuspidal parts of the derived categories of our varieties indexed by loopless and coloopless Schubert matroids and Schubert delta matroids.

\end{abstract}
\maketitle
\tableofcontents
\section{Introduction}

The derived category $\D(X)$ of a smooth variety $X$ has proven to be an important invariant that reveals subtle connections between many disparate fields. Despite its importance, it is a difficult invariant to calculate and explicitly describe. One approach to describing this category is to find a sheaf $\mathcal{T}$, called a \textbf{tilting sheaf}, that classically generates the derived category, has $\operatorname{Ext}^i(\mathcal{T},\mathcal{T}) = 0$ for all $i > 0$, and for which its endomorphism algebra $A = \operatorname{End}_{\mathcal{O}}(\mathcal{T}, \mathcal{T})$ has finite global homological dimension. Given such a sheaf, there is an equivalence of categories $\D(X) \cong \D(\operatorname{mod}(A^{op}))$ determined by the functor
 \[\operatorname{RHom}(\mathcal{T}, -): \D(X) \to \D(\operatorname{mod}(A^{op})).\]

If the tilting sheaf decomposes into line bundles $\mathcal{T} = \bigoplus_{i=1}^{\ell} \mathcal{L}_{i}$, then the line bundles satisfy the following three properties:

\begin{itemize}
    \item (Fullness) The smallest thick triangulated subcategory of $\D(X)$ containing $\{\mathcal{L}_i\}$ is all of $\D(X)$.
    \item (Exceptionality) There exists some linear order on our line bundles $(\mathcal{L}_1,\ldots, \mathcal{L}_{\ell})$ such that
     \[ \operatorname{Ext}^p(\mathcal{L}_i,\mathcal{L}_j) = 0,\]
    for all $p \geq 0$ and $i > j$.
    \item (Strong Exceptionality) For all $i,j$, we have
     \[ \operatorname{Ext}^p(\mathcal{L}_{i}, \mathcal{L}_j) =0\]
     for all $p \geq 1$.
\end{itemize}
Conversely, if we have such a collection, then their sum $\bigoplus_i \mathcal{L}_i$ is a tilting sheaf. A collection of line bundles that satisfy all of these properties is known as a \textbf{full strongly exceptional collection of line bundles} for $\D(X)$. The algebra $A$ is called the \textbf{tilting algebra} of the collection. Tilting sheaves do not in general exist for smooth and complete varieties.

In \cite{King97}, King conjectured that all smooth toric varieties had a full strongly exceptional collection of line bundles, and hence, a tilting sheaf. This was ultimately proven false by Hille and Perling in \cite{Hille2006}. Later, Micha\l{}ek found an infinite family of counter-examples in \cite{Mich2011}. Despite this general failure, there are many cases in which this conjecture holds \cite{Costa2012, Hara2017, Larson}. The focus on the conjecture has shifted to trying to find conditions on the toric varieties that ensure existences of tilting sheaves.

In a weaker direction, it is also useful to find full exceptional collections of $\D(X)$ which are collections of sheaves that satisfy the fullness and exceptionality conditions above and such that
\[\operatorname{Ext}^p(\mathcal{E}, \mathcal{E}) = \begin{cases}
    \kk & \text{if $p =0$}\\
    0 & \text{otherwise.}
\end{cases} \]
for all sheaves in the collection. Such a collection gives a semi-orthogonal decomposition of the derived category and implies various nice properties of cohomological invariants. The most common method for constructing a full exceptional collection is to use results of Orlov \cite{Orlov1993} to transfer full exceptional collections from a variety to a blow-up of the variety or a projective bundle over the variety. However, this procedure destroys exceptionality and does not in general output line bundles, even if one begins with line bundles, and so it is not useful for studying King's conjecture.

In this paper, we give a new combinatorial and convex-geometric method for constructing full strongly exceptional collections of line bundles in terms of the polytopal structure of the nef line bundles of the variety. We use this to construct such collections for the permutahedral variety, the stellahedral variety, and the type $B_n$ permutahedral variety further increasing the class of examples for which King's conjecture holds.

\subsection{Valuations and the Coherent-Constructible Correspondence}
Our main inspiration is the combinatorial theory of valuations of polytopes and the McMullen polytope ring. 

Let $\Sigma$ be a strictly convex polyhedral fan in a lattice $\ZZ^n$ and $\Def(\Sigma)$ be the set of lattice polytopes in $\RR^n$ whose normal fan coarsens $\Sigma$.

The $\ZZ$-module of indicator functions $\mathbb{I}(\Sigma)$ is the $\ZZ$-module spanned by the functions $\one_P: \RR^n \to \ZZ$ given by
 \[ \one_P(x) = \begin{cases}
     1 & \text{ if $x \in P$} \\
     0 & \text{otherwise,}
 \end{cases}\]
for all $P \in \Def(\Sigma)$. This forms a ring with multiplication $\one_{P} \cdot \one_{Q} = \one_{P+Q}$ given by Minkowski sum of polytopes. For our purposes\footnote{The general definition of the polytope ring is different but it coincides with this definition in our examples.}, the \textbf{McMullen polytope ring} $\McM(\Sigma)$ is the quotient of $\mathbb{I}(\Sigma)$ by the translation action of the lattice $M$. 

The main utility of the McMullen polytope ring in combinatorics is that many important combinatorial invariants, known as translation-invariant valuations, factor through this ring, see \cite{Ardila2023} for many examples. Therefore, finding bases of these rings amounts to finding universal invariants from which all other translation-invariant valuations can be calculated.

For a projective toric variety $X_{\Sigma}$, there is a bijection between $T$-equivariant nef line bundles of $X_{\Sigma}$ and polytopes in $\Def(\Sigma)$. We will use $\mathcal{L}_P$ to denote the nef line bundle associated to $P$. Morelli proved\footnote{Morelli's original statement is in a different form. See the appendix of \cite{EHL22} for more information.} the following beautiful result relating the convex geometry of $\McM(\Sigma)$ and the K-theory of $X_{\Sigma}$.
\begin{theorem}[\cite{Morelli1993}]\label{thm: Morelli's Theorem}
Let $\Sigma$ be a smooth, projective toric variety with character lattice $M \cong \ZZ^n$. There are isomorphisms
\[K_0^T(X_{\Sigma}) \cong \mathbb I(\Def(\Sigma)) \quad \quad \text{and} \quad \quad K_0(X_\Sigma) \cong \mathbb{I}(\Def(\Sigma))/M \]
given by $[\mathcal{L}_P]_T \mapsto \one_P$ and $[\mathcal{L}_P] \mapsto [\one_P]$, respectively.
\end{theorem}

This relationship between K-theory and indicator functions of polytopes was categorified in \cite{Fang2011} by Fang, Melissa-Liu, Treumann, and Zaslow as a dg-enhanced equivariant version of the coherent-constructible correspondence first announced by Bondal \cite{Bondal2006}.  See Section \ref{Sec: Background} for descriptions of the categories involved.

\begin{theorem}\cite{Fang2011}
    Let $\Sigma$ be a projective fan. Then, there is an embedding of triangulated dg-categories
     \[\kappa: \operatorname{Perf}_T(X_{\Sigma}) \to \operatorname{Sh}_{cc}(\RR^n). \]
    Further, the image of this functor is the subcategory $\operatorname{Sh}_{cc}(\RR^n,\Lambda_{\Sigma})$ of constructible sheaves whose singular support lies in a particular conical Lagrangian $\Lambda_{\Sigma} \subseteq T^*(\RR^n)$ associated to the fan.
\end{theorem}

When $\Sigma$ is smooth and projective, then the K-theory of the left-hand side is the algebraic K-theory of the variety and the K-theory of the right-hand side is isomorphic to $\mathbb{I}(\Def(\Sigma))$ by the map that sends a class in $K(\operatorname{Sh}_{cc}(\RR^n,\Lambda{\Sigma}))$ to the function $f: \RR^n \to \ZZ$ where $f(x)$ is the Euler characteristic of the stalk at $x$ of the class. This categorifies Morelli's theorem in the sense that the induced map on K-theory $K(X_{\Sigma}) \to \mathbb{I}(\Sigma)$ is the map $[\mathcal{L}_P] \mapsto \one_P$. One can obtain the non-equivariant part of Morelli's theorem through Bondal's non-equivariant version of the coherent-constructible correspondence.

The coherent-constructible correspondence translates the study of coherent sheaves on $X_{\Sigma}$ to the study of certain constructible sheaves supported on polyhedra on $\RR^n$. In particular, for a projective toric variety $X_\Sigma$, where $\Sigma$ is the normal fan of a polytope $Q$, the nef line bundle $\mathcal{L}_P$ is mapped to a complex of sheaves supported on translations of the tangent cones of $Q$. This is a categorification of the Brianchon-Gram formula which expresses an indicator function of a polytope as an alternating sum of indicator functions of its tangent cones.

Our approach exploits the fact that constructible sheaves on polytopes behave similarly to indicator functions of polytopes in order to transport the combinatorial intuition of $\mathbb{I}(\Sigma)$ into $\operatorname{Perf}_T(X_{\Sigma})$.

\subsection{Convex Geometric Approach to Full Strongly Exceptional Collections}

Our method involves two distinct ideas: the first is to use the homotopy theory of set differences of polytopes to prove strong exceptionality, and the second is to use a categorification of the defining relations of the McMullen polytope algebra to prove fullness.

For the first, we use constructible sheaves and the homotopy theory of set differences of polytopes to calculate the homomorphisms between nef line bundles in the dg categories $\operatorname{Perf}_T(X_{\Sigma})$ and $\operatorname{Perf}(X_{\Sigma})$.

\begin{maintheorem}\label{mainthm: dg homs calculation}
    Let $\Sigma$ be a projective fan. Let $P$ and $Q$ be two polytopes in $\Def(\Sigma)$. Let $M$ be the weight lattice of $X_{\Sigma}$. Then, we have quasi-isomorphisms
     \[ \hom_T(\mathcal{L}_P, \mathcal{L}_Q) \cong \widetilde{C}(P \backslash Q)[-1] \]
    in $\operatorname{Perf}_T(X_{\Sigma})$ and
    \[ \hom(\mathcal{L}_P, \mathcal{L}_Q) \cong \bigoplus_{m \in M} \widetilde{C}(P \backslash Q + m)[-1]\]
    in $\operatorname{Perf}(X_{\Sigma})$, where $\widetilde{C}(U)$ is the reduced singular cochain complex of $U$ over $\kk$.
\end{maintheorem}
By taking cohomologies, this recovers the calculation of extensions of nef line bundles given by \cite{Altmann2023}.

Let $\leq_M$ be the order on polytopes in $\Def(\Sigma)$ given by $P \leq_M Q$ if there exists $m \in M$ so that $P + m \subseteq Q$. With this, we say that a collection of polytopes $\{P_i\}$ is \textbf{exceptional} if there is some linear order $(P_1, \ldots, P_{\ell})$ that is a linear extension of the order $\leq_M$ and such that for all $i < j$, we have $\widetilde{H}^k(P_j \backslash P_i + m) = 0$ for all $m \in M$ and all $k \geq 0$. We say it is \textbf{strongly exceptional} if it satisfies the conditions $\widetilde{H}^{k}(P_j \backslash P_i + m)$ for all $i$ and $j$. Any strongly exceptional collection of polytopes can be ordered to be exceptional, say by non-decreasing number of lattice points. Exceptional collections of polytopes give rise to exceptional collections of nef line bundles. This reduces the study of exceptionality to the interesting problem of showing that set differences of polytopes have no reduced cohomology. We will achieve this in all of our examples by showing that the set differences are contractible.

In order to study fullness of collections, we categorify the valuative relations of the ring of indicator functions by using Bj\"orner's theory of CW posets. Recall that a CW poset is a poset $\mathbb{P}$ of cardinality greater than $1$ with a minimal element $\hat{0}$ such that the realization of the order complex of any interval $(\hat{0},x)$ for $x \in \mathbb{P}\backslash \{\hat{0}\}$ is homeomorphic to a sphere. We will view $\Def(\Sigma)$ as a poset under (the usual) inclusion of polytopes.

For any CW poset $\mathbb{P}$ and order-preserving map $F: \mathbb{P} \to \Def(\Sigma)$, we associate a complex $\IE^{\bullet}$ with 
 \[\IE^{k} = \bigoplus_{\substack{x \in \mathbb{P} \\ \rk(x) = k}} \mathcal{L}_{F(x)}\] which we call the \textbf{inclusion-exclusion complex} of $F$. By studying the stalks of the image of this complex under the coherent-constructible correspondence, we give a simple criterion for checking exactness of these complexes.
\begin{maintheorem}\label{mainthm: Exactness Criterion}
    Let $\mathbb{P}$ be a CW poset and $f: \mathbb{P} \to \Def(\Sigma)$ be an order-preserving map such that for each $m \in M$, the upper ideal $I_m = \{x \in \mathbb{P} \; \lvert \; m \in F(x)\}$ has $|I_m| \geq 2$. Then, $\IE^{\bullet}$ is an exact complex.
\end{maintheorem}
As a special case, these include the Koszul complexes of polytopes studied by \cite{Altmann2023}.

Given a collection of polytopes $\mathcal{P}$, we iteratively construct a family of polytopes $\langle \mathcal{P} \rangle$. First, we set $\langle \mathcal{P} \rangle_0$ to be the set of polytopes $\{P + m \; \lvert \; P \in \mathcal{P}, m \in M\}$. Then, we define $\langle \mathcal{P} \rangle_k$ by $Q \in \langle \mathcal{P} \rangle_k$ if there exists a CW poset $\mathbb{P}$ and an order-preserving map $F: \mathbb{P} \to \Def(\Sigma)$ with the following properties:

\begin{enumerate}
    \item there is an $x$ with $F(x) = Q$
    \item for all $y$ with $\rk(y) \not = \rk(x)$ we have $F(y) \in \langle \mathcal{P} \rangle_{k-1}$.
    \item For any $m \in M$, the upper ideal $\{x \in \mathbb{P} \; \lvert \; m \in F(x)\}$ is an interval of cardinality greater than $1$.
\end{enumerate}
Define the set of polytopes generated by $\mathcal{P}$ as the set $\langle \mathcal{P} \rangle = \bigcup_{i=0}^{\infty} \langle \mathcal{P} \rangle_i$. We say that $\mathcal{P}$ is \textbf{full} if $\langle \mathcal{P} \rangle = \Def(\Sigma)$. 

Theorems A and Theorems B then imply the following criterion for constructing full strongly exceptional collections.

\begin{maintheorem}\label{mainthm: Polytopal Criterion}
      Let $\Sigma$ be a smooth projective fan. If an ordered collection of polytopes $(P_1,\ldots, P_m)$ for $P_i \in \Def(\Sigma)$ is a full (strongly) exceptional collection of polytopes in $\Def(\Sigma)$, then the ordered collection of nef line bundles $(\mathcal{L}_{P_1}, \ldots, \mathcal{L}_{P_m})$ is a full (strongly) exceptional collection for $\D(X_{\Sigma})$.
\end{maintheorem}

We demonstrate the method in small examples in Section \ref{sec: examples}. 

\subsection{The Permutahedral Varieties}

In the second half of this paper, we apply our method to the permutahedral variety, the stellahedral variety, and the type $B_n$-analogue of the permutahedral variety.  These beautiful varieties appear in many different contexts including in the study of Chow quotients of Grassmannians, compactifications of $M_{0,n}$, Hesenberg varieties, wonderful compactifications of hyperplane arrangements, and flag varieties. Most notably, they have played pivotal roles in the resolution of the Rota-Welsh conjecture and related conjectures \cite{Katz2016,Braden2022}. 

From the perspective of combinatorics, the class of polytopes that index the nef line bundles of these varieties include polytopal models of virtually all objects studied in algebraic combinatorics including matroids, graphs, posets, simplicial complexes, Schubert polynomials, Schur functions, and cluster complexes. This gives a natural context for studying relationships between different combinatorial objects. 

Chief among these objects is the class of matroids and its generalizations which have recently taken the spotlight in algebraic combinatorics. Given a matroid $M$ on ground set $[n]$ with bases $\mathcal{B}(M)$ and independent sets $\mathcal{I}(M)$, we will use the two type of polytopes
 \[\BP(M) = \operatorname{conv}(e_B \; \lvert \; B \in \mathcal{B}(M)) \quad \quad \IP(M) = \operatorname{conv}(e_I \; \lvert \; I \in \mathcal{I}(M)), \]
where $e_A = \sum_{i \in A} e_i$ for the standard basis vectors $e_i \in \RR^n$. 

We will also encounter the theory of delta matroids, which is a generalization of matroids first studied in optimization. These are related to the type $B_n$-permutahedral variety since they are in bijection with the minuscule type $B_n$ Coxeter matroids. One can associate a polytope to a delta matroid $\mathcal{M}$ with feasible sets $\mathcal{F}(\mathcal{M})$ by
\[ \FP(\mathcal{M}) = \operatorname{conv}(e_B \; \lvert \; B \in \mathcal{F}(\mathcal{M})). \]

The nef line bundles in our collection are indexed by the classes of Schubert matroids and Schubert delta matroids. These are the matroids whose polytopes appear as moment map images of Schubert varieties (over all possible Schubert decompositions) in the Grassmannian and Lagrangian Grassmannian. 

We are now ready to state our main calculations.

\begin{maintheorem}\label{mainthm: FSEC for Perm}
The sequence of line bundles $(\mathcal{L}_{\BP(\Omega)})$ indexed all loopless Schubert matroids on $[n]$ is a full strongly exceptional collection for $\D(\Perm_n)$ when ordered by non-decreasing number of lattice points.
\end{maintheorem}

\begin{maintheorem}\label{mainthm: FSEC for Stell}
The sequence of line bundles $(\mathcal{L}_{\IP(\Omega)})$ indexed by all Schubert matroids $\Omega$ on $[n]$ is a full strongly exceptional collection for $\D(\Stell_n^B)$ when ordered by non-decreasing number of lattice points.

\end{maintheorem}

\begin{maintheorem}\label{mainthm: FSEC for PermB}
The sequence of line bundles $(\mathcal{L}_{\FP(\Delta)})$ indexed by loopless Schubert delta matroids on $[n]$ is a full strongly exceptional collection for $\D(\Perm_n^B)$ when ordered by non-decreasing number of lattice points.
\end{maintheorem}

Showing that these collections are full is the easier part since we can rely on the literature to guide our constructions. In fact, these results can be seen as categorifications of the work of Derksen-Fink \cite{Derksen2012}, Eur-Huh-Larson \cite{EHL22} and Eur-Fink-Larson-Spink \cite{Eur2022} where they calculate various bases for the cohomology/K-rings of the varieties above.

A standard result in the theory of submodular function states that any generalized permutahedra is subdivided into base polytopes of matroids by intersecting with all lattice translations of cubes $[0,1]^n$. This gives an inclusion-exclusion complex which reduces the general case to the case of matroids. Further, Derksen and Fink showed that the indicator function of every $\BP(M)$ can be expressed as an alternating sum of $\BP(\Omega)$ over all Schubert matroids $\Omega$ on $[n]$ by manipulating the Brianchon-Gram formula \cite{Derksen2012}. We categorify this identity by showing that it is the relation coming from an inclusion-exclusion complex obtained by truncating the Brianchon-Gram complex. Eur-Huh-Larson \cite{EHL22} and Eur-Fink-Larson-Spink \cite{Eur2022} showed that this strategy continues to work for the stellahedral variety and the type $B_n$-permutahedral variety, respectively, and so the same categorification works for these varieties.

Showing that these collections are exceptional is more interesting. In Sections \ref{Sec: Schubert matroids are Exceptional}, \ref{Sec: Independence Polytopes of Schubert matroids are exceptional}, and \ref{sec: delta matroids are exceptional}, we construct deformation retracts of these polytopes to maximal vertices of the polytopes inspired by the greedy algorithm, the simplex method, and the various flow algorithms of polymatroids. In these deformation retracts, the points flow in paths that maximize certain linear functionals. By exploiting the optimization properties of Schubert matroids, we can show that these homotopy equivalences restrict to the set difference of polytopes.

\subsection{Tilting Quivers}
As a consequence of our calculations, the varieties $\Perm_n$, $\Stell_n$, and $\Perm_n^B$ have tilting algebras which can be described matroid-theoretically. To do this, we construct three categories of matroids which are extensions of the usual notions of inclusions of matroids and delta matroids and weak maps of matroids.

\begin{definition}

    \begin{enumerate}
        \item The \textbf{category of augmented inclusions of matroids} is the category whose objects are matroids on $[n]$ and whose morphisms are
         \[M_1 \xlongrightarrow{S} M_2, \]
        whenever $S$ is a subset of loops of $M_1$ and $B \cup S \in \mathcal{B}(M_2)$ for all $B \in \mathcal{B}(M_2)$.
        \item The \textbf{category of augmented weak maps of matroids} is the category whose objects are matroids on $[n]$ and whose morphisms are
         \[M_1 \xlongrightarrow{S} M_2, \]
        whenever $S$ is a subset of loops of $M_1$ and $I \cup S \in \mathcal{I}(M_2)$ for all $I \in \mathcal{I}(M_2)$.
        \item The \textbf{category of augmented inclusions of delta matroids} is the category whose objects are delta matroids on $[n]$ and whose morphisms are
         \[\mathcal{M}_1 \xlongrightarrow{S} \mathcal{M}_2, \]
        whenever $S$ is a subset of loops of $\mathcal{M}_1$ and $B \cup S \in \mathcal{F}(\mathcal{M}_2)$ for all $B \in \mathcal{F}(\mathcal{M}_1)$.
    \end{enumerate}
\end{definition}

Recall that the category algebra of a category $\mathcal{C}$ is the algebra generated by the arrows of $\mathcal{C}$ with relations given by $f_1f_2 \cdots f_k = g_1 g_2 \cdots g_{\ell}$ if and only if $f_1 \circ f_2 \cdots \circ f_k = g_1 \circ g_2 \cdots \circ g_{\ell}$ when these maps are composable.

We prove the following results in the language of quivers with relations in Propositions \ref{prop: tilting quiver for Perm_n}, \ref{prop: tilting quiver for Stell}, and \ref{prop: tilting quiver for Perm_nB}:

\begin{proposition}

The tilting algebras for $\D(\Perm_n), \D(\Stell_n)$, and $\D(\Perm_n^B)$ corresponding to our collections are:

\begin{itemize}
    \item For $\Perm_n$, the category algebra of the full subcategory of the category of augmented inclusions of matroids given by coloopless Schubert matroids on $[n]$.
    \item For $\Stell_n$, the category algebra of the full subcategory of the category of augmented weak maps of matroids given by all Schubert matroids on $[n]$.
    \item For $\Perm_n^B$, the category algebra of the full subcategory of the category of augmented inclusions of delta matroids given by coloopless Schubert delta matroids on $[n]$.
\end{itemize}    
\end{proposition}

Since the ranks of the K-groups of $\Perm_n$, $\Stell_n$ and $\Perm_n^B$ are given by the number of permutations, decorated permutations, and signed permutations of $[n]$, it would be interesting to find a purely permutation-theoretic description of these subcategories/quivers.

\subsection{Invariances of Our Collections}

Any toric variety admits an action given by the symmetries of the defining fan. For the permutahedral variety, this is the $S_2 \times S_n$ action induced by the outer automorphisms of $S_n$. There have been various interesting results on the representations induced by this action on cohomological invariants of $\Perm_n$. For instance, Stembridge showed that the $S_n$-representation induced on $H^*(\Perm_n)$ is a permutation representation \cite{Stembridge1994}. Another example is the result of Abe, Harada, Horiguchi, and Masuda, proved in \cite{Abe17}, that the $S_n$-invariant part of $H^*(\Perm_n)$ is isomorphic to the cohomology ring of the Peterson variety, which is a variety important in the study of quantum cohomology of the flag varieties.

Viewing the derived category as an enhanced cohomological invariant, it is natural to study the induced action on the derived category and on its full exceptional collections. We say that a collection of sheaves $\{ \mathcal{E}_i\}$ is invariant under an action $G$ on the variety if $g \cdot \mathcal{E}_i$ is isomorphic to some $\mathcal{E}_j$ in the collection, where $g \cdot \mathcal{E}$ is the sheaf under the induced action of $G$ on sheaves. Castravet and Tevelev studied the $S_2 \times S_n$ action on the Losev-Manin compactification of $M_{0,n}$ which is isomorphic to $\Perm_n$ and proved the following:

\begin{theorem}\cite{Castravet2020}
    There exists a full exceptional collection of sheaves on $\Perm_n$ that is invariant under the action of $S_2 \times S_n$ on $\Perm_n$.
\end{theorem}

We strengthen this result by showing that the collection can be strongly exceptional and can consist of only line bundles.

\begin{proposition}
    The full strongly exceptional collection of nef line bundles from Theorem \ref{mainthm: FSEC for Perm} for $\Perm_n$ is $S_2 \times S_n$-invariant.
\end{proposition}

In a similar vein, the stellahedral variety admits an $S_n$-action and the type $B_n$-permutahedral variety admits an $S_n^B$ action, where $S_n^B$ is the Weyl group of type $B_n$. With this action, we have the two results

\begin{proposition}
    The full strongly exceptional collection of nef line bundles from Theorem \ref{mainthm: FSEC for Stell} for $\Stell_n$ is $S_n$-invariant.
\end{proposition}

\begin{proposition}
    The full strongly exceptional collection of nef line bundles from Theorem \ref{mainthm: FSEC for PermB} for $\Perm_n^B$ is $S_n^B$-invariant.
\end{proposition}

\subsection{Cuspidal Decomposition and a Curious Equicardinality}

For a variety $X$ together with a collection of maps $\pi_i: X \to Y_i$, Castravet and Tevelev defined the \textbf{cuspidal derived category} $\D_{cusp}(X)$ of $X$ as the subcategory of $\D(X)$ consisting of all sheaves $\mathcal{E}$ such that $R\pi_{i*} \mathcal{E} = 0$ for all the maps $\pi_i$ \cite{Castravet2020}. If the maps $\pi_i$ satisfy certain mild conditions, then the category $\D(X)$ has a semi-orthogonal decomposition called the \textbf{cuspidal decomposition} consisting of the subcategories $\D_{cusp}(X)$ and $\pi_i^*\D_{cusp}(Y_i)$.

For each of our varieties, we construct natural forgetful maps $\pi_S$ using the recursive structure of the fans and Joyal's theory of species. In the various moduli space interpretations of our varieties, these are the maps that forget all marked points not in $S$. We calculate the cuspidal derived categories of our varieties given by these maps.
\begin{maintheorem}\label{mainthm: cuspidal part fsec}
    Let $\D_{cusp}(\Perm_n)$,$\D_{cusp}(\Stell_n)$, and $\D_{cusp}(\Perm_n^B
    )$ be the cuspidal categories of the varieties with respect to the forgetful maps defined in Proposition \ref{prop: forgetful maps for permutahedral varieties}. Then,
    \begin{enumerate}
        \item The line bundles $(\mathcal{L}_{\BP(\Omega)}^{-1})$ indexed by loopless and coloopless Schubert matroids on $E$ give a full strongly exceptional collection for $\D_{cusp}(\Perm_n)$ when ordered by non-increasing number of lattice points.
        \item The line bundles $(\mathcal{L}_{\IP(\Omega)}^{-1})$ indexed by loopless and coloopless Schubert matroids give a full strongly exceptional collection for $\D_{cusp}(\Stell_n)$ when ordered by non-increasing number of lattice points.
        \item The line bundles $(\mathcal{L}_{\FP(\Delta)}^{-1})$ indexed by loopless and coloopless Schubert delta matroids give a full strongly exceptional c for $\D_{cusp}(\Perm_n^B)$ when ordered by non-increasing number of lattice points.
    \end{enumerate}
\end{maintheorem}
We also show that our maps satisfy the conditions needed to have a cuspidal decomposition. This decomposition gives a recurrence for the ranks of the cuspidal derived categories. In the permutahedral variety case, Castravet and Tevelev calculated the rank of $K(\D_{cusp}(\Perm_n))$ as the number of derangements. Recall that a derangement of $[n]$ is a permutation with no fixed points and that a signed derangement $\pi$ of $[\pm n]$ is a signed permutation with no $i \in [\pm n]$ such that $\pi(i) = i$. 

By the same method, we have the following calculation of the ranks of the cuspidal parts:
\begin{proposition}
    The rank of $K(\D_{cusp}(\Perm_n))$ and $K(\D_{cusp}(\Stell_n))$ is the number of derangements of $[n]$ and the rank of $K(\D_{cusp}(X_{B_n}))$ is the number of signed derangements of $[\pm n]$. 
\end{proposition}

This gives two curious equicardinalities.   

\begin{corollary}
    The number of loopless and coloopless Schubert matroids on $[n]$ is the same as the number of derangements.
\end{corollary}
\begin{corollary}
    The number of loopless and coloopless Schubert delta matroids on $[n]$ is the same as the number of signed derangements.
\end{corollary}

Since the matroids indexing these cuspidal parts are all loopless and coloopless, augmented inclusions and weak maps reduce to the usual notation of inclusions and weak maps. Therefore, the tilting algebras of these categories are just the incidence algebras of the posets of inclusions and weak maps of loopless and coloopless Schubert matroids and Schubert delta matroids. It would be interesting to find a description of these algebras as incidence algebras of posets on derangements and signed derangements.

\subsection*{Acknowledgements} We thank Allen Knutson for the many wonderful chats about algebraic geometry as well as for pointing out the existence Karshon and Tolman's paper \cite{Karshon1993} whose similarity with the theory of valuations eventually led us to this paper. We also thank Daniel Halpern-Leistner whose comments about derived categories led to a more self-contained paper and Chris Eur for showing us that the exceptionality of Schubert matroids is harder and more interesting than we first thought.

\section{Geometric Notation and Background}\label{Sec: Background}

In this section, we gather notation and basic results from various fields of geometry. We will work over an algebraically closed field $\kk$ of characteristic $0$ although many results involving the coherent-constructible correspondence work over $\ZZ$. Let $N \cong \ZZ^n$ be a lattice and $M$ be the dual lattice $\Hom(N,\ZZ)$. We will use the notation $M_\RR = M \otimes_{\ZZ} \RR \cong \RR^n$ and similarly for $N_\RR$. 

We will use $[n]$ to denote the set $\{1,\ldots,n\}$. For any subset $S \subseteq [n]$, we use $e_S = \sum_{i \in S} e_i$ where $e_i$ are the standard basis vectors of $\RR^n$.
\subsection{Convex Geometry} For a general reference on convex geometry see [Zig]. A \textbf{polyhedron} in $M_\RR$ is a non-empty intersection of half-spaces. A \textbf{polytope} is a bounded polyhedron. For any $\alpha \in N_\RR$, let $P_{\alpha}$ denote the face of $P$ maximized in direction $\alpha$. We say that $P$ is a \textbf{lattice polytope} if every vertex of $P$ is in $M$. 

A \textbf{(polyhedral) fan} is a collection of cones which are polyhedra in $M_\RR$ such that the intersection of any two cones is a face of both and is contained in the collection. We say that a cone is \textbf{strictly convex} if it does not contain any subspace of $N_{\RR}$. A fan is \textbf{strictly convex} if every cone is strictly convex. A fan is \textbf{complete} if the union of all cones is $M_\RR$. A fan $\Sigma$ \textbf{coarsens} $\Sigma'$ if every cone of $\Sigma$ is a union of cones of $\Sigma'$. 

The \textbf{normal fan} of a polyhedron $P \subseteq M_\RR$ is the polyhedral fan in $N_\RR$ consisting of the cones
 \[\sigma_F = \overline{\{ \alpha \in N_\RR \; \lvert \; P_{\alpha} = F\}} \]
for all faces of $P$. This is complete when $P$ is a polytope.

A fan is \textbf{projective} if it is the normal fan of some lattice polytope. For any projective fan $\Sigma$, the \textbf{deformation cone} of $P$ is the set of lattice polytopes whose normal fan coarsens $\Sigma$. It forms a cone under Minkowski sum of polytopes defined by 
     \[P+Q = \{x + y\; \lvert \; x \in P, y \in Q\}. \]

\subsection{Toric Geometry}\label{Ssec: Convex Geometry and Toric Geometry}
For general references on toric geometry see \cite{Cox2011} and \cite{Oda1987}. We will usually assume that our fan is complete and strictly convex. We will use the following notation:

\begin{itemize}
    \item $X_{\Sigma}$ is the toric variety, defined over $\kk$, associated to the fan $\Sigma$.
    \item We say $\Sigma$ is smooth if $X_{\Sigma}$ is smooth.
    \item For any cone $\sigma \in \Sigma$, let
    \[\sigma^{\vee} = \{x \in M_{\RR} \; \lvert \; \langle x, v \rangle \geq 0 \text{ for all $v \in \sigma$}\}. \]
    \item We denote the set of $k$-dimensional cones of $\Sigma$ by $\Sigma(k)$.
\end{itemize}

To every ray $\rho$ of $\Sigma$, we have a $T$-invariant divisor $D_{\rho}$ of $X_{\Sigma}$ and every $T$-invariant divisor can be expressed as a sum
 \[D = \sum_{\rho \in \Sigma(1)} a_{\rho} D_{\rho}, \]
where $a_{\rho} \in \ZZ$. For any divisor $T$-invariant divisor $D$, let $P_D$ be the polyhedron defined by
 \[P_D = \{x \in M_{\RR} \; \lvert \; \langle x, \rho \rangle \leq a_{\rho}, \text{ for all $\rho \in \Sigma(1)$}\}. \]
By construction, when $P_D$ is a polytope, its normal fan will coarsen the fan $\Sigma$. We say that a function $f: N_{\RR} \to N_{\RR}$ is piece-wise linear on $\Sigma$ if it is linear when restricted to any cone of $\Sigma$. In particular, every piece-wise linear function is determined by the values on the rays.

\begin{proposition}
    The T-invariant divisor $D = \sum_{\rho \in \Sigma(1)} a_{\rho} D_{\rho}$ is nef if and only if the piece-wise linear function $f$ defined by $f(v_{\rho}) = a_\rho$, where $v_{\rho}$ is the primitive lattice point of $\rho$, is convex.  The divisor $D$ is ample if and only if $f$ is strictly convex. In that case, the normal fan of $P$ is $\Sigma$. We will refer to such polytopes as ample.
\end{proposition}

\begin{proposition}
   There is a bijection between $T$-equivariant nef line bundles and the deformation cone $\Def(\Sigma)$. We let $\mathcal{L}_P$ denote the $T$-equivariant nef line bundle corresponding to $P \in \Def(\Sigma)$.

   Further, two nef line bundles $\mathcal{L}_P$ and $\mathcal{L}_Q$ are isomorphic if and only if $P + m = Q$ for some $m \in M$. Hence, the nef cone of $X_\Sigma$ is isomorphic to the quotient of $\Def(\Sigma)$ by the translation action on polytopes.
\end{proposition}

We will use $\mathcal{L}_P$ to denote both the $T$-equivariant nef line bundle and the underlying nef line bundle. When we compute homomorphisms, $\Hom_T$ will denote $T$-equivariant homomorphisms and $\Hom$ will denote non-equivariant homomorphisms. We have the following nice characterization of homomorphisms of nef line bundles: In the $T$-equivariant case,
 \[ \Hom_T( \mathcal{L}_P, \mathcal{L}_Q) = \begin{cases}
     \kk & \text{if $P \subseteq Q$} \\
     0 & \text{otherwise}
 \end{cases} \]
 and non-equivariantly
 \[\Hom \left ( \mathcal{L}_P, \mathcal{L}_Q \right ) \cong \kk[m \in M\; \lvert \; P+m \subseteq Q]. \]

\subsection{Category Theory}

For our purposes, a \textbf{dg category} is a $\kk$-linear additive category $\mathcal{C}$ where for any objects $A,B\in \mathcal{C}$, the homomorphisms $\hom(A,B)$ form a chain complex and for any $C$ the composition maps $\hom(A,B) \otimes \hom(B,C) \to \hom(A,C)$ are chain complex homomorphisms. We will prefer the cohomological convention in this paper. The \textbf{homotopy category} $H^{0}(\mathcal{C})$ of $\mathcal{C}$ is the category with the same objects as $\mathcal{C}$ and with homomorphisms $H^{0}(\hom(A,B))$. We will use $\hom$ to denote homomorphisms in a dg category and $\Hom$ for homomorphisms in an ordinary category.

For a variety $X$, we let $\Coh(X)$ denote the dg localization of the dg category of bounded complexes of coherent sheaves by acyclic complexes. We use $\Coh_G(X)$ to denote the related category of $G$-equivariant sheaves. We will often abuse notation by saying that $\mathcal{E}$ is a sheaf in $\Coh(X)$ when we mean the complex whose $0$th term is $\mathcal{E}$ and every other term is $0$. Let $\operatorname{Perf}_G(X)$ and $\operatorname{Perf}(X)$ denote the subcategories of $\Coh_G(X)$ and $\Coh(X)$ consisting of objects such that on every affine chart they are quasi-isomorphic to a bounded complex of vector bundles. These form triangulated categories.

We will use $\D(X)$ to denote the bounded derived category of the abelian category of coherent sheaves on $X$ and $\D_G(X)$ for the related derived category of $G$-equivariant sheaves. When $X$ is smooth, the homotopy categories of $\operatorname{Perf}(X_{\Sigma})$ and $\operatorname{Perf}_G(X_{\Sigma})$ coincide with $\D(X)$ and $\D_{G}(X)$, respectively.

\subsection{The Coherent-Constructible Correspondence}\label{Sec: Coherent-Constructible Correspondence}

While we will use the language of constructible sheaves, we will not need the exact definitions. Let $\operatorname{Sh}(\RR^n)$ denote the dg-localization at acyclic complexes of the category of complexes of sheaves on $\RR^n$. Let $\operatorname{Sh}_{cc}(\RR^n)$ be the subcategory of compactly supported (complexes of) sheaves which are constructible with respect to some Whitney stratification of $\RR^n$ and whose stalks are perfect complexes of modules. 
Given a map $f: X \to Y$, we will use $f_*, f^*, f_!, f^!$ to denote the dg-enhanced versions of the pushforward, pullback, proper pushforward, and proper pullback maps. These functions in general output a complex of sheaves. They satisfy the usual adjunctions from the six functor formalism.

The \textbf{standard constructible sheaf} associated to an open set $U \subseteq \RR^n$ is the complex given by $i_{U*} \kk$ where $i_{U}: U \hookrightarrow \RR^n$ is the inclusion map and $\kk$ is the constant sheaf on $U$. The \textbf{constandard constructible sheaf} associated to $U$ is the complex $i_{U!} \omega_{U}$ where $\omega_U$ is the orientation sheaf on $U$ shifted to degree $\dim U$.

If $U$ is an open set of a manifold $X$ and $i: V \to X$ is an inclusion. Then,
 \[\Gamma(U,i_{*}R_V) \cong C^*(U \cap Y) \] 
where $R_V$ is the $R$-constant sheaf on $V$ and $C^*(U \cap Y)$ is the $R$-valued singular cochain complex. Likewise,
 \[\Gamma(U,i_!R_Y) \cong C^*(U \cap \overline{Y}, U \cap \partial Y). \]
In \cite{Bondal2006}, Bondal announced a relationship between the derived category of coherent sheaves on a toric variety and constructible sheaves on $\RR^n$ which is referred to as the coherent-constructible correspondence. A dg-enhanced equivariant version of this result was proven by Fang, Melissa Liu, Tremann, and Zaslow in \cite{Fang2011}. The main form we will use is the following:

\begin{theorem}\label{thm: Coherent Constructible Correspondence}\cite{Fang2011}
    Let $\Sigma$ be a projective fan. Then, there is an embedding of triangulated dg-categories
     \[\kappa: \operatorname{Perf}_T(X_{\Sigma}) \to \operatorname{Sh}_{cc}(\RR^n). \]
    Further, the image of this functor is the subcategory $\operatorname{Sh}_{cc}(\RR^n,\Lambda_{\Sigma})$ of constructible sheaves whose singular support lies in a particular conical Lagrangian $\Lambda_{\Sigma} \subseteq T^*(\RR^n)$ associated to the fan.
\end{theorem}

They give a description of the images of $T$-equivariant line bundles under this correspondence. We will only need it for nef line bundles so we focus on that case. For any cone $\sigma \in \Sigma$ and any $m \in M$, let $j: (\sigma^{\vee} + m)^{\circ} \hookrightarrow \RR^n$ be the inclusion map and let $\Theta(\sigma, m) = j_{*} \kk$ be the standard constructible sheaf associated to this open cone. Whenever $(\tau^{\vee} + m) \subseteq (\sigma^{\vee} + m)$, we have a natural restriction map $\rho_{\sigma, \tau}: \Theta(\sigma,m) \to \Theta(\tau,m')$.

\begin{definition}
    Let $P \in \Def(\Sigma)$. For any cone $\sigma \in \Sigma$, let $a_{\sigma}$ be any point in $P$ maximized by a linear functional in the interior of $\sigma$. Then, the \textbf{Brianchon-Gram complex}\footnote{This complex is called a twisted polytope sheaf in \cite{Fang2011} and \cite{Zhou2019} in connection with the twisted polytopes of \cite{Karshon1993}. We choose this name to avoid using both polytopes and twisted polytopes as well as to highlight the connection to the theory of valuations.} is the complex
     \[ \operatorname{BG}^{\bullet}(P) = 
     \left ( \bigoplus_{\sigma_0 \in \Sigma(0)} \Theta(\sigma, a_{\sigma}) \rightarrow \bigoplus_{\sigma_1 \in \Sigma(1)} \Theta(\sigma, a_{\sigma}) \rightarrow \cdots \right ),\]
     with differentials
      \[ d_k = \sum_{\sigma_{k-1} \subseteq \sigma_{k}} \operatorname{sgn}(\sigma_{k-1},\sigma_k) \rho_{\sigma_k, \sigma_{k-1}}.\] 
    The sign $\operatorname{sgn}(\sigma_{k-1},\sigma_k) = \pm 1$ can be chosen so that this is a chain complex (see \cite{Zhou2019}). 
\end{definition}

\begin{proposition}\cite{Fang2011}\label{prop: Brianchon-Gram complex and CCC}
Let $P \in \Def(\Sigma)$ and $\Sigma$ be a projective fan. Then, $\kappa(\mathcal{L}_P)$ is quasi-isomorphic to $\operatorname{BG}^{\bullet}(P)$.
\end{proposition}

We also have the following useful description of the image of duals of ample line bundles which can be seen as categorification of the Brianchon-Gram formula.
\begin{proposition}\cite{Fang2011}\label{prop: ccc standard sheaf quasiisomorphism}
    Let $\mathcal{L}_{P}^{-1}$ be the dual of an ample line bundle on $X_{\Sigma}$. In other words, the normal fan of $P$ is $\Sigma$ and so it is full-dimensional. Let $i_P:P^{\circ} \to \RR^n$ be the inclusion of the interior of $P$ into $\RR^n$. Then, $\kappa(\mathcal{L}_P^{-1})$ is quasi-isomorphic to the standard constructible sheaf $i_{P*}(\kk)$.
\end{proposition}

\subsection{Generators of Derived Categories and Tilting Sheaves}\label{subsec:generators of derived categories and tilting sheaves}

For general references on triangulated categories and derived categories of schemes see \cite{Weibel1994}. A subcategory of a triangulated category $\mathcal{C}$ is \textbf{thick} if it is closed under direct sums, taking direct summands, taking cones of morphisms, and shifts. We say that an object $X \in \mathcal{C}$ \textbf{classically generates $\mathcal{C}$} if the smallest thick subcategory of $X$ which contains $X$ is all of $\mathcal{C}$.

When $\mathcal{C}$ is the derived category of an abelian category $\mathcal{A}$, then a subcategory $\mathcal{D} \subseteq \mathcal{C}$ is closed under taking cones of morphisms if whenever we have a short exact sequence
 \[0 \rightarrow A \rightarrow B \rightarrow C \rightarrow 0 \]
with two of the terms in $\mathcal{D}$, then the third term is also in $\mathcal{D}$. Since every exact sequence splits into a collection of short exact sequences, if we have any exact sequences of $n$ terms where any $n-1$ terms are in $\mathcal{D}$, then all terms are in $\mathcal{D}$. This is how we usually check that an object classically generates the category.

The tilting algebra of a collection can be described as the path algebra of a quiver with relations called the \textbf{tilting quiver}. This quiver can be described algebraic geometrically as follows:

For a sequence of distinct line bundles $\mathcal{L}_1, \ldots, \mathcal{L}_k$, we say that a $T$-invariant section $s \in H^0(X, \mathcal{L}_i^{-1} \otimes \mathcal{L}_j)$ is \textbf{indecomposable} if there are no sections $s_i' \in H^0(X,\mathcal{L}_i^{-1} \otimes \mathcal{L}_k)$ and $s_i'' \in H^0(X,\mathcal{L}_k^{-1} \otimes \mathcal{L}_j)$ such that $\operatorname{div}(s) = \operatorname{div}(s') + \operatorname{div}(s'')$ for any $1 \leq k \leq n$.

\begin{definition}
    The \textbf{quiver of sections} of a full strongly exceptional collection $\mathcal{L}_1, \ldots, \mathcal{L}_n$ is the quiver with relations $(Q,I)$ defined as follows:

    \begin{enumerate}
        \item The vertices of the quiver are the line bundles $\mathcal{L}_i$.
        \item There is a directed edge from $\mathcal{L}_i$ to $\mathcal{L}_j$ for each indecomposable section in $H^0(X,\mathcal{L}_i^{-1} \otimes \mathcal{L}_j)$. We label the edge by the section $s$.
        \item For each path $p$ in this quiver, let $\operatorname{div}(p)$ be the sum of the divisors $\operatorname{div}(s_i)$ appearing in the labels of the path. Then, the relations are
         \[ I = \langle p - q \; \lvert \; \text{ $p,q:u \to v$ paths with $\operatorname{div}(p) = \operatorname{div}(q)$} \rangle \]
    \end{enumerate}
\end{definition}

\begin{proposition}\label{thm: tilting quiver as quiver of sections calculation}
    Let $X$ be a smooth variety over $\kk$ with a fully strong exceptional sequence $\mathcal{L}_1, \ldots, \mathcal{L}_n$ with quiver of sections $(Q,I)$. Then, the tilting algebra is isomorphic to the path algebra of the quiver of sections.
\end{proposition}

A \textbf{semi-orthogonal decomposition} of a triangulated category $\mathcal{C}$ is a sequence of strictly full triangulated subcategories $\mathcal{A_1},\ldots \mathcal{A}_k$ such that the smallest triangulated subcategory containing all $\mathcal{A}_i$ is $\mathcal{C}$ and for all $i < j$ we have no homomorphisms between any object in $\mathcal{A}_i$ and any object in $\mathcal{A}_j$. Every full strong exceptional collection $\mathcal{E}_1, \ldots, \mathcal{E}_k$ gives rise to a semi-orthogonal decomposition where $\mathcal{A}_i$ is the category generated by $\mathcal{E}_i$. If we have a semi-orthogonal decomposition, then $K(\mathcal{C}) = K(\mathcal{A}_1) \oplus \cdots \oplus K(\mathcal{A}_k)$. Therefore, if we have a full exceptional collection for $\D(X)$, then the sheaves in the collection form a basis for $K(X)$.

\section{Combinatorial Approach to Full Strongly Exceptional Sequences}\label{Sec: Combinatorial Approach}

\subsection{Homotopy Theory of Set Differences of Polytopes}

We begin by proving some results on the homotopy theory of differences of polytopes.

Let $P$ be a polytope defined by a collection of affine linear functionals $\phi_i$ for $i \in [\ell]$. That is $P = \{x \in \RR^n \; \lvert \; \phi_i(x) \geq 0 \text{ for all $i \in [\ell]$}\}$. Then, for any collection of $a_i \in \RR^n$ with $a_i > 0$ for $i \in [\ell]$, define
 \[P_{> \mathbf{a}} = \{x \in \RR^n \; \lvert \; \phi_i(x) > - a_i \text{ for all $i \in [\ell]$}\}.\]
The set $P_{> \mathbf{a}}$ is convex, open, and contains $P$. If we do not specify the choice of linear functionals, then we mean that the result holds for any choice. The following result is a minor modification of a result by \cite{Altmann2020}.

\begin{proposition}\cite{Altmann2020}\label{prop: P Q enlarging Q}
    Let $P$ and $Q$ be polytopes in $\RR^n$. Then, there exists a constant $c > 0$ such that for all $a_1, \ldots, a_{\ell}$ with $0 < a_i < c$ for all $i$, we have a deformation retract from $P \backslash Q$ to $P \backslash Q_{> \mathbf{a}}$.
\end{proposition}

We will use the following theorem of Smale:

\begin{theorem}[Smale's Theorem]\label{thm: Smale's Theorem}
    Let $X$ and $Y$ be connected, locally compact separable metric spaces.  Assume also that $X$ is locally contractible.  Consider a proper surjective continuous map $f: X \rightarrow Y$.  Assume that for all $y \in Y$, the space $f^{-1}(y)$ is contractible and locally contractible.  Then $f$ is a weak homotopy equivalence.
\end{theorem}

\begin{proposition}\label{prop: Minkowski sum deformation retract}
    Let $P_1,P_2$ and $Q$ be polytopes in $\RR^n$. Then, we have a homotopy equivalence between $P_1 \backslash P_2$ and $(P_1 + Q) \backslash (P_2 + Q)$. We also have the open version $P_1 \backslash P_2^{\circ} \sim (P_1 + Q) \backslash (P_2 + Q)^{\circ}$.
\end{proposition}

\begin{proof}
    Without loss of generality, we assume that $0 \in Q$. We prove the first equivalence and the second one follows by the same argument.
    
    Consider the morphism $m: P_1 \times Q \to P_1 + Q$ given by $m(x,y) = x + y$ and the inclusion $i: (P_2 + Q) \cap (P_1 + Q) \hookrightarrow P_1 + Q$. Let $F$ be the pullback of these maps. That is, we have the Cartesian square

\begin{equation*} 
    \begin{tikzcd} 
        F \ar[hookrightarrow]{r} \ar[d] &  P_1 \times Q \ar[d,"m"]  \\
        (P_2 + Q) \cap (P_1 + Q)  \ar[hookrightarrow]{r}{i}& P_1 + Q
    \end{tikzcd}
\end{equation*}
    Then, $F$ consists of all pairs $(x,y)$ with $x \in P_1$ and $y \in Q$ such that $x + y$ is in $P_2 + Q$. It is clear that the image of the restriction of $m$ to $P_1 + Q \backslash F$ is $P_1 + Q \backslash P_2 + Q$ and further the fiber at $z \in P_1 + Q \backslash P_2 + Q$ is $\{(x,y) \in P_1 \times Q\; \lvert \;  z = x + y\}$ which is convex and hence contractible and locally contractible. This map is a closed map with compact fibers so it is proper. Since each fiber is connected, this means that this map induces a bijection between connected components of the image and the domain. By restricting the function to each connected component and then applying Smale's theorem, we have that each connected component of $P_1 \times Q \backslash F$ is weakly homotopy equivalent to the corresponding connected component of $P_1 + Q \backslash P_2 + Q$.

    Let $\pi: P_1 \times Q \backslash F \to  \RR^n$ be the proper surjective map given by projection onto the first component. If $x \in P_1 \backslash P_2$, then $(x,0)$ is in $P_1 \times Q \backslash F$ hence the image contains $P_1 \backslash P_2$. Further, if $x \in P_2 \cap P_1$, then $x + y \in P_2 + Q$ for all $y \in Q$ by definition. So, this is a surjective proper map $\pi: P_1 \times Q \backslash F \to P_1 \backslash P_2$.

    The fiber of $x \in P_1 \backslash P_2$ is
     \[\pi^{-1}(x) = \{ (x,y) \in P_1 \times Q \; \lvert \; x + y \not \in P_2\} = \{(x,y) \in P_1 \times Q \; \lvert \; y \in P_2 - x \cap Q\}. \]
     The fiber is convex and hence contractible and locally contractible. Again, each fiber is connected and so by Smale's theorem, we have that each connected component $P_1 \times Q \backslash F$ is weakly homotopy equivalent to the corresponding component of $P_1 \backslash P_2$. Combining this with the previous equivalence we have a weak homotopy equivalence between components of $P_1 \backslash P_2$ and $(P_1 + Q) \backslash (P_2 + Q)$.

     By Proposition \ref{prop: P Q enlarging Q}, we have that $P \backslash P'$ is homotopy equivalent to $P \backslash P'_{> \mathbf{a}}$ which is a disjoint union of CW complexes for any two polytopes $P$ and $P'$. By Whitehead's theorem, a weak homotopy equivalence between spaces homotopy equivalent to CW complexes is automatically a homotopy equivalence. Therefore, we have a homotopy equivalence component-wise of $(P_1+Q) \backslash (P_2+Q)$ and $P_1 \backslash P_2$ and hence a homotopy equivalence on the entire spaces.
\end{proof}

We will use the following result in the proof of Theorem \ref{mainthm: dg homs calculation}.

\begin{lemma}\label{lem: intersection to set difference homotopy}
    Let $P$ and $Q$ be full-dimensional polytopes in $\RR^n$ with the same normal fan $\Sigma$. Let $P^{\circ}$ and $Q^{\circ}$ denote their relative interior. Then, there is a weak homotopy equivalence from $\partial P \cap Q^{\circ}$ to $Q \backslash P$, when the first is non-empty.
\end{lemma}

\begin{proof}
    First, we build a strong deformation retract from $Q^{\circ} \backslash P^{\circ}$ to $\partial P \cap Q^{\circ}$. First note that $P \cap Q^{\circ}$ is a convex and closed subset of $Q^{\circ}$. It is a standard fact that such a subset (of a metric space) is a Chebyshev set. This means that for each $x \in Q^{\circ}$ there is a unique point in $r(x) \in P \cap Q^{\circ}$ with minimal distance $d(x,r(x))$. Further, the function $r(x)$ is continuous. We can restrict this function to obtain a function $r: Q^{\circ} \backslash P^{\circ} \to \partial P \cap Q^{\circ}$ since the closest point $r(x)$ for $x \in Q^{\circ} \backslash P^{\circ}$ is necessarily on the boundary of $P$.

    This gives a homotopy $H(x,t) = (1-t)x + t(r(x))$ which is a strong deformation retract.

    Now, choose affine linear functionals $\phi_i$ indexed by the rays of $\Sigma$ such that
     \[P = \{x \in \RR^n \; \lvert \; \phi_i(x) \geq 0 \text{ for all $i$}\}\]
     Let $\varepsilon > 0$ and consider the enlargement of $Q$ given by $P + \varepsilon Q$. Since $P$ and $Q$ share their normal fan, we can choose $\mathbf{a}$ with $a_i > 0$ such that $P_{ > \mathbf{a}} = (P + \varepsilon Q)^{\circ}$.

     By Proposition \ref{prop: P Q enlarging Q}, we have a homotopy equivalence
      \[ Q \backslash P \sim Q \backslash P_{> \mathbf{a}} = Q \backslash (P + \varepsilon Q)^{\circ}.\]

      By Proposition \ref{prop: Minkowski sum deformation retract}, we can subtract $\varepsilon Q$ from both the terms of the set difference to get the equivalence
       \[ Q \backslash (P + \varepsilon Q)^{\circ} \sim ((1- \varepsilon) Q) \backslash P^{\circ}.\]

        Considering the sequence
         \[ \cdots \hookrightarrow \left(1 - \frac{1}{n} \right )Q \backslash P^{\circ} \hookrightarrow \left(1 - \frac{1}{n+1} \right)Q \backslash P^{\circ} \hookrightarrow \cdots ,\]
    we obtain a weak homotopy equivalence between $\left( 1- \varepsilon \right )Q\backslash P^{\circ}$ for small $\varepsilon$ and the direct limit of this sequence which is $Q^{\circ} \backslash P^{\circ}$.
        
    All these homotopy equivalences together complete the proof.
      
\end{proof}

\subsection{Calculation of Homomorphisms and Exceptionality of Polytopes}

We now turn to calculating the homomorphisms between nef line bundles in the dg categories $\operatorname{Perf}_T(X_{\Sigma})$ and $\operatorname{Perf}(X_{\Sigma})$ through the coherent-constructible correspondence. As a consequence, we will recover a result of \cite{Altmann2020} that calculates the extensions between nef line bundles. We will use $C(V)$ denote the singular cochain complex on $V \subseteq \RR^n$ over $\kk$ and let $\widetilde{C}(V)$ denote the augmented complex which has $\kk$ at degree $-1$. For $U \subseteq V$, let $C(V,U)$ denote the relative cochains of the pair $(V,U)$.

\begin{theorem}[Theorem \ref{mainthm: dg homs calculation}]\label{thm: calculation of homs in Perf}
     Let $X_\Sigma$ be a projective toric variety and let $P,Q \in \Def(\Sigma)$. Then, we have the quasi-isomorphisms
     \[ \hom_T(\mathcal{L}_P^{-1}, \mathcal{L}_Q^{-1}) \cong \hom_T(\mathcal{L}_Q, \mathcal{L}_P) \cong \widetilde{C}(Q \backslash P)[-1] \]
     in $\operatorname{Perf}_T(X_{\Sigma})$ and
      \[ \hom(\mathcal{L}_{P}^{-1}, \mathcal{L}_Q^{-1}) \cong \hom(\mathcal{L}_Q, \mathcal{L}_P) \cong  \bigoplus_{m \in M} \widetilde{C}(Q \backslash (P + m))[-1]\]
    in $\operatorname{Perf}(X_{\Sigma})$.
\end{theorem}

\begin{proof}
We begin with the case where $P$ and $Q$ are ample polytopes. That is, their normal fan is exactly $\Sigma$. Since $\Sigma$ is strictly convex, this means that $P$ and $Q$ are full-dimensional polytopes in $M_{\RR}$. It suffices to calculate the homomorphisms after applying the coherent-constructible correspondence. Since $P$ and $Q$ are ample, we have that $\kappa(\mathcal{L}_P^{-1})$ is the standard constructible sheaf $i_{P^{\circ}*}\kk$ where $i$ is the inclusion map $P^{\circ} \to \RR^n$ and likewise for $Q$ by Proposition \ref{prop: ccc standard sheaf quasiisomorphism}.

In \cite{Nadler2008}, Nadler and Zaslow calculated the set of homomorphisms between constructible sheaves on open sets $U_0$ and $U_1$. By a standard collection of identities and adjunctions, we have
    \begin{align*}
        \hom(i_{U_0 *}\kk, i_{U_1 *}\kk) &\cong \Gamma(\RR^n, \sheafhom(i_{U_0*} \kk, i_{U_1*} \kk)) \\
        &\cong \Gamma(\RR^n, i_{U_1*} \sheafhom(i_{U_1}^*i_{U_0*} \kk, i_{U_0} \kk)) \\
        &\cong \Gamma(\RR^n, i_{U_1 *} \sheafhom(\omega_{U_1}, i_1^{!}i_{0!} \omega_{U_0})) \\
        &\cong \Gamma(\RR^n, i_{U_1*}i_{U_1}^*i_{U0!} \kk),
    \end{align*}
    where $\omega_{U_0}$ and $\omega_{U_1}$ are the dualizing complexes. Therefore, 
     \[\hom(i_{U_0 *}\kk, i_{U_1 *}\kk) \cong \Gamma(U_1, i_{U_0!}\kk) \cong C(\overline{U_0} \cap U_1, \partial U_0 \cap U_1).\]
    Applying this to our situation, we have
      \[ \hom(i_{P^{\circ}*}\kk, i_{Q^{\circ}*} \kk) \cong C(P \cap Q^{\circ}, \partial P \cap Q^{\circ})\]
    in $\operatorname{Sh}_{cc}(\RR^n, \Lambda_{\Sigma})$. If $P \cap Q^{\circ}$ is empty, then this complex is the zero complex and $\widetilde{C}(P/Q)[-1]$ is acyclic so they are quasi-isomorphic. Suppose otherwise and consider the usual exact sequence of complexes
     \[0 \rightarrow C(P \cap Q^{\circ}, \partial P \cap Q^{\circ}) \to \widetilde{C}(P \cap Q^{\circ}) \rightarrow \widetilde{C}(\partial P \cap Q^{\circ}) \rightarrow 0. \]
     Since the intersection of $P$ and $Q^{\circ}$ is non-empty and both are convex sets, we have that $P \cap Q^{\circ}$ is contractible, and hence the middle complex is acyclic. Since this exact sequence is split, we get that the connecting homomorphism of the corresponding long exact sequence is given by a geniune map of chain complexes $\delta: \widetilde{C}(\partial P \cap Q^{\circ}) \to C(P \cap Q^{\circ}, \partial P \cap Q^{\circ})[1]$. Since the middle term has no cohomology, this map is a quasi-isomorphism. Thus,
      \[ \hom(\kappa(\mathcal{L}^{-1}_P), \kappa(\mathcal{L}^{-1}_Q)) \cong C(P \cap Q^{\circ}, \partial P \cap Q^{\circ}) \cong \widetilde{C}(\partial P \cap Q^{\circ})[-1].\]
     
    By Lemma \ref{lem: intersection to set difference homotopy}, $\partial P \cap Q^{\circ}$ is weakly homotopy equivalent to $Q \backslash P$ which completes the calculation for ample $T$-equivariant nef line bundles since weak homotopy equivalences induce isomorphisms of singular cohomology.

     For the ample and non-equivariant case, let $\operatorname{Res}: \operatorname{Sh}_T(X) \to \operatorname{Sh}(X)$ be the restriction functor and $\operatorname{Ind}$ be its right adjoint as in \cite{bernstein2006}. We have that $\operatorname{Ind} \circ \operatorname{Res}(\mathcal{L}_P) \cong \bigoplus_{m \in M} \mathcal{L}_{P + m}$ with the quasi-isomorphism being given by the maps $\operatorname{Ind} \circ \operatorname{Res} (\mathcal{L}_P) \to \mathcal{L}_{P+m}$ induced by the isomorphisms $\operatorname{Res}(\mathcal{L}_P) \to \operatorname{Res}(\mathcal{L}_{P+m})$ through the adjunction. This amounts to considering every possible equivariant structure on the nef line bundle $\operatorname{Res}(\mathcal{L}_{P})$. Then by the adjunction between induction and restriction, we have
    \begin{align*}
    \hom(\operatorname{Res}(\mathcal{L}_P^{-1}), \operatorname{Res}(\mathcal{L}_Q^{-1})) 
     &=\cong\hom_T(\mathcal{L}_{P}^{-1}, \operatorname{Ind} \circ \operatorname{Res} (\mathcal{L}_{Q}^{-1})) \\
        &\cong \bigoplus_{m \in M} C((Q+m) \cap P^{\circ}, \partial (Q+m) \cap P^{\circ}) \\
    &\cong  \bigoplus_{m \in M} \widetilde{C}(Q \backslash (P + m))[-1]. 
     \end{align*}
    Note that only finitely many of the terms on the right are not acyclic and hence nonzero in $\operatorname{Sh}(\RR^n)$.

    Finally, we consider the general case where $P, Q \in \Def(\Sigma)$ might not be ample. Then, $P$ and $Q$ are not necessarily full-dimensional which precludes us from using Nadler and Zaslow's calculation. We fix this by perturbing $P$ and $Q$ to ample polytopes.

    Let $P'$ be an ample polytope in $\Def(\Sigma)$. We have the standard facts that the map 
     \[ \mathcal{E} \mapsto \mathcal{E} \otimes \mathcal{L}_{P'} \]
     gives an autoequivalence of the category $\operatorname{Perf}_T(X_{\Sigma})$ since $\mathcal{L}_Q$ is invertible and that for any two polytopes $P$ and $Q$ we have $\mathcal{L}_P \otimes \mathcal{L}_Q = \mathcal{L}_{P + Q}$ where $P + Q$ is the Minkowski sum of the polytopes. This gives us
      \[ \hom_T(\mathcal{L}_{P}^{-1}, \mathcal{L}_{Q}^{-1}) \cong \hom_T(\mathcal{L}_{P + P'}^{-1}, \mathcal{L}_{Q + P'}^{-1})). \]
    We are now in the ample case since Minkowski summing with an ample polytope gives an ample polytope. By Proposition \ref{prop: Minkowski sum deformation retract}, $Q \backslash P$ is homotopy equivalent to $(Q + P') \backslash (P + P')$ which completes the proof.
\end{proof}

\begin{remark}
    The same proof calculates the homomorphisms between $\mathcal{O}(D)$ and $\mathcal{O}(D')$ for any nef $\RR$-divisors $D$ and $D'$ in terms of the corresponding non-lattice polytopes whose normal fan coarsen $\Sigma$. We will only focus on the line bundle case so we leave this to the reader.
\end{remark}

With this, we are led to the following combinatorial category:

\begin{definition}
    The \textbf{equivariant polytope category} $\Poly_T(\Sigma)$ of $\Sigma$ is the dg category whose objects are polytopes in $\Def(\Sigma)$ and
     \[\hom_T(P,Q) := \hom_T(\mathcal{L}_{P}, \mathcal{L}_{Q}) \cong \widetilde{C}(P \backslash Q)[-1].  \]
    The \textbf{non-equivariant polytope category} $\Poly(\Sigma)$ of $\Sigma$ is the dg category whose objects are polytopes in $\Def(\Sigma)$ up to translation and
     \[ \hom(P,Q) :=\hom(\mathcal{L}_{P}, \mathcal{L}_Q) \cong  \bigoplus_{m \in M} \widetilde{C}(P \backslash (Q + m))[-1]. \]
    Both of these categories are monoidal with respect to $P \otimes Q := P+Q$ where $+$ is the Minkowski sum of polytopes.
\end{definition}

When $\Sigma$ is projective and smooth, the category $\operatorname{Perf}_T(X_{\Sigma})$ is generated by $T$-equivariant ample line bundles. The previous calculation shows that the functor $\Poly_T(\Sigma) \to \operatorname{Perf}_T(X_{\Sigma})$ given by $P \mapsto \mathcal{L}_P^{-1}$ gives rise to an.equivalence $\operatorname{Tr}(\Poly_T(\Sigma)) \cong \operatorname{Perf}_T(X_{\Sigma})$ where $\operatorname{Tr}(\mathcal{C})$ is a triangulated envelope of the dg category $\mathcal{C}$ (through say twisted complexes). 

By taking cohomology we obtain another proof of the following theorem of the result of \cite{Altmann2020}.

\begin{corollary}\label{cor: Calculation of Exts of nef line bundles}
    Let $P$ and $Q$ be polytopes in $\Def(\Sigma)$ for a projective fan $\Sigma$. Then,
      \[\operatorname{Ext}_T^p(\mathcal{L}_{P},\mathcal{L}_Q) \cong \tilde{H}^{p-1}(P\backslash Q) \]
    and
     \[\operatorname{Ext}^p(\mathcal{L}_{P},\mathcal{L}_Q) \cong \bigoplus_{m \in M}\tilde{H}^{p-1}(P\backslash Q + m),\]
     for all $p \geq 0$.
\end{corollary}

If we are only interested in finding a basis for $K_0(X_{\Sigma})$ for smooth projective varieties, then it suffices to study exceptionality of polytopes.

\begin{proposition}\label{prop: exceptionality gives basis}
    Let $\Sigma$ be a smooth projective fan and let $(P_1,\ldots,P_{\ell})$ be an exceptional collection of polytopes in $\Poly(\Sigma)$ where $\ell$ is the number of maximal cones of $\Sigma$. Then, the classes $\{[\mathcal{L}_{P_i}]\}$ form a basis for $K_0(X_{\Sigma}) \cong \McM(\Sigma)$.
\end{proposition}

\begin{proof}
    Since $(P_1,\ldots, P_{\ell})$ is an exceptional collection of polytopes, we have an exceptional collection of nef line bundles $(\mathcal{L}_{P_{1}}, \mathcal{L}_{P_{2}}, \ldots, \mathcal{L}_{P_{\ell}})$ in $\operatorname{Perf}(X_{\Sigma})$. This gives a semiorthogonal decomposition $\operatorname{Perf}(X_{\Sigma}) = \langle \mathcal{A}, E_1,\ldots, E_{\ell}\rangle$ where $E_i$ is the category generated by $\mathcal{L}_{P_i}$ and $\mathcal{A}$ is the category orthogonal to $\langle E_1, \ldots, E_{\ell} \rangle$.

    At the level of $K$-theory, this implies
     \[K(\operatorname{Perf}(X_{\Sigma})) = K(\mathcal{A}) \oplus \bigoplus_{i=1}^{\ell} K(E_{i}). \]
    Since $K(E_i)$ is the subgroup of $K(\operatorname{Perf}(X_{\Sigma}))$ generated by the class of $\mathcal{L}_{P_i}$, the classes $\{[\mathcal{L}_{P_i}]\}$ are then linearly independent. 

    Since $X_{\Sigma}$ is smooth, we know that the rank of $K(\operatorname{Perf}(X_{\Sigma}))$ is equal to the Euler characteristic of the variety. From standard results (see \cite{Fulton1993}), the Euler characteristic of a toric variety is the number of maximal cones. Therefore, $\ell$ is the rank of $K(\operatorname{Perf}(X_{\Sigma}))$ and so the elements $[\mathcal{L}_{P_i}]$ are a basis of $K^0(X_{\Sigma})$. The second isomorphism follows from Morelli's Theorem.
    \end{proof}

\begin{remark}
    In the previous proof, we saw that every exceptional collection of size equal to the number of maximal cones gives a decomposition of the category into $\langle \mathcal{A}, E_1,\ldots,E_{\ell} \rangle$. It follows from the arguments that $K(\mathcal{A}) = 0$. This collection is full if and only if $\mathcal{A} = 0$. For general schemes $X$, there exist (admissible) subcategories of $\operatorname{Perf}(X)$ that are non-zero but whose K-group is $0$. These are known as phantom categories. If we could show that our varieties do not have phantom categories, then these exceptional collections would automatically be full.
\end{remark}

\subsection{Inclusion-Exclusion Complexes}\label{Sec: Inclusion-Exclusion Complexes}

In order to prove fullness of collections in a derived category, it is helpful to have access to many exact sequences. Since morphisms between the nef line bundles $\mathcal{L}_P$ and $\mathcal{L}_Q$ correspond to inclusions of polytopes, it is natural to build complexes out of inclusion-exclusion of polytopes. The only worry is the typical problem of assigning signs to these inclusions to obtain chain complexes. We will handle this sign bookkeeping by employing Bj\"orner's theory of CW posets from \cite{Bjo1984}. This theory will also help us study exactness of these complexes.

For a poset $\mathbb{P}$, let $\operatorname{Cov}(\mathbb{P})$ denote the set of covering relations $x \lessdot y$. We say that $\mathbb{P}$ is \textbf{thin} if every interval of length $2$ has four elements, i.e. is a diamond. An \textbf{order-preserving} map $f: \mathbb{P} \to \mathbb{P'}$ is a map on the elements of the posets such that $x \leq y$ implies $f(x) \leq f(y)$ and an \textbf{order-reversing map} is one where $x \leq y$ implies $f(y) \leq f(x)$. A \textbf{lower ideal} is a subset $I \subseteq \mathbb{P}$ where if $y \in I$ and $x \leq y$, then $x \in I$.  An \textbf{upper ideal} is a subset $I \subseteq \mathbb{P}$ where if $x \in I$ and $x \leq y$, then $y \in I$. The \textbf{order complex} of $\mathbb{P}$ is the simplicial complex $\Delta(\mathbb{P})$ whose simplicies are chains in $\mathbb{P}$ and inclusion of simplicies is given by refinement of chains. Given a poset $\mathbb{P}$, we define the \textbf{dual} or \textbf{opposite} poset $\mathbb{P}^{op}$ to be the poset on the same with the reversed order. In other words, $x \leq y$ in $\mathbb{P}^{op}$ if and only if $y \leq x$ in $\mathbb{P}$. We refer the reader to \cite{Stanley2011} for more information on posets.

We will consider $\Def(\Sigma)$ as a poset under inclusion of polytopes. 
\begin{definition}
    A \textbf{CW poset} is a poset $\mathbb{P}$ such that
    \begin{itemize}
        \item $\mathbb{P}$ has a bottom element $\hat{0}$.
        \item $| \mathbb{P} | > 1$.
        \item For any $x \in \mathbb{P} - \{\hat{0}\}$, the realization of the order complex of the open interval $(\hat{0},x)$ is homoemorphic to a sphere.
    \end{itemize}
\end{definition}

We will only consider finite CW posets for this paper.

\begin{lemma}
    Let $\mathbb{P}$ be a CW poset.

    \begin{enumerate}
        \item $\mathbb{P}$ is the face poset of a regular CW complex and every face poset of a regular CW complex is a CW poset.
        \item $\mathbb{P}$ is a ranked poset. We use $\rk(x)$ to denote the rank function on $\mathbb{P}$.
        \item $\mathbb{P}$ is thin.
        \item Any lower ideal $I$ of $\mathbb{P}$ of cardinality greater than $1$ is again a CW poset.
    \end{enumerate}
\end{lemma}

\begin{example}\label{ex: Examples of CW posets}
Here are some interesting and important examples of CW posets:
\begin{itemize}
    \item The \textbf{boolean} poset $2^I$ for a finite set which is the poset consisting of all subsets of $I$ including the empty set.
    \item For any polytope $P$, the \textbf{face lattice} $\mathcal{L}_P$ is the poset of faces of $P$ including the empty set. The opposite poset $\mathcal{L}_{P}^{opp}$ is also a CW poset since it is the face of the dual polytope of $P$. This opposite poset is isomorphic to the face lattice $\mathcal{F}(\Sigma)$ of the normal fan $\Sigma$ whose elements are cones of $\Sigma$ under inclusion as well as a top element $\mathcal{1}$ that does not correspond to a cone.
    \item For any Coxeter group $W$, the weak and strong Bruhat orders on $W$ are CW posets.
\end{itemize}
\end{example}

By the previous lemma, there is a regular CW complex $\Gamma$ whose face poset is $\mathbb{P}$. Define a function $c: Cov(\mathbb{P}) \to \{-1,+1\}$ where $c(x \lessdot y)$ is the incidence number of the cells $\sigma_x$ and $\sigma_y$ in $\Gamma$. Set $c(\hat{0} \lessdot x) = +1$ for all $x$ covering $\hat{0}$.

\begin{definition}
Given a CW poset $\mathbb{P}$ and on order-preserving map $F: \mathbb{P} \to \operatorname{Def}(\Sigma)$, the \textbf{inclusion-exclusion complex} $\IE^{\bullet} = \IE_{F}^{\bullet}$ is the complex of $T$-equivariant nef line bundles given by
 \[\IE_{F}^{k} = \bigoplus_{\substack{x \in \mathbb{P} \\ \operatorname{rk}(x) = k}} \mathcal{L}_{F(x)} \quad \quad \]
and differential $d$ given by
 \[ d_k = \sum_{\substack{x \in \mathbb{P} \\ \rk(x) = k}} \sum_{x \lessdot y} c(x \lessdot y) 1_{F(x), F(y)}, \]
where $1_{P,Q}: \mathcal{L}_P \to \mathcal{L}_Q$ is the map induced by the inclusion $P \subset Q$.
\end{definition}

\begin{remark}
    When $\mathcal{P}$ is a poset whose dual is a CW poset and $f$ is an order-reversing map, we can associate an inclusion-exclusion complex to the opposite map $f: \mathcal{P}^{op} \to \Def(\Sigma)$.
\end{remark}

This is a complex since the signs are transferred from the signs of the (reduced) cellular complex of the regular CW complex $\Gamma$. We will typically omit $c$ from the notation.

We now record two simple operations on inclusion-exclusion complexes which we will use in the future. Given a poset $\mathbb{P}$ and order-preserving map $\phi: \mathbb{P} \to \Def(\Sigma)$, we say that $Q$ \textbf{intersects stably} with $\phi$ if $Q \cap \phi(x) \in \Def(\Sigma)$ for every $x \in \mathbb{P}$.

\begin{definition}
    Let $\IE_{\phi}^{\bullet}$ be an inclusion-exclusion complex and $Q$ be a polytope that intersects stably with $\phi$. The \textbf{truncation} of $\IE^{\bullet}$ by $Q$ is the complex given by the map 
     \[ \overline{\phi}(x) = \phi(x) \cap Q.\]
\end{definition}

We will see soon that truncating an exact inclusion-exclusion complex gives another exact complex. We also have the following simple operation.

\begin{proposition}\label{prop: tensor product inclusion-exclusion}
    Let $\phi: \mathbb{P} \to \operatorname{Def}(\Sigma)$ be an exact order-preserving map.

    \begin{enumerate}
        \item If $P \in \Def(\Sigma)$, then the inclusion-exclusion complex corresponding to
         \[ \phi'(x) = \phi(x) + P \]
        is exact.
        \item If $P \in \Def(\Sigma)$, and every $\phi(x)$ for $x \in \mathbb{P}$ has $P$ as a Minkowski summand, then the inclusion-exclusion complex corresponding to
         \[ \phi'(x) = \phi(x) - P\]
         is exact.
    \end{enumerate}
\end{proposition}

\begin{proof}
    The new complex is obtained by tensoring every term of the previous complex by either $\mathcal{L}_{P}$ or $\mathcal{L}_P^{-1}$. Since tensoring by vector bundles is exact, this preserves exact sequences.
\end{proof}

\subsection{Exactness Criterion}

Inclusion-exclusion complexes are not in general exact, but there is an easy criterion for exactness by considering the stalks on the constructible side.

\begin{definition}
    The $T$-stalk functor at $m \in M_{\RR}$ is the functor $\mathcal{L}_P \to (\mathcal{L}_P)_{(m)}$ defined on $T$-equivariant nef line bundles by
     \[(\mathcal{L}_P)_{(m)} = \begin{cases}
         \kk & \text{if $m \in P$} \\
         0 & \text{otherwise}
     \end{cases} \]
     such that the map in $\Hom_T(\mathcal{L}_P, \mathcal{L}_Q)$ induced by the inclusion map is sent to the identity map $(\mathcal{L}_P)_{(m)} \to (\mathcal{L}_Q)_{(m)}$ with $m$ is in both $P$ and $Q$ and is the zero map otherwise.
\end{definition}

\begin{lemma}\cite{Bjo1984}\label{lem: Bjoner Brenti}
    Let $\mathbb{P}$ be a CW poset. Let $I$ be a finite interval $[a,b]$ in $\mathbb{P}$ of cardinality greater than $1$. Then the complex
     \[ 0 \rightarrow C^{0} \rightarrow C^{1} \rightarrow \cdots \rightarrow C^{r} \rightarrow 0 \]
    where $C_k$ is the free vector space generated by elements in the poset of the interval $[a,b]$ of rank $\rk(a) + k$ and morphisms
     \[ d_k(x) = \sum_{x\lessdot y \leq b} c(x \lessdot y) y \]
    is exact.
\end{lemma}

\begin{proof}
    This was proven by Bj\"orner and Brenti for intervals of the weak order of a Coxeter group in \cite{Bjo1984} which forms a CW poset; however, their proof holds for all CW posets.
\end{proof}

\begin{proposition}\label{prop: Augmented Brianchon-Gram complex}
    For any $m \in M_{\RR}$ and any $P$, the following complex
     \[ \overline{\operatorname{BG}}(P)_m = 0  \rightarrow \operatorname{BG}(P)^{0}_m \rightarrow \cdots \rightarrow \operatorname{BG}(P)^n_m \rightarrow (\mathcal{L}_P)_{(m)} \rightarrow 0\]
     is acyclic, where the last map is the zero map if $m \not \in P$ and otherwise it is the direct sum of the identity maps when $m \in P$.
\end{proposition}

\begin{proof}
    In the proof of Theorem 3.7 in \cite{Fang2011}, it is shown that the stalk of $\operatorname{BG}(P)$ at $m \in M$ is quasi-isomorphic to the zero complex if $m \not \in M$. Thus, when $m \not \in P$, the complex $\overline{\operatorname{BG}}(P)_m$ is just the stalk of $\operatorname{BG}(P)$ at $m$ and so it is acyclic.

    If $m \in P$, then every tangent cone $(\sigma^{\vee} + v_{P,\sigma})$ contains $m$ and so this complex is the complex (or rather its dual) from Lemma \ref{lem: Bjoner Brenti} corresponding to the entire poset $\mathcal{F}(\Sigma)$, which is an interval, and so it is exact.
\end{proof}
The following proposition shows that as with the usual stalk functor, it suffices to check exactness of complexes of nef line bundles at $T$-stalks to verify exactness of the whole complex.

\begin{proposition}\label{prop:checking exactness at T-stalks}
    Let $C^{\bullet}$ be a complex of the form 
    \[ 0\rightarrow \bigoplus_{i \in I_0} \mathcal{L}_{P_i} \rightarrow \bigoplus_{i \in I_1} \mathcal{L}_{P_i} \rightarrow \cdots \rightarrow \bigoplus_{i \in I_k} \mathcal{L}_{P_i} \rightarrow 0 .\]
    Then, $C^{\bullet}$ is exact if and only if the $T$-stalk complex $(C^{\bullet})_{(m)}$ is exact for all $m \in M_{\RR}$.
\end{proposition}

\begin{proof}
    We have that $C^{\bullet}$ is the same as the total complex of the complex of complexes
     \[0 \rightarrow C^{0} \rightarrow C^{1} \rightarrow C^{2} \rightarrow \cdots\]
     where we view each $C^{i}$ as a complex concentrated at degree $0$. Applying the coherent-constructible correspondence map $\kappa$, we have that $\kappa(C^{\bullet})$ is quasi-isomorphic to the total complex of the double complex $T^{\bullet,\bullet}$ of sheaves on $\RR^n$ given by
      \[ 0 \rightarrow \bigoplus_{i \in I_0} \operatorname{BG}(P_i)^{\bullet} \rightarrow \bigoplus_{i \in I_1} \operatorname{BG}(P_i)^{\bullet} \rightarrow \cdots .\]

    Since $\kappa$ maps acyclic complexes to acyclic complexes, we only need to check that the total complex of the previous complex is acyclic which we do by showing exactness at all stalks for $m \in M_{\RR}$. Let $T^{\bullet,\bullet}_m$ denote this stalk double complex.

     Extend the double complex $T^{\bullet,\bullet}_m$ by the terms $(\mathcal{L}_{P_i})_{(m)}$ to get the double complex whose columns are direct sums of \[ 0 \rightarrow \operatorname{BG}^{0}(P_i)_m \rightarrow \cdots \rightarrow \operatorname{BG}(P_i)_m^n \rightarrow (\mathcal{L}_{P_i})_{(m)} \rightarrow 0,\]
     as in Lemma \ref{lem: Bjoner Brenti}. This double complex is bounded and has exact columns and hence the total complex is acyclic by the acyclic assembly lemma. This implies that $(C^{\bullet})_{(m)}$ is quasi-isomorphic to the total complex of $T^{\bullet,\bullet}_m$.

    Therefore, exactness of $(C^{\bullet})_{(m)}$ for all $m \in M_{\RR}$ implies that the total complex of $T^{\bullet,\bullet}_m$ is acyclic for all $m \in M_{\RR}$ which further implies that $C^{\bullet}$ is acyclic.
\end{proof}

We now turn to showing that exactness of an inclusion-exclusion complex follows from an easily verifiable property of the CW family. We refer to this as the exactness criterion.

\begin{theorem}[Theorem \ref{mainthm: Exactness Criterion}]\label{prop: Exact inclusion-exclusion complexes}
    Let $F: \mathbb{P} \to \Def(\Sigma)$ be an order-preserving map with inclusion-exclusion complex $\IE^{\bullet}$. If for all $m \in M_{\RR}$, the upper ideal $I_m = \{x \in \mathbb{P} \; \lvert \; m \in F(x)\}$ is either empty or an interval of cardinality greater than $1$, then $\IE^{\bullet}$ is exact.
\end{theorem}

\begin{proof}
    By the previous proposition, it suffices to check at the level of $T$-stalks. For any $m \in M_\RR$ that is not contained in any of the polytopes $\mathbb{P}(x)$, we have that the $T$-stalk complex $\IE^{\bullet}_{(m)}$ is the zero complex and so it is exact. 
    
    If $m$ is contained in one of the polytopes $F(x)$ then it is contained in all $F(y)$ for $x \leq y$, thus the stalk complex $\IE^{\bullet}_{(m)}$ is the complex described in the Lemma \ref{lem: Bjoner Brenti} for the poset $I_m$. This complex is exact since it is an interval of cardinality greater than $1$.
\end{proof}

\subsection{Subdivision and Truncated Brianchon-Gram Complexes}\label{Sec: Subdivision and Brianchon-Gram Complexes}

For our main calculations, it will suffice to study two special classes of inclusion-exclusion complexes.

\begin{definition}
    Let $P \in \Def(\Sigma)$ and $k \geq 1$. We say that a collection of polytopes $\mathcal{P} = \{P_1, \ldots, P_k \}$ in $\Def(\Sigma)$ is a \textbf{weak subdivision} of $P$ if
     \begin{enumerate}
         \item $\bigcup_{i=1}^{k} P_i = P$
         \item $P_i \cap P_j \in \Def^+(\Sigma)$ for all $i,j$.
     \end{enumerate}
    It is a \textbf{strict subdivision} if the intersection $P_i \cap P_j$ is a face of both $P_i$ and $P_j$. 
\end{definition}

\begin{definition}
    Given a subdivision (weak or strict) $\mathcal{P} = \{P_1, \ldots, P_k\}$ of $P \in \Def(\Sigma)$ with $k \geq 2$, consider the order-preserving map $F: 2^{[k]} \to \Def(\Sigma)$ defined by
     \[ F(S) = P_S = \bigcap_{i \not \in S} P_i,\]
    where $F([k]) = P_{\emptyset} = P$. The \textbf{subdivision Koszul complex} is the corresponding inclusion-exclusion complex.
\end{definition}
It is clear that the exactness criterion applies and hence every subdivision Koszul complex is exact. This complex first appeared in the work of \cite{Altmann2023}. The name comes from the fact that the $T$-stalk complexes are Koszul complexes.

\begin{example} Consider the del Pezzo surface $dP_5 = \operatorname{Bl}_2(\mathbb{CP}^2)$. It is the smooth projective toric variety of the polytope $P$ below

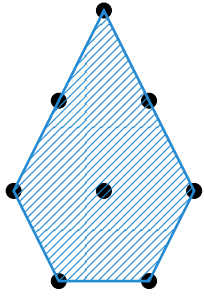
\begin{center}
    \begin{polyhedron}{}
     \vertex{point={(-1,2)},color=black}
    \vertex{point={(0,0)},color=black}
    \vertex{point={(2,0)},color=black}
    \vertex{point={(1,2)},color=black}
    \vertex{point={(3,2)},color=black}
    \vertex{point={(0,4)},color=black}
    \vertex{point={(2,4)},color=black}
    \vertex{point={(1,6)},color=black}
        \polygon{points = {(0,0),(2,0),(3,2),(1,6),(-1,2)},status = open}
    \end{polyhedron}
\end{center}
Consider the weak subdivision of $P$ given by the three polytopes 
    
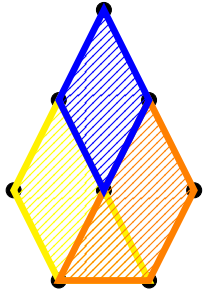
\begin{center}
\begin{polyhedron}{}
    \vertex{point={(-1,2)},color=black}
    \vertex{point={(0,0)},color=black}
    \vertex{point={(2,0)},color=black}
    \vertex{point={(1,2)},color=black}
    \vertex{point={(3,2)},color=black}
    \vertex{point={(0,4)},color=black}
    \vertex{point={(2,4)},color=black}
    \vertex{point={(1,6)},color=black}
       \polygon{points={(0,0),(2,0),(0,4),(-1,2)},color=yellow,thickness=2.5,status=open}
       \polygon{points={(0,0),(2,0),(3,2),(2,4)},thickness=2.5,color=orange,status=open}
       \polygon{points={(1,2),(2,4),(1,6),(0,4)},thickness=2.5,color=blue,status=open}
\end{polyhedron}
\end{center}
The corresponding subdivision Koszul complex is
 \[ 0 \rightarrow \mathcal{L}_{P_0} \rightarrow \mathcal{L}_{P_1} \oplus \mathcal{L}_{P_2} \oplus \mathcal{L}_{P_3} \rightarrow \mathcal{L}_{P_4} \oplus \mathcal{L}_{P_5} \oplus \mathcal{L}_{P_6} \rightarrow \mathcal{L}_P \rightarrow 0\]
 where

 \[ \begin{polyhedron}{} \vertex{point={0,-0.5},opacity=0,text=$P_0$} \vertex{point = {0,0}} \end{polyhedron} 
 \quad \quad 
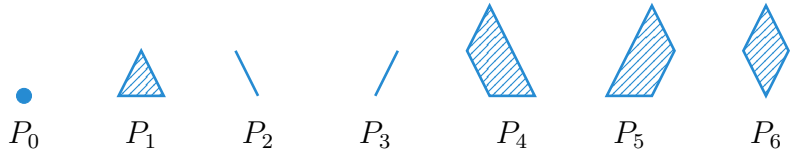
 \begin{polyhedron}{} \vertex{point={0,-0.5},opacity=0,text=$P_1$}
 \polygon{points = {(-0.5,0),(0.5,0),(0,1)},status=open} \end{polyhedron} \quad \quad 
 \begin{polyhedron}{} \vertex{point={0,-0.5},opacity=0,text=$P_2$} \polygon{points = {(0,0),(-0.5,1)},status=open} \end{polyhedron}
 \quad \quad
 \begin{polyhedron}{} \vertex{point={0,-0.5},opacity=0,text=$P_3$} \polygon{points = {(0,0),(0.5,1)},status=open} \end{polyhedron}
  \quad \quad 
 \begin{polyhedron}{} \vertex{point={0,-0.5},opacity=0,text=$P_4$} \polygon{points = {(-0.5,0),(0.5,0),(-0.5,2),(-1,1)},status=open} \end{polyhedron}\quad \quad 
 \begin{polyhedron}{} \vertex{point={0,-0.5},opacity=0,text=$P_5$} \polygon{points = {(-0.5,0),(0.5,0),(1,1),(0.5,2)}, status=open} \end{polyhedron}\quad \quad 
 \begin{polyhedron}{} \vertex{point={0,-0.5},opacity=0,text=$P_6$} \polygon{points = {(0,0),(0.5,1),(0,2),(-0.5,1)}, status = open} \end{polyhedron}\] 
\end{example}

The second class of complexes will come from truncating the Brianchon-Gram complex.

\begin{definition}
    Let $\Sigma$ be a projective fan and $P$ be a polytope in $\Def(\Sigma)$ with Brianchon-Gram complex $\operatorname{BG}(P)$. Let $Q$ be a polytope that contains $P$ such that for every tangent cone $\tau = (\sigma^{\vee} + v_{P,\sigma})$ of $P$ we have that $Q \cap \tau$ is in $\Def(\Sigma)$. Then, we say that $Q$ \textbf{truncates} $\operatorname{BG}(P)$.

    The \textbf{truncated Brianchon-Gram complex} $\operatorname{TBG}_Q^{\bullet}(P)$ is the inclusion-exclusion complex given by the order-preserving map $\phi: \mathcal{F}(\Sigma)^{op} \to \Def(\Sigma)$ defined by
     \[\phi(\sigma) = Q \cap (\sigma^{\vee} + v_{P,\sigma}) \]
    and $\phi(\hat{1}) = P$.
\end{definition}

This complex is exact since the $T$-stalks for $m \in Q$ are just the complexes from Proposition \ref{lem: Bjoner Brenti} and are the zero complex otherwise.

\begin{example}
    Consider the toric variety $\mathbb{CP}^1 \times \mathbb{CP}^1$. It's fan is determined by the coordinate axes in $\RR^2$. We will truncate the Brianchon-Gram complex of the line segment $P$ by the square $Q$ below

    \begin{figure}[H]
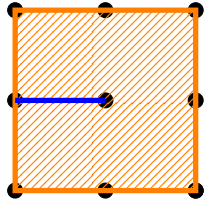

        \centering
        \begin{polyhedron}{}
            \vertex{point = {0,0},color=black}
            \vertex{point = {2,0},color=black}
            \vertex{point = {4,0},color=black}
            \vertex{point = {2,2},color=black}
            \vertex{point = {2,4},color=black}
            \vertex{point = {4,0},color=black}
            \vertex{point = {4,2},color=black}
            \vertex{point = {4,4},color=black}
            \vertex{point = {0,2},color=black}
            \vertex{point = {0,4},color=black}
            \polygon{points = {(0,0),(4,0),(4,4),(0,4)},color = orange, status = open,thickness=2}
            \polygon{points = {(0,2),(2,2)}, color = blue, status = open,thickness=2}
        \end{polyhedron}
        \caption{The line segment $P$ contained in the square $Q$.}
        \label{fig:P in Q}
    \end{figure}
The cones that appear in the Brianchon-Gram complex are the four orthants, the four half-spaces cut out by the axes, and the entire space $\RR^2$. Intersecting with the four orthants, we get the polytopes
   
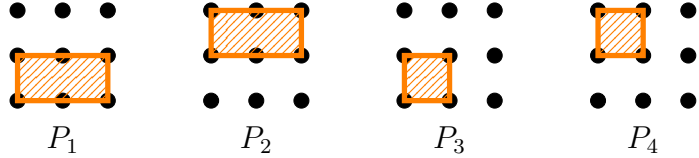
\begin{figure}[H]
        \centering
        \begin{polyhedron}{}
        \vertex{point={1,-0.5},opacity=0,text=$P_1$}
            \vertex{point = {0,0},color=black}
            \vertex{point = {1,0},color=black}
            \vertex{point = {2,0},color=black}
            \vertex{point = {0,1},color=black}
            \vertex{point = {1,1},color=black}
            \vertex{point = {2,1},color=black}
            \vertex{point = {0,2},color=black}
            \vertex{point = {1,2},color=black}
            \vertex{point = {2,2},color=black}
            \polygon{points = {(0,0),(2,0),(2,1),(0,1)},color = orange, status = open,thickness=2}
        \end{polyhedron}
        \quad
            \begin{polyhedron}{}\vertex{point={1,-0.5},opacity=0,text=$P_2$}
            \vertex{point = {0,0},color=black}
            \vertex{point = {1,0},color=black}
            \vertex{point = {2,0},color=black}
            \vertex{point = {0,1},color=black}
            \vertex{point = {1,1},color=black}
            \vertex{point = {2,1},color=black}
            \vertex{point = {0,2},color=black}
            \vertex{point = {1,2},color=black}
            \vertex{point = {2,2},color=black}
            \polygon{points = {(0,1),(2,1),(2,2),(0,2)},color = orange, status = open,thickness=2}
        \end{polyhedron}
        \quad
               \begin{polyhedron}{}\vertex{point={1,-0.5},opacity=0,text=$P_3$}
            \vertex{point = {0,0},color=black}
            \vertex{point = {1,0},color=black}
            \vertex{point = {2,0},color=black}
            \vertex{point = {0,1},color=black}
            \vertex{point = {1,1},color=black}
            \vertex{point = {2,1},color=black}
            \vertex{point = {0,2},color=black}
            \vertex{point = {1,2},color=black}
            \vertex{point = {2,2},color=black}
            \polygon{points = {(0,0),(1,0),(1,1),(0,1)},color = orange, status = open,thickness=2}
        \end{polyhedron}
        \quad
               \begin{polyhedron}{}\vertex{point={1,-0.5},opacity=0,text=$P_4$}
            \vertex{point = {0,0},color=black}
            \vertex{point = {1,0},color=black}
            \vertex{point = {2,0},color=black}
            \vertex{point = {0,1},color=black}
            \vertex{point = {1,1},color=black}
            \vertex{point = {2,1},color=black}
            \vertex{point = {0,2},color=black}
            \vertex{point = {1,2},color=black}
            \vertex{point = {2,2},color=black}
            \polygon{points = {(0,1),(1,1),(1,2),(0,2)},color = orange, status = open,thickness=2}
        \end{polyhedron}
        \label{fig:intersection cones}
    \end{figure}
Intersecting by the four half-spaces, we get the polytopes
 
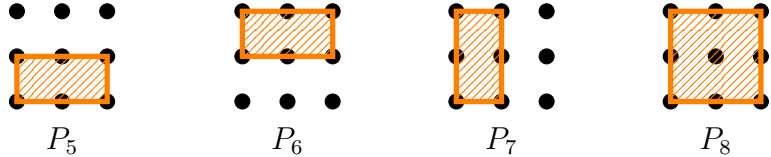
\begin{figure}[H]
        \centering
        \begin{polyhedron}{}
         \vertex{point={1,-0.5},opacity=0,text=$P_5$}           \vertex{point = {0,0},color=black}
            \vertex{point = {1,0},color=black}
            \vertex{point = {2,0},color=black}
            \vertex{point = {0,1},color=black}
            \vertex{point = {1,1},color=black}
            \vertex{point = {2,1},color=black}
            \vertex{point = {0,2},color=black}
            \vertex{point = {1,2},color=black}
            \vertex{point = {2,2},color=black}
            \polygon{points = {(0,0),(2,0),(2,1),(0,1)},color = orange, status = open,thickness=2}
        \end{polyhedron} \quad\quad
            \begin{polyhedron}{}
             \vertex{point={1,-0.5},opacity=0,text=$P_6$}           \vertex{point = {0,0},color=black}
            \vertex{point = {1,0},color=black}
            \vertex{point = {2,0},color=black}
            \vertex{point = {0,1},color=black}
            \vertex{point = {1,1},color=black}
            \vertex{point = {2,1},color=black}
            \vertex{point = {0,2},color=black}
            \vertex{point = {1,2},color=black}
            \vertex{point = {2,2},color=black}
            \polygon{points = {(0,1),(2,1),(2,2),(0,2)},color = orange, status = open,thickness=2}
        \end{polyhedron}\quad\quad
               \begin{polyhedron}{}
                \vertex{point={1,-0.5},opacity=0,text=$P_7$}
            \vertex{point = {0,0},color=black}
            \vertex{point = {1,0},color=black}
            \vertex{point = {2,0},color=black}
            \vertex{point = {0,1},color=black}
            \vertex{point = {1,1},color=black}
            \vertex{point = {2,1},color=black}
            \vertex{point = {0,2},color=black}
            \vertex{point = {1,2},color=black}
            \vertex{point = {2,2},color=black}
            \polygon{points = {(0,0),(1,0),(1,2),(0,2)},color = orange, status = open,thickness=2}
        \end{polyhedron}\quad \quad
        \begin{polyhedron}{}
         \vertex{point={1,-0.5},opacity=0,text=$P_8$}
            \vertex{point = {0,0},color=black}
            \vertex{point = {1,0},color=black}
            \vertex{point = {2,0},color=black}
            \vertex{point = {0,1},color=black}
            \vertex{point = {1,1},color=black}
            \vertex{point = {2,1},color=black}
            \vertex{point = {0,2},color=black}
            \vertex{point = {1,2},color=black}
            \vertex{point = {2,2},color=black}
            \polygon{points = {(0,0),(2,0),(2,2),(0,2)},color = orange, status = open,thickness=2}
        \end{polyhedron}
        \label{fig:intersectino half-spaces}
    \end{figure}
We also have the polytopes $P$ and $Q$ appearing in the complex. The polytope $Q$ appears as the intersection with the cone consisting of all of $\RR^n$ and $P$ appears as $\phi(\hat{1})$. The complex then is
 \[ 0 \rightarrow \mathcal{L}_{P} \rightarrow \mathcal{L}_{P_1}\oplus\mathcal{L}_{P_2}\oplus\mathcal{L}_{P_3}\oplus\mathcal{L}_{P_4} \rightarrow \mathcal{L}_{P_5}\oplus\mathcal{L}_{P_6}\oplus\mathcal{L}_{P_7}\oplus\mathcal{L}_{P_8}\rightarrow \mathcal{L}_{Q}\rightarrow 0.\]
\end{example}

\begin{corollary}\label{cor: truncations of exact complexes are exact}
    If $\IE^{\bullet}$ is an exact inclusion-exclusion complex, then so is any truncation.
\end{corollary}

\begin{proof}
    Say $Q$ is the polytope we used to truncate. If $m \not \in Q$, then the $T$-stalk of the truncation of $\IE^{\bullet}$ is the zero complex and otherwise the $T$-stalk complex of the truncation $\IE^{\bullet}_{(m)}$ is just the stalk of $\IE^{\bullet}_{(m)}$ which is exact by assumption.
\end{proof}

\subsection{The Polytopal Criterion}

We now state the polytopal criterion for constructing full strong exceptional collections.

\begin{definition}
    Let $\mathcal{P} = \{P_i\}$ be a set of polytopes in $\Def(\Sigma)$. We recursively define $\langle \mathcal{P} \rangle_k$ by $\langle \mathcal{P} \rangle_0 = \{P + m \; \lvert \; P \in \mathcal{P}, m\in M\}$ 
and $Q \in \langle {P} \rangle_k$ if and only if there is a CW poset $\mathbb{P}$ and an order-preserving map $F: \mathbb{P} \to \operatorname{Def}(\Sigma)$ satisfying:

\begin{enumerate}
    \item there is an $x$ with $F(x) = Q$
    \item for all $y$ with $\rk(y) \not = \rk(x)$ we have $F(y) \in \langle \mathcal{P} \rangle_{k-1}$.
    \item For any $m \in M$, the upper ideal $\{x \in \mathbb{P} \; \lvert \; m \in F(x)\}$ is an interval of cardinality greater than $1$.
\end{enumerate}
Then, we define 
 \[\langle \mathcal{P} \rangle = \bigcup_{k=0}^{\infty} \langle \mathcal{P} \rangle_k. \]
 We say that the collection $\mathcal{P}$ is \textbf{full} in $\Poly(\Sigma)$ if $\langle \mathcal{P} \rangle = \Def(\Sigma)$.
\end{definition}

\begin{theorem}[\ref{mainthm: Polytopal Criterion}]\label{thm: Polytopal Criterion}
    Let $\Sigma$ be a smooth projective fan. If an ordered collection of polytopes $(P_1,\cdots, P_m)$ for $P_i \in \Def(\Sigma)$ is a full (strongly) exceptional collection of polytopes in $\Poly(\Sigma)$ then the ordered collection of nef line bundles $(\mathcal{L}_{P_1}, \ldots, \mathcal{L}_{P_m})$ is a full (strongly) exceptional collection for $\operatorname{Perf}(X_{\Sigma})$ and hence for $\D(X_{\Sigma})$.
\end{theorem}

\begin{proof}
    By the definition of exceptionality of polytopes, we immediately see that $P_1,\ldots$, $P_m$ are (strongly) exceptional if and only if $(\mathcal{L}_{P_m}^{-1}, \ldots, \mathcal{L}_{P_1}^{-1})$ is (strongly) exceptional. Since taking inverses of line bundles preserves exceptionality while flipping the order, we obtain the exceptionality part of the statement.

    For fullness, if $(P_1,\ldots,P_m)$ is full, then we have exact sequences obtained from inclusion-exclusion complexes that show that the smallest triangulated category of $\operatorname{Perf}(X_{\Sigma})$ containing $\{ \mathcal{L}_{P_i}\}$ for $i \in [k]$ and $m \in M$ contains the category generated by $T$-equivariant nef line bundles. By \cite{Fang2012}, we have that $\operatorname{Perf}(X_{\Sigma})$ is generated by the $T$-equivariant ample line bundles for smooth $\Sigma$. Hence, our collection of nef line bundles generates the entire category. Since they generate $\operatorname{Perf}(X_{\Sigma})$, they also generate its homotopy category which is $\D(X_{\Sigma})$.
\end{proof}

\begin{remark}
    The bottleneck for extending this result to other projective fans is showing that the category $\operatorname{Perf}(X_{\Sigma})$ is generated by $T$-equivariant nef line bundles. Everything else holds for all projective fans.
\end{remark}

\section{First Examples}\label{sec: examples}
We demonstrate our method in small examples. The results here have been proven through various other methods. It will suffice for these examples to consider subdivision Koszul complexes but this will not be true for our main calculations.

\subsection{Projective Space}
We begin with a combinatorial proof of the calculation of the derived category of $\mathbb{CP}^n$ given by Beilinson \cite{Beilinson1978}. Arguably, this is the calculation that set off the serious study of derived categories.

The variety $\mathbb{CP}^n$ is the toric variety associated to the normal fan\footnote{Technically speaking, this fan is not a strictly convex fan since the simplex is not full-dimensional in $\RR^n$. This is fixed by considering the fan in the subspace where the coordinates sum to $1$. We will not worry about this minor difference for exposition reasons.} of the standard simplex 
 \[ \Delta_{n} = \left \{ x \in \RR^n \; \lvert \; 0 \leq x_i \leq 1 \text{ for all $1 \leq i \leq n$ and } \sum_{i=1}^n x_i = 1. \right \}.\]
Any polytope combinatorially equivalent to a simplex is known to be indecomposable under Minkowski sum \cite{Grunbaum1967}. This means that the deformation cone of $\Delta_{n}$ only consists of dilations of the polytope (including the polytope which is a point). Under the toric geometry dictionary, the $k$-th dilation $k \Delta_{n}$ corresponds to the nef line bundle $\mathcal{O}_{\mathbb{CP}^n}(k)$.

We will show that the dilations $\operatorname{pt}, \Delta_{n}, 2\Delta_{n}, \ldots, (n-1)\Delta_{n}$ form a full strongly exceptional collection of polytopes, where $\operatorname{pt}$ is the polytope which is just the point $0$.

We first show they generate the category. Consider the dilated simplex $k \Delta_{n}$ with $k \geq n$ and the collection of simplices $P_i$ for $i \in [n]$ defined as
 \[ P_i = k\Delta_{n} \cap \{ x \in \RR^{n} \; \lvert \; x_i \geq 1\}.\] 
 For any point $x \in k \Delta_{n}$, at least one of the coordinates of $x$ must have a value greater or equal to $1$ since we have $n$ non-negative numbers summing to $k \geq n$. This means that $\bigcup P_i = k \Delta_{n}$. Further, the intersection of these simplices are again simplices in $\Def(\Sigma)$. Therefore this gives a weak subdivision of $k\Delta_n$ in terms of translations of the $P_i$. Using this construction, we can inductively generate any polytope $k\Delta_{n}$ from the polytopes $\{a \Delta_{n}\; \lvert \; 0 \leq  a \leq n-1\}$.

 \begin{figure}[h]
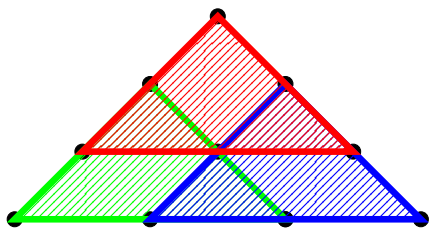

     \centering
 \begin{polyhedron}{} 
 \vertex{point={0,0},color=black}
 \vertex{point={1.5,1.5},color=black}
 \vertex{point={3,0},color=black}
 \vertex{point={-1.5,1.5},color=black}
 \vertex{point={3,0},color=black}
 \vertex{point={4.5,1.5},color=black}
 \vertex{point={-3,0},color=black}
 \vertex{point={6,0},color=black}
 \vertex{point={0,3},color=black}
 \vertex{point={3,3},color=black}
 \vertex{point={1.5,4.5},color=black}
 \polygon{points={(-3,0),(3,0),(0,3)},color=green,thickness=2.5,status=open}
 \polygon{points={(0,0),(6,0),(3,3)},color=blue,thickness=2.5,status=open}
 \polygon{points={(-1.5,1.5),(4.5,1.5),(1.5,4.5)},color=red,thickness=2.5,status=open}
 \end{polyhedron}
     \caption{Weak subdivision of $3 \Delta_{3}$ into three translations of $2\Delta_{3}$. The pairwise intersections are translations of $\Delta_3$ and the intersection of all three is a translation of $\{\operatorname{pt}\}$.}
     \label{fig:Decomposition of Triangle}
 \end{figure}

Next, we need to show that this collection is strongly exceptional. By Proposition \ref{prop: Minkowski sum deformation retract}, we have a homotopy equivalence

\[k_1 \Delta_n \backslash k_2 \Delta_n \sim (k_1 - k_3) \Delta_n \backslash (k_2 - k_3) \Delta_n \]
where $k_3$ is the minimum of $k_1$ and $k_2$. So, this set difference is either $\operatorname{pt} \backslash (k\Delta+m)$ or $k\Delta \backslash (\operatorname{pt}+m)$ where $k < n$. The first is clearly either empty or contractible. The second is contractible whenever $m$ is not an interior lattice point of $k \Delta$. By the standard convex geometric fact that no dilation $k \Delta_n$ with $k < n$ has interior lattice points, we see that this collection is strongly exceptional.

Since full strongly exceptional collections of sheaves are closed under tensoring by line bundles, we recover Beilinson's result.
\begin{theorem}\cite{Beilinson1978}
    The line bundles $(\mathcal{O}(m),\mathcal{O}(m+1),\ldots, \mathcal{O}(m+n-1))$ are a full strongly exceptional collection for $\operatorname{D}^b(\mathbb{CP}^n)$ for any $m \in \NN$.
\end{theorem}

\subsection{Hirzebruch Surfaces}

For any $n \in \mathbb{N}$, the \textbf{Hirzebruch Surface} $\mathcal{H}_n$ is the two-dimensional smooth projective variety given as the projective bundle over $\mathbb{CP}^1$ corresponding to the divisor $\mathcal{O}_{\mathbb{CP}^1} \oplus \mathcal{O}_{\mathbb{CP}^1}(-n)$.  This is a toric variety associated to the complete fan $\Sigma_n \subseteq N \cong \ZZ^2$ determined by the rays
 \[ u_1 = -e_1, u_2 = e_2, u_3 = -e_2, u_4 = e_1 + ne_2.\]
The corresponding fan is the normal fan of the polytope $\operatorname{conv}(0,(n+1)e_1, e_2, e_1+e_2)$.

The deformation cone of the fan $\Sigma_n$ is generated under Minkowski sums by the polytopes $P_0 = \{\operatorname{pt}\},  P_1=\operatorname{conv}(0,e_1),P_3 = \operatorname{conv}(0,ne_1,e_2)$, and $P_4 = \operatorname{conv}(0,(n+1)e_1,e_2,e_1+e_2)$.
\begin{figure}[H]
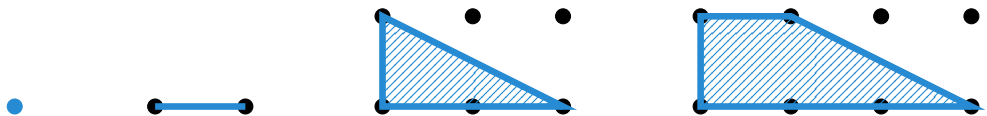

\[\begin{polyhedron}{}
    \vertex{point={0,0}}
\end{polyhedron}
\quad \quad \quad \quad
\begin{polyhedron}{}
    \vertex{point={0,0},color=black}
    \vertex{point={2,0},color=black}
    \polygon{points={(0,0),(2,0)},thickness=2.5,status=open}
\end{polyhedron} \quad \begin{polyhedron}{}
    \vertex{point={-2,0},opacity=0}
    \vertex{point={0,0},color=black}
    \vertex{point={2,0},color=black}
    \vertex{point={4,0},color=black}
    \vertex{point={0,2},color=black}
    \vertex{point={2,2},color=black}
    \vertex{point={4,2},color=black}
    \polygon{points={(0,0),(4,0),(0,2)},thickness=2.5,status=open}
\end{polyhedron} \quad \begin{polyhedron}{}
    \vertex{point={-2,0},opacity=0}
    \vertex{point={0,0},color=black}
    \vertex{point={2,0},color=black}
    \vertex{point={4,0},color=black}
    \vertex{point={0,2},color=black}
    \vertex{point={2,2},color=black}
    \vertex{point={4,2},color=black}
    \vertex{point={6,0},color=black}
    \vertex{point={6,2},color=black}
    \polygon{points={(0,0),(6,0),(2,2),(0,2)},thickness=2.5,status=open}
\end{polyhedron}\]
\caption{The four polytopes $P_0,P_1,P_2,$ and $P_3$ that generate the deformation cone of the fan of $\mathcal{H}_2$.}
\end{figure}
We claim that the four polytopes above give a full strongly exceptional collection for $\D(\mathcal{H}_n)$.

We leave the proof that these polytopes form a strongly exceptional collection as a fun exercise for the reader. Remember that it suffices to show that $P \backslash Q + m$ is contractible for $m \in M$ and that we only consider lattice translations.

Showing that this is a full collection is also fun, but we will spoil it. Given $P \in \Def(\Sigma)$, consider the polytopes obtained by intersecting $P$ by all the strips $a \leq x_1 \leq a + 1$ for $a \in \ZZ$. It is clear that intersecting by these strips preserves all edge directions and gives lattice polytopes, so this weakly subdivides $P$ into polytopes in $\Def(\Sigma)$. Further, the intersections of these pieces are all dilations of the edge $P_1$ and so we get a subdivision Koszul complex that expresses any polytope $P \in \Def(\Sigma)$ in terms of translations of polytopes contained in the strip $0 \leq x_1 \leq 1$.

\begin{figure}[H]
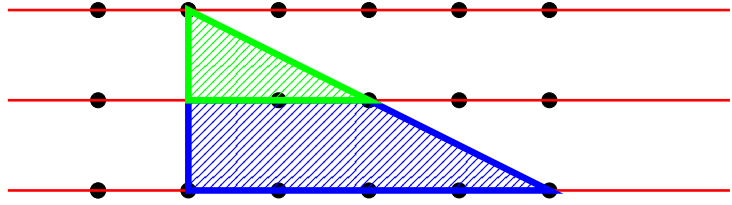

\centering
\begin{polyhedron}{}
\vertex{point = {-2,0},color=black}
\vertex{point = {0,0},color=black}
\vertex{point = {2,0},color=black}
\vertex{point = {4,0},color=black}
\vertex{point = {6,0},color=black}
\vertex{point = {8,0},color=black}
\vertex{point = {-2,0},color=black}
\vertex{point = {-2,2},color=black}
\vertex{point = {2,2},color=black}
\vertex{point = {4,2},color=black}
\vertex{point = {6,2},color=black}
\vertex{point = {8,2},color=black}
\vertex{point = {-2,4},color=black}
\vertex{point = {0,4},color=black}
\vertex{point = {2,4},color=black}
\vertex{point = {4,4},color=black}
\vertex{point = {6,4},color=black}
\vertex{point = {8,4},color=black}

\edge{points = {(-4,2),(12,2)},color=red}
\edge{points = {(-4,4),(12,4)},color=red}
\edge{points = {(-4,0),(12,0)},color=red}
\polygon{points ={(0,0),(8,0),(4,2),(0,2)},thickness=2.5,color=blue,status=open}
\polygon{points ={(0,2),(4,2),(0,4)},thickness=2.5,color=green,status=open}
\end{polyhedron}
\caption{The weak subdivision of $P_2 + P_2$ given by intersecting with the strips. The maximal pieces are a translation of $P_2$ and $P_1 + P_2$ and they intersect at a translation of $2P_1$.}
\end{figure}

By this, it suffices to generate all polytopes contained in the strip $0 \leq x_1 \leq 1$. The only polytopes contained in this strip not already in our collection are $kP_1$ for $k >1$ and $kP_1 + P_3 = (k+1)P_1 + P_2$ for $k \geq 1$. It is clear that the dilated edge $kP_1$ has a weak subdivision given by $k$ copies of $P_1$. 

There is a weak subdivision of $P_1 + P_3$ by $P_3$ and $P_3 + e_1$ where the intersection of the two polytopes is $P_2+ e_1$.
\begin{figure}[H]
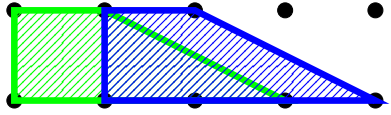

\centering
\begin{polyhedron}{}
\vertex{point = {0,0},color=black}
\vertex{point = {2,0},color=black}
\vertex{point = {4,0},color=black}
\vertex{point = {6,0},color=black}
\vertex{point = {8,0},color=black}

\vertex{point = {0,2},color=black}
\vertex{point = {2,2},color=black}
\vertex{point = {4,2},color=black}
\vertex{point = {6,2},color=black}
\vertex{point = {8,2},color=black}

\polygon{points ={(0,0),(6,0),(2,2),(0,2)},thickness=2.5,color = green, status=open}
\polygon{points ={(2,0),(8,0),(4,2),(2,2)},thickness=2.5,color = blue, status=open}
\end{polyhedron}
\caption{The weak subdivision of $P_1 + P_3$ by two translations of $P_3$ for $\mathcal{H}_2$. Their intersection is $P_2 + e_1$.}
\end{figure}
If we have a weak subdivision of $P$ by $\{P_i\}$ and $Q$ is another polytope in $\Def(\Sigma)$, then we have a weak subdivision of $P+Q$ by $\{P_i + Q\}$. Applying this to $kP_1 + P_3$ which is $P_1 + P_3 + (k-1)P_1$, we have that $\langle P_0,P_1,P_2,P_3 \rangle$ contains all polytope contained in $0 \leq x_1 \leq 1$ which means that it contains all the polytopes in $\Def(\Sigma)$.

With this, we have a combinatorial proof of the following result of Orlov:

\begin{theorem}\cite{Orlov1993}
    The nef line bundles $\{\mathcal{L}_{P_0}, \mathcal{L}_{P_1}, \mathcal{L}_{P_2}, \mathcal{L}_{P_3} \}$ are a fully strong exceptional sequence for $\mathcal{H}_n$.
\end{theorem}

\section{Permutahedral Variety and Base Polytopes of Matroids}\label{Sec: Matroids and the Permutahedral Variety}

    The \textbf{braid arrangement} $\mathcal{Br}_{n}$ is the collection of hyperplanes 
     \[\{H_{ij}\; \lvert \; 1 \leq i < j \leq n+1\}\] 
     given by $H_{ij} = \{x \in \RR^{n} \; \lvert \; x_i = x_j \}$. This hyperplane arrangement is non-essential since the line where $x_i = x_j$ for all $i,j \in [n]$ is contained in all hyperplanes. To fix this, we will consider the hyperplane arrangement $\overline{\mathcal{Br}}_{n}$ obtained by intersecting all the hyperplanes with the subspace $V \subseteq \RR^n$ given by
      \[V = \{ x \in \RR^n\; \lvert \; \sum_{i=1}^{n} x_i = 0\}. \]
     Let $N$ be the lattice $\ZZ^n \cap V$.
     
     The \textbf{braid fan} $\Sigma_{A_{n-1}}$ is the fan whose maximal cones are the chambers of the complement of $\overline{\mathcal{Br}}_n$ in $N$. Let $M$ be the dual lattice. In particular, we have $M_{\RR} \cong \RR^{n}/ \mathbf{1} \RR^n$ where $\mathbf{1}$ is the all one vector.

The facial structure of the braid fan is well-understood. Recall that an \textbf{ordered set partition} of a finite set $E$ is an ordered tuple $A = (A_1,\ldots, A_k)$ of disjoint non-empty sets called blocks such that $A_1 \sqcup A_2 \sqcup \cdots \sqcup A_k = E$.

\begin{proposition}\label{prop: cone descrpition of permutahedral fan}
    There is a bijection between ordered set partitions of $[n]$ and cones of $\Sigma_{A_{n-1}}$. The cone $\sigma_A$ corresponding to $A$ is the set of points in $N$ such that $x_i = x_j$ if $i$ and $j$ are in the same block and $x_i < x_j$ if $i$ is in a block that comes before the block containing $j$.

    In particular, the dimension of the cone $\sigma_A$ is the number of blocks minus one.
\end{proposition}

 We will label the ray $\sigma_{(A_1,A_2)}$ by the subset $A_1$. In this way, the rays are labeled by non-empty and proper subsets of $[n]$.

The \textbf{permutahedral variety} $\Perm_n$ is the toric variety of the braid fan $\Sigma_{A_{n-1}}$. This variety is a smooth, projective variety of dimension $n-1$. The name comes from the fact that this is the toric variety of the permutahedron
 \[\pi_n = \operatorname{conv}( g \cdot (1,2,\ldots,n) \; \lvert \; g \in S_n)\]
in $\RR^n / \RR \mathbf{1}$.
\begin{remark}
    This variety appears in many natural contexts.

    \begin{enumerate}
        \item The first appearance of this variety is probably in Kapranov's work on the Chow quotient of the Grassmannian \cite{Kapranov1992}.
        \item When $n = 3$, this toric variety is isomorphic to the del Pezzo surface $\operatorname{Bl}_3(\mathbb{CP}^2)$. This is the largest $n$ for which $\Perm_n$ is Fano. This is simple to see since the permutahedron for $n \geq 4$ has more than one interior lattice point.
        \item The permutahedral variety $\Perm_n$ embeds into the type $A_n$ flag variety $\operatorname{GL}_n/B$ as the torus orbit closure of a generic point. 
        \item Losev and Manin gave a moduli space interpretation of $\Perm_n$ as a toric compactification of $M_{0,n}$ \cite{Losev2000}.
        \item When $h: [n] \to [n]$ is the function $h(i) = i+1$ for $i \leq n-1$ and $h(n) = n$, then $\Perm_n$ is isomorphic to the semisimple Hessenberg variety associated to $h$ \cite{HessenbergVar}.
        \item It is the wonderful compactification of the hyperplane arrangement given by the coordinate hyperplanes in the sense of \cite{DeConcini95}.
    \end{enumerate}
\end{remark}

\subsection{Matroids and Generalized Permutahedra}\label{Sec: Matroids and Generalized Permutahedra}

The polytopes in the deformation cone of the braid arrangement have been studied by many people in different fields. The first appearance was in optimization under the name polymatroids studied by Edmonds \cite{Edmonds2003} and they were later studied combinatorially under the name generalized permutahedra by Postnikov \cite{Postnikov2009}. There is a slight difference in the definition of polymatroids and generalized permutahedra which is usually insignificant; however, the difference will be important for us. For that reason, we will study generalized permutahedra in this section and polymatroids in the next.

\begin{definition}
    A \textbf{generalized permutahedron} is a polytope in $\RR^n$ such that every edge is parallel to a vector $e_i - e_j$ where $e_i$ is the $i$th standard basis vector of $\RR^n$. 

    Equivalently, $P$ is a generalized permutahedron if and only if the normal fan of $P$ coarsens the fan whose maximal cones are the chambers of $\mathcal{Br}_n$.
\end{definition}

From this definition, we immediately see that every generalized permutahedron is contained in a hyperplane of the form $x_1 + \cdots + x_n = k$ for some $k \in \RR^n$. Further, $\Def(\Sigma)$ consists of the projections of lattice generalized permutahedra $P$ under the map $\RR^n \mapsto \RR^n / \mathbf{1}\RR$, or equivalently, of lattice generalized permutahedra up to translation by $\mathbf{1}$.

Since the nef cone of $\Perm_n$ consists of line bundles indexed by generalized permutahedra up to translation, there is no harm in abusing notation by identifying $\Def(\Sigma)$ with the set of lattice generalized permutahedra in $\RR^n$ when studying $\D(\Perm_n)$. We will freely do this identification for the remainder of this section.

A \textbf{submodular function} $\mu: 2^{[n]} \to \RR$ is a function such that $\mu(\emptyset) = 0$ and
 \[ \mu(A) + \mu(B) \geq \mu(A \cap B) + \mu(A \cup B),\]
for all $A,B \subseteq [n]$. The \textbf{base polyhedron} of $\mu$ is the polytope
 \[P_{\mu} = \{ x \in \RR^n \; \lvert \; \sum_{i =1}^n x_i = \mu([n]), \; \sum_{i \in A} x_i \leq \mu(A) \text{ for all $A \subsetneq [n]$}\}. \]
The map $\mu \mapsto P_{\mu}$ gives a bijection between submodular functions and generalized permutahedra. Further, pointwise sums of submodular functions correspond to Minkowski sum of generalized permutahedra. If we restrict to $\ZZ$-valued submodular functions then we get a cone isomorphic to $\Def(\Sigma_{A_{n-1}})$.

This bijection can be seen geometrically by considering the description of the line bundles $\mathcal{L}_P$ in terms of their Weil divisor. Let $D_S$ denote the $T$-invariant Weil divisor of $\Perm_n$ corresponding to the ray indexed by $\emptyset \not = S \subsetneq [n]$. Then, we have
 \[c_1(\mathcal{L}_{P_{\mu}}) = \sum_{\emptyset \not = S \subseteq [n]} \mu(S) D_S. \]

Our nef line bundles will come from the theory of matroids. See \cite{White1986} for a general reference on matroids or \cite{Katz2016} for an introduction aimed at algebraic geometers. 

For a matroid $M$ on ground set $[n]$, we will use $\mathcal{B}(M)$ to denote its bases, $\mathcal{I}(M)$ to denote its independent sets, and $\rk:2^{[n]} \to \ZZ$ to denote its rank function. The \textbf{base polytope} of a matroid $M$ is the polytope
 \[\BP(M) = \operatorname{conv}(e_B \; \lvert \; B \in \mathcal{B}(M)) \]
where $e_B = \sum_{i \in B}e_i$. Every base polytope of a matroid is a generalized permutahedron with a submodular function given by the rank function of the matroid. In fact, Gelfand, Goresky, MacPherson, and Serganova found the following characterization of base matroid polytopes:

\begin{theorem}\label{thm: ggms}\cite{Gelfand1987}
    A polytope in $\RR^n$ is the base polytope of a matroid if and only if it is a generalized permutahedron whose vertices are contained in the vertices of the cube $[0,1]^n$.
\end{theorem}

A \textbf{loop} of a matroid $M$ is an element $i \in [n]$ such that no basis of $M$ contains $i$ and a \textbf{coloop} is an element that is in every basis of $M$. We say $M$ is \textbf{loopless} if it contains no loops and likewise for \textbf{coloopless}. The following is immediate:

\begin{proposition}\label{prop: flipping loops is translation for base polytopes}
    Let $M$ be a matroid and $i$ be a loop of $M$. Let $M'$ be the matroid obtained by adding $i$ to every basis. We will refer to this operation as changing a loop to a coloop. Then, the polytope $\BP(M)$ is a translation of the polytope $\BP(M')$.

    In particular, the line bundles $\mathcal{L}_{\BP(M)}$ and $\mathcal{L}_{\BP(M')}$ are isomorphic.
\end{proposition}
There are many important operations on matroids, the main ones we will use are the following:

\begin{definition}
Let $M$ be a matroid on ground set $[n]$.
    
\begin{itemize}
 \item The \textbf{dual} of $M$ is the matroid $M^*$ which has bases $\mathcal{B}(M^*) = \{ [n] - B \; \lvert \; B \in \mathcal{B}(M)\}$.
    \item The \textbf{restriction} $M|S$ of a matroid $M$ to a subset $S \subseteq [n]$ is defined as the matroid on ground set $S$ defined by
 \[\mathcal{I}(M|S) = \{I \in \mathcal{I}(M) \; \lvert \; I \subseteq S\}.\]
 \item The \textbf{contraction} of $M$ by $T \subseteq [n]$ is the matroid on ground set $[n] \backslash T$ defined as the dual of restriction
  \[ M / T = (M^*|_{[n] \backslash T})^*.\]
  When $T$ is an independent set of $M$, then $\mathcal{I}(M/T)$ consists of all subsets $A \subseteq [n] \backslash T$ such that $A \cup T$ is an independent set of $M$.
    \item The \textbf{truncation} of a matroid $M$ of rank $r$ is the matroid of rank $r-1$ defined by
     \[\mathcal{I}(T(M)) = \{I \in \mathcal{I}(M)\;\lvert\; \rk(I) \leq r-1 \}. \]
    If $M$ is a matroid of rank $r$ and $k < r$, we will use $T_k(M)$ to denote the truncation of $M$ to rank $k$ defined by $T_k(M) = T^{r-k}(M)$.
    \item The \textbf{direct sum} of two matroids $M_1$ and $M_2$ on ground sets $S$ and $T$, respectively, is the matroid $M_1 \oplus M_2$ defined by
     \[\mathcal{B}(M_1 \oplus M_2) = \{ A \sqcup B \subseteq S \sqcup T \; \lvert \; A \in \mathcal{B}(M_1), B \in \mathcal{B}(M_2) \}. \]
     \item A \textbf{minor} of $M$ is a matroid obtained as a sequence of restrictions and contractions of $M$.
\end{itemize}
\end{definition}
 These operations manifest as operations on base polytopes. In particular, there is the following standard result:

\begin{proposition}\label{prop: BP(M) is closed under faces}
    Every face of $\BP(M)$ is the base polytope of a direct sum of minors of $M$.
\end{proposition}

Our nef line bundles will be indexed by a special class of matroids that have been studied under the names of Schubert matroids, nested matroids \cite{Hampe2017}, shifted matroids \cite{Klivans2007}, and generalized Catalan matroids \cite{Ardila2003}. We will use the first name.

For any linear order $\leq_\ell$ on $[n]$, the \textbf{Gale order} on subsets of $[n]$ of cardinality $k$ is the order given by
 \[ S \leq_\ell T\]
for $S = \{i_1 \leq_\ell \cdots \leq_\ell i_k\}$ and $T = \{j_1 \leq_\ell \cdots \leq_\ell j_k\}$ if and only if $i_a \leq_\ell j_a$ for all $1 \leq a \leq k$.
\begin{definition}
    Let $w \in S_n$ be a permutation, $k = 0,1,\ldots,n$, and $\leq_{w}$ be the Gale order on subsets of $[n]$ of size $k$ induced by the linear order corresponding to $w$. Let $S \subseteq [n]$ with $|S| = k$. The \textbf{Schubert matroid} $\Omega_{S}^{w}$ is the matroid with bases
     \[\mathcal{B}(\Omega_{S}^{w}) = \left \{T \in \binom{[n]}{k} \; \big\lvert \; S \leq_w T \right \}. \]
\end{definition}

\begin{proposition}\label{prop: Schubert matroids are closed under operations}
    The class of Schubert matroids is closed under restrictions, contractions, direct sums, truncations, isomorphisms, and duality.
\end{proposition}

Let $\Sch_n$ denote the set of Schubert matroids on ground set $[n]$ and $\overline{\Sch}_n$ denote the set of loopless Schubert matroids on ground set $[n]$. The \textbf{standard Schubert matroids} are the Schubert matroids $\Omega_{S}^{w}$ with $w = e$. In some sources, the name Schubert matroid refers exclusively to standard Schubert matroids.

\begin{remark}
    These polytopes are moment map images of Schubert cells in the Grassmannians, see \cite{Gelfand1987}. The $w$ comes from the choice of base flag one uses for the Schubert decomposition and $S$ comes from the choice of cell.
\end{remark}

\subsection{Base Polytopes of Schubert Matroids are Exceptional}\label{Sec: Schubert matroids are Exceptional}

In this section, we show that the set of base polytopes of Schubert matroids on $[n]$ is strongly exceptional. There are many ways to prove this result. We choose a proof that highlights the connection with the theory of optimization of submodular functions for potential future connections.

We will construct a sequence of deformation retracts associated to generalized permutahedra that are inspired by the simplex algorithm and the greedy algorithm. For a generalized permutahedron $P$ and any $(i,j) \in \binom{[n]}{2}$, the \textbf{exchange capacity} $c(x,i,j)$ for $x \in P$ is defined as
 \[c(x,i,j) = \max\{\alpha \in \RR^n \; \lvert \; \alpha \geq 0 \; \text{ and } x + \alpha(e_i - e_j) \in P \}.\]

We have the following lemma from the study of flow algorithms for polymatroids:

\begin{lemma}\cite{Fujishige2013}\label{lem: independence of exchange capacity}
Let $P$ be a generalized permutahedron. Let $a_1,a_2,b_1,b_2 \in [n]$ such that
\[ c(x,a_2,b_2) \not = 0 \quad \quad \text{and} \quad \quad c(x,a_1,b_2) = 0.\]
Let $y = x + \alpha (e_{a_2} - e_{b_2})$ for any $0 \leq \alpha \leq c(x,a_2,b_2)$. Then, $c(x,a_1,b_1) = 0$ if and only if $c(y,a_1,b_1) = 0$.
\end{lemma}
 
Let $\leq_{lex}$ denote the lexographical order on $\binom{[n]}{2} = \{(i,j) \; \lvert\; 1 \leq i < j \leq n\}$ induced by the standard order on $[n]$.  Define for $q \in \binom{[n]}{2}$
 \[P_{q} = \{x \in P \; \lvert \; c(x,i,j) = 0 \text{ for all $(i,j) \leq_{lex} q$} \}. \]
We will use the convention that $P_\emptyset = P$ and that $\emptyset \leq q$ for all $q \in \binom{[n]}{2}$. We say that $q-1$ is the element covered by $q$ or $\emptyset$ if $q = (1,2)$.

Let $F_q: P_{q-1} \times [0,1] \to P_{q-1}$ for $q \in \binom{[n]}{2}$ be the homotopy
 \[F_q(x,t) = x + t c(x,i,j) (e_i - e_j). \]

\begin{proposition}\label{prop:deformation retract part of greedy algorithm}
    The map $F_q$ is well-defined and is a strong deformation retract from $P_{q-1}$ to $P_q$.
\end{proposition}

\begin{proof}
    First to see that $F_q$ is well-defined we need to check that $F_q(x,t)$ is indeed in $P_{q-1}$ for any $x \in P_q$ and $t \in [0,1]$. Let $(a_1,b_1) <_{lex} q = (a_2,b_2)$. We need to show that $c(F(x,t),a_1,b_1)$ is $0$ for all such $a_1,b_1$. Note that if $c(x,a_2,b_2) = 0$, then $F(x,t) = x = F(x,0)$ and so this follows from the statement that $x \in P_{q-1}$. Otherwise, we have that $c(x,a_2,b_2)$ and $c(x,a_1,b_1) = 0$. Now it follows from Lemma \ref{lem: independence of exchange capacity} that $F_q(x,t) \in P_{q-1}$.

    We can see from definitions that $F(x,0) = x$ for $x \in P_{q-1}$, $F(x,1) \in P_{q}$, and $F(y,t) = y$ for all $y \in P_{q}$. Hence, the map is a strong deformation retract.
\end{proof}

The reason generalized permutahedra are useful in the theory optimization is that they admit nice greedy, local optimization algorithms. For instance, we have the following result:

\begin{lemma}\cite{Fujishige2013}
    Let $P$ be a generalized permutahedron and $w \in W$. Then, the face of $P$ minimized by $w$ is a vertex and it is the unique point in $P$ such that $c(x,i,j) = 0$ for all $i,j$ such that $w(i) < w(j)$.
\end{lemma}

\begin{lemma}
Let $w = (1,2,\ldots,n)$. Then, $P_{(n-1,n)}$ is the maximal vertex in direction $w$ of $P$.
\end{lemma}

\begin{proof}
This follows immediately from the second part of the previous lemma by noting that $c(x,i,j) = 0$ whenever $i < j$.
\end{proof}

The \textbf{greedy deformation retract} $G_P: P \times [0,1] \to P$ is the composition of deformation retracts
 \[ P \xrightarrow[]{F_{(1,2)}} P_{(1,2)} \xrightarrow[]{F_{(1,3)}} \cdots \xrightarrow[]{F_{(n-1,n)}} P_{(n-1,n)} = \text{pt}.\]

\begin{figure}[H]
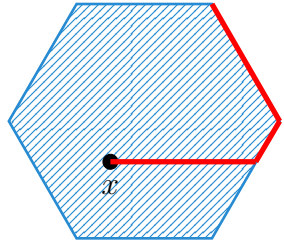

    \centering
    \begin{polyhedron}{}
        \polygon{points={(-3/2,-3/2*1.7320),(3/2,-3/2*1.7320),(3,0),(3/2,3/2*1.7320),(-3/2,3/2*1.7320),(-3,0)},status=open}
        \vertex{point ={-3/2*0.5,-3/2*0.6}, color=black,text=$x$,anchorb=below }
        \edge{points= {(-3/2*0.5,-3/2*0.6),(2.47,-3/2*0.6)},color = red,thickness = 2}
        \edge{points= {(2.47,-3/2*0.6),(3,0)},color = red,thickness = 2}
        \edge{points= {(3,0),(3/2,3/2*1.7320)},color = red,thickness = 2}
    \end{polyhedron}
    \caption{The colored line is the path the point $x$ follows through the greedy deformation retract.}
    \label{fig:greedy deformation retract}
\end{figure}

We will use the following result which states that it suffices to consider the inequalities in the hyperplane description of $\BP(\Omega_S)$ given by intervals $[\ell]$.

\begin{proposition}\label{prop: submodular function only needs intervals}
    Consider the standard Schubert matroid $\Omega_S = \Omega_S^e$ for $S \in \binom{[n]}{k}$ with submodular function $\mu$. Then,
     \[\BP(\Omega_S) = \{ x \in \RR^n \; \lvert \; \sum x_i = k, 0 \leq \sum_{i=1}^{\ell} x_i \leq \mu([\ell]) \text{ for all $\ell \in [n]$}\}.\]
\end{proposition}

\begin{proof}
    This is a consequence of the stronger statement given in \cite{Bidkhori2012} which describes exactly which intervals one needs to consider.
\end{proof}

\begin{theorem}\label{thm: Schuberts are Exceptional}
    The set $\{\BP(\Omega)\}$ for $\Omega \in \overline{\Sch}_n$ is a strongly exceptional collection for $\Poly(\Sigma_{A_{n-1}})$.
\end{theorem}
\begin{proof}
Let $\Omega_1$ and $\Omega_2$ be any matroids in $\overline{\Sch}_n$. We will show that $\BP(\Omega_2) \backslash \BP(\Omega_1) + m$ is contractible or empty for all $m \in M$.

First, if $m \not \in [-1,1]^n \cap \ZZ^n$ then the intersection of $\BP(\Omega_2)$ and $\BP(\Omega_1) + m$ is empty since every coordinate in $\BP(\Omega_2)$ and $\BP(\Omega_1)$ satisfies $0 \leq x_i \leq 1$. Then, the set difference will be homeomorphic to a ball and so contractible. Supposing otherwise, say that $m = e_S - e_T$ for two disjoint subsets $S, T \subseteq [n]$, where $e_A = \sum_{i \in A} e_i$. Then again by the conditions $0 \leq x_i \leq 1$, we have that $\BP(\Omega_2) \cap (\BP(\Omega_1)+m)$ is contained in $[0,1]^n \cap (\BP(\Omega_1)+m)$ which is
 \[\left \{x \in \BP(\Omega_1)\; \lvert \; x_i = 0 \text{ for $i \in S$ and $x_i = 1$ for $i \in T$}\right \} + m\]
If this set is non-empty, then we have by the greedy algorithm that
 \[\left \{x \in \BP(\Omega_1)\; \lvert \; x_i = 0 \text{ for $i \in S$ and $x_i = 1$ for $i \in T$}\right \} = \BP(\Omega_1)_m,\]
where $\BP(\Omega_1)_m$ is the face minimized by $m$. In this face, we have that every $i \in S$ is a loop and every $i \in T$ is a coloop. Therefore, the translation $\BP(\Omega_1)_m + m$ is the matroid polytope for the matroid $M|S$ where the elements in $S$ are changed to coloops and the elements in $T$ are changed to loops. Since faces and these translations of Schubert matroids are again Schubert matroids, we have
 \[\BP(\Omega_2) \backslash \left ( \BP(\Omega_1) + m \right) = \BP(\Omega_2) \backslash \BP(\Omega'), \]
 where $\Omega'$ is another Schubert matroid. Hence, it suffices to study the case where $m = 0$.
 
 If $\Omega_1$ and $\Omega_2$ have different ranks then $\BP(\Omega_2) \backslash \BP(\Omega_1) = \BP(\Omega_2)$ since they will be contained in different hyperplanes $\{x \in \RR^n \;\lvert \; \sum x_i = k\}$. Therefore, we can restrict to the case where $\Omega_1$ and $\Omega_2$ have the same rank $k$.

Consider the action of $S_n$ on $\RR^n$ given by permuting the coordinates. This action sends $\BP(M)$ to the base polytope of a matroid isomorphic to $M$ and every isomorphism of matroids is realized in this way. It is clear that this action does not change the homotopy type of a space. Since each Schubert matroid is isomorphic to a standard Schubert matroid, we can assume without loss of generality that $\Omega_1$ is a standard Schubert matroid $\Omega_S$ for some $S \in \binom{[n]}{k}$. By Proposition \ref{prop: submodular function only needs intervals},
    \[\BP(\Omega_1)  = \{ x \in \RR^n \; \lvert \; \sum x_i = k, 0 \leq \sum_{i=1}^{\ell} x_i \leq \mu([\ell]) \text{ for all $\ell \in [n]$}\}.\]
If $x \in \BP(\Omega_2) \backslash \BP(\Omega_1)$, then one of these inequalities must not hold for $x$. Since $x \in \BP(\Omega_2)$ satisfies both $\sum x_i = k$ and $x_i \geq 0$ for all $i \in [n]$, there must be some $\ell$ so that $\sum_{i=1}^{\ell} x_i > \mu([\ell])$. Now consider the greedy deformation retract $G(x,t)$ on $\BP(\Omega_2)$. By construction, $G(x,t)$ is obtained from $x$ by adding positive roots, i.e. vectors of the form $e_i - e_j$ for $i < j$. Therefore, it can only increase all the sums $\sum_{i=1}^{\ell} x_i$. So, if $x \in \BP(\Omega_2) \backslash \BP(\Omega_1)$, then $G(x,t) \in \BP(\Omega_2) \backslash \BP(\Omega_1)$ for all $t \in [0,1]$. This shows that that we can restrict the deformation retract to $\BP(\Omega_2) \backslash \BP(\Omega_1)$ which implies that this space is contractible when it is not empty.
\end{proof}

This theorem combined with Proposition \ref{prop: exceptionality gives basis} gives a new\footnote{The original proof studies subdivisions of matroids (which we will also do in the next section to prove fullness). However, if we only care about $\McM(\Sigma)$ then it suffices to study exceptionality.} proof of a result that is essentially due to Derksen and Fink.

\begin{corollary}\label{cor: theorem of Derksen and Fink Schuberts}
    The classes of $[\BP(\Omega)]$ for $\Omega \in \overline{\Sch}_n$ form a basis for $\McM(\Sigma_{A_{n-1}})$.
\end{corollary}

Note that in $\McM(\Sigma_{\Perm_n})$ the class of a base polytope of Schubert matroid is the same as the class of the base polytope of the Schubert matroid obtained by flipping loops to coloops. This is why it suffices to restrict to loopless matroids to get a basis.

\subsection{Full Strongly Exceptional Collection for the Permutahedral Variety}\label{Sec: Schubert Matroids Generate}

We now show that $\overline{\Sch}_n$ is a full collection. This is a straightforward categorification of the work of Derksen and Fink. This is especially easy since they proved the following:

\begin{lemma}\label{lem: matroid truncated Brianchon-Gram complexes} For any matroid $M$ of rank $k$, the polytope $\Delta_{k,n}$ truncates $\operatorname{BG}(\BP(M))$ and so we have a truncated Brianchon-Gram complex $\operatorname{TBG}_{\Delta_{k,n}}^{\bullet}(\BP(M))$.
\end{lemma}

This implies that for any matroid polytope $\BP(M)$ where $M$ has rank $k$, there is a truncated Brianchon-Gram complex $\operatorname{TBG}_{\Delta_{k,n}}(\BP(M))$ whose terms consist of base polytopes of Schubert matroids.

The following is a standard result on submodular functions translated into polytopes:

\begin{lemma}\label{lem: subdividing generalized permutahedra into matroids}
    Let $P$ be a lattice generalized permutahedra and let $Q$ be any lattice translation of a face of the unit cube. Then, $P \cap Q$ is a translation of a matroid base polytope. This induces a subdivision of $P$ into matroid polytopes by considering all lattice translations of the unit cube.
\end{lemma}

\begin{proposition}\label{prop: fullness for base polytopes of Schubert matroids}
    The collection of polytopes $\{ \BP(\Omega)\}_{\Omega \in \overline{Sch}_n}$ is full for $\Def(\Sigma_{A_{n-1}})$.
\end{proposition}

\begin{proof}
    If $M$ is a matroid obtained by flipping a loop to a coloop of $M'$, then $\BP(M)$ and $\BP(M')$ are translations of each other. Since every Schubert matroid can be changed into a loopless Schubert matroid by flipping loops, we have that $\BP(\Omega) \in \langle \overline{Sch}_n \rangle_0$ for all $\Omega \in \Sch_n$. Further, every translation of one of these polytopes is in $\langle \overline{Sch}_n \rangle_0$.

 By Lemma \ref{lem: matroid truncated Brianchon-Gram complexes}, we have a truncated Brianchon-Gram complex for the base polytope $\BP(M)$ by intersecting with the hypersimplex $\Delta_{k,n}$ where $k$ is the rank of $M$. Every term in this complex except the first is indexed by a Schubert matroid. Since we have all translations of base polytopes in $\langle \overline{Sch}_n \rangle_0$, this implies that every translation of a base polytopes of any matroid is in $\langle \overline{Sch}_n \rangle_1$.

 By Lemma \ref{lem: subdividing generalized permutahedra into matroids}, we have a strict subdivision of every lattice generalized permutahedron into translations of generalized permutahedra contained in a cube. By Theorem \ref{thm: ggms}, every generalized permutahedron contained in a cube is the base polytope of a matroid. Therefore $\langle \overline{\Sch} \rangle_2 = \Def(\Sigma_{A_{n-1}})$.
    
\end{proof}

Combining this with the results of the previous section, we have shown that $\overline{\Sch_n}$ is a full strongly exceptional collection of polytopes for $\Sigma_{A_{n-1}}$. By the polytopal criterion, we obtain the main result of this section.

\begin{theorem}[Theorem \ref{mainthm: FSEC for Perm}]\label{thm: full strongly exceptional collection for Perm}
   The sequence of line bundles $(\mathcal{L}_{\BP(\Omega)})$ indexed by $\Omega \in \overline{\Sch}_n$ is a full strongly exceptional collection for $\D(\Perm_n)$ when ordered by non-decreasing number of lattice points.
\end{theorem}

\subsection{Invariances of the Collection}

We now turn to study the symmetries of this collection of nef line bundles. We begin with a well-known fact:

\begin{proposition}\label{prop:automorphisms of braid fan}
    The automorphism group of the braid fan $\Sigma_{A_{n-1}}$ is isomorphic to the automorphism group of $S_n$. Hence, the automorphism group is $S_2 \times S_n$. The non-identity element in $S_2$ acts by $x \mapsto -x$ and $S_n$ acts by the action of the Weyl group on the Coxeter arrangement.
\end{proposition}

\begin{proposition}\label{prop: invariant of full strongly exceptional collections for Perm}
    The full strongly exceptional collection $\{\mathcal{L}_{\BP(\Omega)}\}_{\Omega \in \overline{\Sch}_n}$ of $\D(\Perm_n)$ is invariant under the $S_2 \times S_n$ action.
\end{proposition}

\begin{proof}

    The induced action of $S_2 \times S_n$ on $M$ from the action on this variety is one where $S_n$ acts on $M$ by permuting coordinates and the non-identity element of $S_2$ acts on $M$ by the map $\operatorname{cr}:M \to M$
     \[\operatorname{cr}(x) = - x. \]
    To show that this collection of nef line bundles is invariant under this action, it suffices to show that the collection of indexing polytopes is closed under this action and translations since translation does not change the isomorphism type of the line bundle.

    It is immediate from the definitions that $\BP(M_1) = \sigma \cdot \BP(M_2)$ for some $\sigma \in S_n$ if and only if the matroids $M_1$ and $M_2$ are isomorphic. Since loopless Schubert matroids are closed under isomorphism, our collection is invariant under the action of $S_n$.
    
    For the action of the non-identity element $g$ of $S_2$, we have
     \[g \cdot \mathcal{L}_{\BP(\Omega)} = \mathcal{L}_{\operatorname{cr}(\BP(\Omega))} = \mathcal{L}_{\BP(\Omega^{\vee}) - \mathbf{1}} \cong \mathcal{L}_{\BP(\Omega^{\vee})}, \]
    where $\mathbf{1}$ is the all one vector. The dual of a loopless Schubert matroid is a coloopless Schubert matroid, but since changing loops to coloops amounts to translating the polytope, we have that $\mathcal{L}_{\BP(\Omega^{\vee})}$ is some bundle in our collection.
\end{proof}

The permutahedral variety is isomorphic to the Losev-Manin compactification of $M_{0,n}$ and so this gives the strengthening of the result of Castravet and Tevelev \cite{Castravet2020} as described in the introduction.

\subsection{Augmented Inclusions of Matroids and the Tilting Algebra}

We now discuss the tilting algebra corresponding to this fully strong exceptional sequence. It will be more convenient to index our fully strong exceptional sequence by the set of coloopless Schubert matroids. Since changing the coloops of a matroid to loops amounts to a translation of the base matroid polytope, the operation does not change the isomorphism class of the line bundle. Since we still need the class of the point, we consider the rank zero matroid which has no bases as a coloopless Schubert matroid. 

We can describe the tilting quiver purely in terms of matroids.

\begin{definition}
    Let $M_1$ and $M_2$ be coloopless matroids and $S$ be a subset of the loops of $M_1$. We say there is an \textbf{augmented inclusion} of $M_1$ into $M_2$ if for all  $B \in \mathcal{B}(M_1)$ we have that $B \cup S \in \mathcal{B}(M_2)$. We denote an augmented inclusion by
     \[ M_1 \xlongrightarrow{S} M_2.\]
    The \textbf{composition} of two augmented inclusions 
    \[M_1 \xlongrightarrow{S}  M_2 \xlongrightarrow{T}  M_3\]
    is the augmented inclusion
     \[M_1 \xlongrightarrow{S \cup T} M_3. \]
\end{definition}

The augmented inclusions of coloopless Schubert matroids form a category with identity maps given by $M \xlongrightarrow{\emptyset} M$. An augmented inclusion $M_1 \xlongrightarrow{S}M_2$ is \textbf{indecomposable} if it is not a composition of two other augmented inclusions. 


Consider the labeled quiver $Q_n$ where the vertices of $Q_n$ are the coloopless Schubert matroids on ground set $[n]$ and the arrows are the indecomposable augmented inclusions. The arrow corresponding to $M_1 \xlongrightarrow{S} M_2$ is labeled by $S$. For any path $p:v \to u$, let $\operatorname{lab}(p)$ be the union of the labels appearing in the arrows of the path $p$. Define the relations $I$ on the path algebra $\kk Q_n$ by
 \[I = \big \langle p - q \; \lvert \; p,q: v \to u \text{ paths with $\operatorname{lab}(p) = \operatorname{lab}(q)$} \big \rangle. \]

\begin{proposition}\label{prop: tilting quiver for Perm_n}
    The tilting algebra $\operatorname{End}(\mathcal{T})$ is isomorphic to the path algebra of $Q_n$. Thus,
         \[\D(\Perm_n) \cong \operatorname{D}^b(\operatorname{mod}(\kk Q_n/I)^{op}). \]
\end{proposition}

\begin{proof}
    This is only a matter of translating between different notations. In particular, the homomorphisms between nef line bundles $\mathcal{L}_P$ and $\mathcal{L}_Q$ are given by
     \[H^0(\mathcal{L}_Q \otimes \mathcal{L}_P^{-1}) \cong \Hom(\mathcal{L}_P, \mathcal{L}_Q) = \kk[ m \; \lvert \; m \in M, P+m \subseteq Q]. \]
    This isomorphism can be chosen so that the $m$ on the right-hand side corresponds to the section $s \in \mathcal{L}_Q \otimes \mathcal{L}_P^{-1}$ with $\operatorname{div}(s) = \operatorname{div}(\chi^m)$ where $\chi^m$ is the character corresponding to $m$.

    The augmented inclusion $M_1 \xlongrightarrow{S} M_2$ gives rise to an inclusion of base polytopes $\BP(M_1) + e_S \subseteq \BP(M_2)$ and every inclusion of polytopes gives an augmented inclusion. The label $S$ records the translation $e_S$. This inclusion corresponds to the section $s$ with $\operatorname{div}(s) = \operatorname{div}(\chi^m)$.
    
    Then, the definition of the quiver $Q_n$ is just the quiver of sections of these line bundles as described in Section \ref{subsec:generators of derived categories and tilting sheaves}.
\end{proof}

The symmetries of a full strongly exceptional collection descend to give a symmetry of the tilting quiver in the following sense:

\begin{definition}
    Let $Q$ be quiver $(V,E,s,t)$ where $V$ is the set of vertices, $E$ is the set of edges and $s,t: E \to V$ are the source and target functions of edges. A \textbf{group action} of a finite group $G$ on $Q$ is an action of $G$ on $V$ and $E$ such that
     \[s(g \cdot e) = g \cdot s(e) \]
    and
     \[t(g \cdot e) = g \cdot t(e)\]
     for all $e \in E$. The action is \textbf{admissible} if there are no arrows between two vertices in the same orbit under the action of $G$.
\end{definition}

The $S_2 \times S_n$ action on the collection gives an admissible $S_2 \times S_n$ action on the quiver $Q_n$ by the definition of the quiver of sections.

\begin{example}
The six polytopes that give the full strongly exceptional collection for $\Perm_3$ are

\begin{figure}[H]
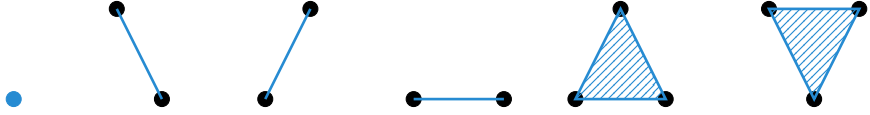

    \centering
    \begin{polyhedron}{}
        \vertex{point = {0,0}}
        \polygon{points = {(0,0)},status = open}
    \end{polyhedron}
    \quad
    \begin{polyhedron}{}
        \vertex{point = {2,0},color=black}
        \vertex{point = {1,2},color=black}
        \polygon{points = {(2,0),(1,2)},status = open}
    \end{polyhedron}
    \quad
    \begin{polyhedron}{}
        \vertex{point = {0,0},color=black}
        \vertex{point = {1,2},color=black}
        \polygon{points = {(0,0),(1,2)},status = open}
    \end{polyhedron}
    \quad
    \begin{polyhedron}{}
        \vertex{point = {0,0},color=black}
        \vertex{point = {2,0},color=black}
        \polygon{points = {(0,0),(2,0)},status = open}
    \end{polyhedron}
    \begin{polyhedron}{}
        \vertex{point = {0,0},color=black}
        \vertex{point = {2,0},color=black}
        \vertex{point = {1,2},color=black}
        \polygon{points = {(0,0),(2,0),(1,2)},status = open}
    \end{polyhedron}
    \quad
    \begin{polyhedron}{}
        \vertex{point = {0,2},color=black}
        \vertex{point = {1,0},color=black}
        \vertex{point = {2,2},color=black}
        \polygon{points = {(0,2),(1,0),(2,2)},status = open}
    \end{polyhedron}
    \caption{Base polytopes of loopless/coloopless Schubert matroids on $[3]$.}
    \label{fig: Base polytopes on 3}
\end{figure}

The tilting quiver is the quiver
\begin{figure}[H]
    \centering
    \begin{tikzpicture}
        \node[shape = circle,draw=black] (A) at (0,0) {$1$};
        \node[shape = circle,draw=black] (B) at (4,2) {$2$};
        \node[shape = circle,draw=black] (C) at (4,0) {$3$};
        \node[shape = circle,draw=black] (D) at (4,-2) {$4$};
        \node[shape = circle,draw=black] (E) at (8,1) {$5$};
        \node[shape = circle,draw=black] (F) at (8,-1) {$6$};
        \begin{scope}[>={Stealth[black]},
              every node/.style={fill=white,circle,align=center,text width=5mm,inner sep = 0},
              every edge/.style={draw = black, very thick}]
    \path[->] (A)  edge[bend left = 10]  node[pos=0.6]{1}(C);
    \path[->](A)  edge[bend right = 10] node[pos=0.6]{3} (C);
    \path[->] (A)  edge[bend left = 10]  node[pos=0.6]{1} (B);
    \path[->] (A)  edge[bend right=10]  node[pos=0.6]{2} (B);
    \path[->] (A)  edge[bend left=10]  node[pos=0.6]{2} (D);
    \path[->] (A)  edge[bend right=10]  node[pos=0.6]{3}(D);
    \path[->] (B)  edge  node[pos=0.3]{3}(E);
    \path[->] (B)  edge (F);
    \path[->] (C)  edge  node[pos=0.3]{2}(E);
    \path[->] (C)  edge (F);
    \path[->] (D)  edge  node[pos=0.3]{1}(E);
    \path[->] (D)  edge (F);
\end{scope}
    \end{tikzpicture}
    \caption{Tilting quiver for $\Perm_3$. The labels of the nodes are chosen in the order of the polytopes above.}
    \label{fig:Tilting Quiver for Perm3}
\end{figure}
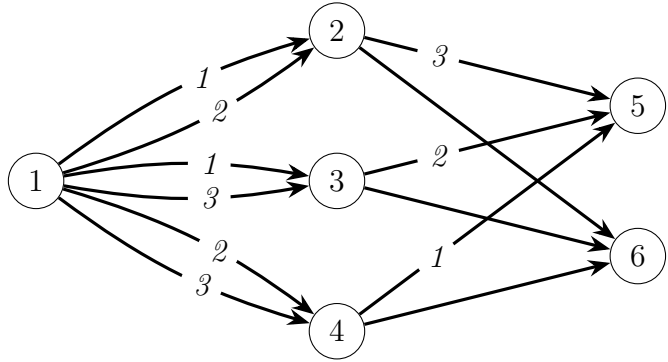

As for the $S_2 \times S_3$-action, the action of $S_2$ permutes the last column of two vertices and the action of $S_3$ permutes the middle column of three vertices.
\end{example}

\section{Stellahedral Variety and Polymatroids}\label{Sec: Stellahedral Variety and Polymatroids}

The Hodge theory of the permutahedral variety played a pivotal role in the resolution of the Rota-Welsh conjecture of matroids given by Adiprasito, Huh, and Katz \cite{Adiprasito2018}. In order to prove similar results for independent sets of matroids, Braden, Huh, Matherne, and Proudfoot introduced the notion of augmented Bergman fans. In this augmentation, the role of the permutahedral variety is played by a variety called the stellahedral variety \cite{Braden2022}. We will show that there is a full strongly exceptional collection for this variety indexed by polytopes that record the independent sets of Schubert matroids.

\begin{definition}
    The \textbf{stellahedron} is the polytope in $\RR^n$ defined as
     \[\Stell_n = \{ x \in \RR^{n} \; \lvert \; y-x = 0 \text{ for some $y \in \Perm_n$}\}. \]
The \textbf{stellahedral fan} $\Sigma_{\Stell}$ is the normal fan of the stellahedron.
\end{definition}

As with the permutahedral variety, the stellahedral variety appears in many other contexts. Notably, it has a moduli space interpretation as the Hassett compactification of $M_{0,n}$ corresponding to the weight $(1, \frac{1}{2},\frac{1}{2} + \varepsilon, \varepsilon,\ldots, \varepsilon)$ where $\varepsilon \leq \frac{1}{n}$, see \cite{RJR16}.

We will use the following description of the cones of the fan given in \cite{Braden2022}:
\begin{definition}
    A pair $(I, F_{\bullet})$ of a subset $I \subseteq E$ and a chain of subsets $F_{\bullet}: F_{1} \subsetneq \cdots \subsetneq F_k \subsetneq E$ is \textbf{compatible} if $I$ is a subset of every $F_i$. We denote this situation by $I \leq F_{\bullet}$.
\end{definition}

\begin{proposition}\cite{Braden2022}
    The cones of the stellahedral fan are in bijection with compatible pairs $I \leq F_{\bullet}$. The cone corresponding to $I \leq F_{\bullet}$ is given by
     \[\sigma_{I \leq F_{\bullet}} = \operatorname{cone}(e_i \; \lvert \; i \in I) + \operatorname{cone}(-e_{E\backslash F} \; \lvert \; F \in F_{\bullet}). \]
    In particular, the rays are given by the cones
     \[\sigma_i := \sigma_{ \{i\} \leq \emptyset} = \operatorname{cone}(e_i) \quad \quad \quad \text{and} \quad \quad \sigma_S := \sigma_{\emptyset \leq \{S\}} = \operatorname{cone}(-e_{E \backslash S}), \]
     for $i \in E$ and $S \subsetneq E$.
\end{proposition}

\subsection{Polymatroids and Independence Polytopes of Matroids}\label{Sec: Polymatroids and Independence Polytopes}

The family of generalized permutahedra was first studied by Edmonds in optimization under the name polymatroid \cite{Edmonds2003}. The two definitions differ in a way that is usually insignificant but that will be important for us.

\begin{definition}
    A \textbf{polymatroid} is a function $\mu: 2^{[n]} \to \RR_{\geq 0}$ such that
    \begin{enumerate}
        \item $\mu(\emptyset) = 0$
        \item $\mu(A) \leq \mu(B)$ whenever $A \subseteq B$
        \item $\mu$ is submodular that is
         \[\mu(S) + \mu(T) \geq \mu(S \cap T) + \mu(S \cup T), \]
         for all $S,T \subseteq [n]$.
    \end{enumerate}
\end{definition}

The base polytopes of polymatroids are exactly the generalized permutahedra contained in the positive orthant. Further, every generalized permutahedron is the base polytope of a polymatroid up to translation.

For any polymatroid $\mu: 2^{[n]} \to \RR_{\geq 0}$, the \textbf{independence polytope} of $\mu$ is
 \[\IP(\mu) = \left \{ x \in \RR^n \; \lvert \; x_i \geq 0, \sum_{i \in A} x_i \leq \mu(A) \text{ for all $A \subseteq [n]$} \right \}.\]
Unlike base polytopes of submodular functions, we do not insist that $\sum_{i \in [n]}x_i = \mu([n])$ only that it is an inequality. The independence polytope of a polymatroid is in general not a generalized permutahedron.

Since the rank function of a matroid is a polymatroid, we will use the notation $\IP(M)$ to denote $\IP(\rk_M)$. For matroids, we have the additional description
 \[\IP(M) = \operatorname{conv}\left(e_I \; \lvert \; I \in \mathcal{I}(M) \right ). \]
From this, we can easily see the following relationship:

\begin{lemma}\label{lem: relation between IP and BP}
Let $M$ be any matroid on $[n]$, then
 \[ \IP(M) = \left ( \BP(M) + [-1,0]^n \right ) \cap [0,1]^n.\]
\end{lemma}

We also have the characterization of which independence polytopes are contained in the cube $[0,1]^n$.

\begin{proposition}\cite{Edmonds2003}\label{prop: independence polytopes in the cube}
    The independence polytopes contained in a cube $[0,1]^n$ are exactly the independence polytopes of matroids.
\end{proposition}

The key relationship between the stellahedral variety and independence polytopes of polymatroids is the following result of Eur, Huh, and Larson.

\begin{proposition}\label{prop: def cone are independence polymatroids}\cite{EHL22}
    There is a bijection between classes of $T$-invariant nef divisors on $\Stell_n$ and polymatroids given by the map
     \[f \mapsto \sum_{\emptyset \subseteq S \subsetneq [n]} f([n]\backslash S) [D_S],  \]
     where $[D_S]$ is the class of the torus-invariant divisor corresponding to the ray $\sigma_{ \emptyset \leq {S}}$.
    Further, the polytope corresponding to the divisor of $f$ is the independence polytope of $f$. Therefore, the deformation cone of $\Sigma_{\Stell_n}$ is the set of translations of independence polytopes of polymatroids on $[n]$.
\end{proposition}

\subsection{Independence Polytopes of Schubert Matroids are Exceptional}\label{Sec: Independence Polytopes of Schubert matroids are exceptional}

We now show that independence polytopes of Schubert matroids from a strongly exceptional collection.

We will reduce this to the permutahedral case via a strong deformation retract from $\IP(M)$ to $\BP(M)$. For any $P \in \Def(\Sigma_{Stell_n})$ and $x \in P$, let
  \[\sat(x,i) = \max\{ \alpha \in \RR \; \lvert \; \alpha \geq 0, x + \alpha e_i \in P \}\]
and
 \[P_i = \{ x \in \RR^n \; \lvert \; \sat(x,j) = 0 \text{ for all $j \leq i$} \}.\]

From the study of flows on polymatroids, we have the following independence result of $\sat$:

\begin{lemma}\cite{Fujishige2013}\label{lem: saturation independence result} Let $i \not = j \in [n]$. For any $0 \leq \alpha \leq \sat(x,i)$, let $y = x + \alpha e_i$. Then,
 \[\sat(x,j) = \sat(y,j). \]
\end{lemma}

We use the convention that $P_0 = P$. Consider the homotopies $F_i: P_{i-1} \times [0,1] \to P_{i-1}$ defined by
 \[F_i(x,t) = x + t \sat(x,i)e_i.\]
Each of these is well-defined by the previous lemma and is a deformation retract from $P_{i-1}$ to $P_i$.
The \textbf{independence-base} deformation retract is the composition
 \[ \operatorname{IB} = F_{n} \circ \cdots \circ F_1.\]
It is clear that this is a deformation retract from $P$ to $P_n$. 

\begin{lemma}\label{lem: fully saturated result is base}
    For any submodular function $\mu$, we have that 
         \[ B(\mu) = \{ x \in P(\mu)\; \lvert \; \sat(x,i) = 0 \text{ for all $i \in [n]$}\}.\]
\end{lemma}

Hence, for any matroid $M$, we have that $\IP(M)_n$ is exactly the base polytope $\BP(M)$. This means that $\operatorname{IB}$ is a deformation retract from $\IP(M)$ to $\BP(M)$.

\begin{proposition}\label{prop: Independence Schuberts are Exceptional}

The set $\{\IP(\Omega)\}$ for $\Omega \in \Sch_n$ is a strongly exceptional collection for $\Def(\Sigma_{\Stell_n})$.
\end{proposition}
\begin{proof}
Let $\Omega_1$ and $\Omega_2$ be any two Schubert matroids on $[n]$. We will show that $\IP(\Omega_2) \backslash \IP(\Omega_1) + m$ is contractible for all $m \in M$.

First, since both $\IP(\Omega_1)$ and $\IP(\Omega_2)$ are contained in $[0,1]^n$, it suffices to consider the case where $m \in [-1,1]^n \cap \ZZ^n$. Otherwise, we would have $\IP(\Omega_2) \cap \IP(\Omega_1) + m = \emptyset$ so the set difference would be the entire polytope $\IP(\Omega_2)$ which is topologically a ball.

Say that $m = e_S - e_T$ where $S$ and $T$ are disjoint subsets of $[n]$. Then,
 \[\left (\IP(\Omega_1) + m\right ) \cap [0,1]^n = \operatorname{conv}\{e_I \; \lvert \;  I \in \mathcal{I}(M), S \cap I = \emptyset, T \subseteq I\} + m. \]

When the polytope on the right is non-empty, it is the independence polytope of the matroid obtained by the following sequence of operations:
\begin{enumerate}
    \item Restrict $\Omega_1$ to the complement of $S$.
    \item Contract by $T$.
    \item Trivially extend the matroid so that its ground set is again $[n]$.
    \item Flip loops to coloops until the elements $i \in S$ are coloops and the $i \in T$ are loops.
\end{enumerate}
By Proposition \ref{prop: Schubert matroids are closed under operations}, the class of Schubert matroids is closed under all these operations and therefore it suffices to restrict to the case where $m =0$.

 For this, consider the restriction to $\IP(\Omega_2) \backslash \IP(\Omega_1)$ of the independence-base deformation retract $F$ for $\IP(\Omega_2)$. For any independence polytope, $P$ if $x \in P$, then $x - \lambda e_i \in P$ for $0 \leq \lambda \leq x_i$. In the contrapositive, this means that if $x \not \in \IP(\Omega_1)$, then $F(x,t) \not \in \IP(\Omega_1)$ for any $t \in [0,1]$. Therefore the independence-base deformation retraction restricts to a strong deformation retract
 \[ \IP(\Omega_2)\backslash \IP(\Omega_1) \to \BP(\Omega_2) \backslash \IP(\Omega_1).\]
Let $r = \rk(\Omega_2)$. If $\rk(\Omega_1) < r$, then $\BP(\Omega_2) \cap \IP(\Omega_1) = \emptyset$, since $\BP(\Omega_2)$ is contained in the hyperplane where $\sum_{i} x_i = r$ and $\IP(\Omega_1)$ is contained in the half-space where $\sum_{i} x_i \leq \rk(\Omega_1)$. Hence this set difference is contractible. Otherwise,
 \[\BP(\Omega_2) \cap \IP(\Omega_1) = \BP(\Omega_2) \cap \BP(T_r(\Omega_1)),\]
where $T_r(\Omega_1)$ is the truncation of $\Omega_1$ to rank $r$. By Proposition \ref{prop: Schubert matroids are closed under operations}, we see that truncations of Schubert matroids are again Schubert matroids. Therefore, we have a strong deformation retract from
 \[ \IP(\Omega_2) \backslash \IP(\Omega_1) \to \BP(\Omega_2) \backslash \BP(\Omega'),\]
where $\Omega'$ is another Schubert matroid. In Theorem \ref{thm: Schuberts are Exceptional}, we showed that this set difference is contractible or empty for any two Schubert matroids.
\end{proof}

As a consequence of this result together with Proposition \ref{prop: exceptionality gives basis}, we have a new proof of the following result of Eur, Huh, and Larson:

\begin{theorem}\label{thm: basis for  K(Stell)}\cite{EHL22}
    The nef line bundles associated to the independence polytope of (all) Schubert matroids on $[n]$ form a basis for $K_0(\Stell_n)$. Under the Chern character, this also gives a basis of $H^*(\Stell_n)$.
\end{theorem}

\subsection{Full Strongly Exceptional Sequence for the Stellahedral Variety}\label{Sec: independence schubert matroids generate}

We now show that the independence polytopes of all Schubert matroids generate $\Def(\Sigma_n)$.

We begin by noting that the cube intersection property for base polytopes also holds for independence polytopes:

\begin{proposition}\label{prop: polymatroid cube intersection}
    Let $P \in \Def(\Sigma_{\Stell_n})$. The intersection of $P$ with a lattice translate of $[0,1]^n$ is again in $\Def(\Sigma_{\Stell_n})$. Therefore, every $P$ has strict subdivisions by translations of polytopes in $\Def(\Sigma_{\Stell_n})$ contained in a cube.
\end{proposition}

\begin{proposition}\label{prop: Indepedence Polytopes Generate}
    The set of indepedence polytopes $\IP(\Omega)$ for $\Omega \in \Sch_n$ is a full collection of polytopes in $\Def(\Stell_n)$.
\end{proposition}

\begin{proof}
    By Proposition \ref{prop: polymatroid cube intersection}, it suffices to show that $\langle \IP(\Omega) \rangle$ contains all translates of independence polytopes that are contained in a cube. Since all independence polytopes contained in the cube are independence polytopes of matroids, it suffices to generate all independence polytopes of matroids.

    Let $M$ be a matroid and $\IP(M)$ be its independence polytope. Using the results of Section \ref{Sec: Schubert Matroids Generate}, there is a truncated Brianchon-Gram complex $\operatorname{TBG}_{\Delta_{k,n}}^{\bullet}$ of $\BP(M)$ corresponding to an order-preserving map $\phi$, where $\phi(\hat{0}) = \BP(M)$ and every other $\phi(x)$ is the base polytope of a Schubert matroid. 

    The polytope $[0,1]^n$ is in $\Def(\Sigma_{\Stell_n})$ since it is the independence polytope of the uniform matroid $U_{n,n}$, which is the matroid where every subset of $[n]$ is an independent set. This means that the translation $[-1,0]^n$ is also in $\Def(\Sigma_{\Stell_n})$. Build the inclusion-exclusion complex where
     \[ \phi'(x) = \phi(x) + [-1,0]^n.\]
     By Proposition \ref{prop: tensor product inclusion-exclusion}, this is an exact inclusion-exclusion complex for $\Sigma_{\Stell_n}$. By Proposition \ref{prop: polymatroid cube intersection}, the cube $[0,1]^n$ intersects $\phi'$ properly. So, we can truncate $\phi'$ by $[0,1]^n$ and so by Corollary \ref{cor: truncations of exact complexes are exact}, we obtain an exact inclusion-exclusion complex corresponding to the poset morphism
      \[ \psi(x) = \left(\phi(x) + [-1,0]^n \right ) \cap [0,1]^n.\]

    By Lemma \ref{lem: relation between IP and BP}, we have that for any matroid $M$,
     \[ \IP(M) = \left (\BP(M) + [-1,0]^n \right ) \cap [0,1]^n. \]
     Therefore, $\psi$ is an exact inclusion-exclusion complex where $\psi(\hat{0}) = \IP(M)$ and every other $\psi(x)$ is the independence polytope of a Schubert matroid. This implies that $\IP(M)$ is in $\langle \IP(\Omega) \rangle_{\Omega \in \Sch_n}$.
\end{proof}

Combining the results of the last two sections, we have that the independence polytopes of Schubert matroids on $[n]$ form a full strongly exceptional collection of polytopes for $\Def(\Sigma_{\Stell_n})$. By the polytopal criterion, we obtain the main result of our section. 

\begin{theorem}[Theorem \ref{mainthm: FSEC for Stell}]\label{thm: full strongly exceptional collection for Stell}
    The sequence of nef line bundles $( \mathcal{L}_{\IP(\Omega)})$ for $\Omega \in \Sch_n$ is a full strongly exceptional collection for $\Stell_n$ when ordered by non-decreasing number of lattice points.
\end{theorem}

\subsection{Invariance of the Collection}
Let $S_n$ act on $N_{\RR} \cong \RR^n$ by
 \[w(e_i) = e_{w(i)}, \]
for $w \in S_n$ and $e_i$ the standard basis of $N$. Define an $S_n$ action on subsets of $I \subseteq [n]$ and chains of subsets $F_{\bullet}$ by
 \[ \phi(I) = \{ \phi(i) \; \lvert \; i \in I\}\]
and 
 \[ \phi(F_{\bullet}): \phi(F_1) \subsetneq \phi(F_2) \subsetneq \cdots \subsetneq \phi(F_k),\]
 where $\phi(i)$ is the usual action of $S_n$ on $[n]$.

\begin{proposition}\label{prop: cone description of stellahedral fan}
    The $S_n$ action on $N$ gives an action on $\Sigma_{\Stell}$.
\end{proposition}

\begin{proof}
    By definition, every map of the action is a lattice isomorphism of $N$. All we need to prove is that this function maps cones to cones. For any $\phi \in S_n$, it is immediate that if $I \leq F_{\bullet}$ is compatible, then so is $\phi(I) \leq \phi(F_{\bullet})$ where $\phi(F_{\bullet})$ is the chain $\phi(F_1) \subsetneq \phi(F_2) \subsetneq \cdots \subsetneq \phi(F_k)$. Further, we can verify that
     \begin{align*}
    \phi(\sigma_{I \leq F_{\bullet}}) &= \phi\left( \operatorname{cone}(e_i \; \lvert \; i \in I) + \operatorname{cone}(-e_{E\backslash F} \; \lvert \; F \in F_{\bullet}) \right) \\
    &=\operatorname{cone}(e_i \; \lvert \; i \in \phi(I)) + \operatorname{cone}(-e_{E\backslash F} \; \lvert \; F \in \phi(F_{\bullet})) \\
    &=\sigma_{\phi(I) \leq \phi(F_{\bullet})}.
    \end{align*}
\end{proof}

\begin{proposition}\label{prop: symmetry of collection for Stell}
    The full strongly exceptional collection of nef line bundles $\mathcal{L}_{\IP(\Sigma)}$ of $\Stell_n$ for $\Omega \in \Sch_n$ is $S_n$-invariant.
\end{proposition}

\begin{proof}
    The induced action of $S_n$ on $M \cong \ZZ^n$ is given again by the permutation action on the coordinates. To show that our collection is invariant, it suffices to show that the polytopes indexing our collection are invariant under the $S_n$ action up to translation.

    For two matroids $M_1$ and $M_2$ and any $g \in S_n$, we have that $g(\IP(M_1)) = \IP(M_2)$ if and only if $M_1$ and $M_2$ are isomorphic matroids. Since Schubert matroids are closed under isomorphism, the result follows.
\end{proof}

\begin{remark} In terms of the Hasset compactification description, this $S_n$-action corresponds to the $S_n$-action on the marked points.
\end{remark}

\subsection{Augmented Weak Maps and the Tilting Algebra}

The description of the tilting quiver follows from the same arguments as for the permutahedral variety. Now, we form a quiver with vertices given by all Schubert matroids (not just the coloopless ones).

\begin{definition}
    Let $M_1$ and $M_2$ be two matroids and $S$ be a subset of the loops of $M_1$. We say that we have an \textbf{augmented weak map} of $M_1$ into $M_2$ if for all $I \in \mathcal{I}(M_1)$, we have that $I \cup S \in \mathcal{I}(M_2)$. We denote the augmented weak map by
     \[ M_1 \xlongrightarrow{S} M_2.\]
    The \textbf{composition} of two augmented weak maps
     \[M_1 \xlongrightarrow{S} M_2 \xdashrightarrow{T} M_3 \]
    is the augmented weak map
     \[ M_1 \xlongrightarrow{S \cup T} M_3.\]
\end{definition}

An augmented weak map is \textbf{indecomposable} if it is not the composition of two other augmented weak maps. 


The \textbf{augmented weak map quiver} $Q_n$ is the labeled quiver whose vertices are all Schubert matroids on ground set $[n]$ and the arrows are indecomposable augmented weak maps. The arrow $\Omega_1 \xlongrightarrow{S} \Omega_2$ is labeled by $S$. For any path $p: v \to u$, let $\operatorname{lab}(p)$ denote the union of the labels of the arrows in the path $p$.

Let $I$ denote the ideal
 \[ I = \big \langle p - q \; \lvert \; p,q:v \to u \text{ for paths $p,q$ with $\operatorname{lab}(p) = \operatorname{lab}(q)$} \big \rangle.\]

\begin{proposition}\label{prop: tilting quiver for Stell}
    The tilting quiver corresponding to the full strongly exceptional collection $(\mathcal{L}_{\IP(\Omega)})_{\Omega \in \Sch_n}$ is $(Q_n,I)$.
\end{proposition}

\begin{proof}
    This follows from the same argument as Proposition \ref{prop: tilting quiver for Perm_n} by noting that there is an augmented weak map $M_1\xlongrightarrow{S} M_2$ if and only if $\IP(M_1) + e_S$ is contained in $\IP(M_2)$.
\end{proof}

    As before, the $S_n$-symmetry of the full strongly exceptional collection gives an admissible $S_n$ action on the quiver by $g \cdot \IP(\Omega) = \IP(g \cdot \Omega)$.

\begin{example}
    For $\Stell_2$, the polytopes in our full strongly exceptional collection are 
    \begin{figure}[H]
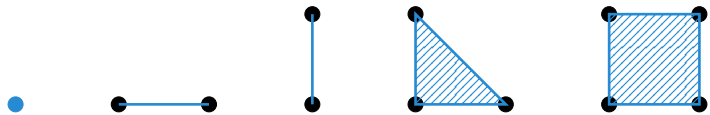

        \centering
        \begin{polyhedron}{}
            \vertex{point = {0,0}}
            \polygon{points = {(0,0)},status=open}
        \end{polyhedron}
        \quad
        \begin{polyhedron}{}
            \vertex{point = {0,0},color=black}
            \vertex{point = {2,0},color=black}
            \polygon{points = {(0,0),(2,0)},status=open}
        \end{polyhedron}
        \quad
        \begin{polyhedron}{}
            \vertex{point = {0,0},color=black}
            \vertex{point = {0,2},color=black}
            \polygon{points = {(0,0),(0,2)},status=open}
        \end{polyhedron}
        \quad
        \begin{polyhedron}{}
            \vertex{point = {0,0},color=black}
            \vertex{point = {2,0},color=black}
            \vertex{point = {0,2},color=black}
            \polygon{points = {(0,0),(2,0),(0,2)},status=open}
        \end{polyhedron}
        \quad
        \begin{polyhedron}{}
            \vertex{point = {0,0},color=black}
            \vertex{point = {2,0},color=black}
            \vertex{point = {0,2},color=black}
            \vertex{point = {2,2},color=black}
            \polygon{points = {(0,0),(2,0),(2,2),(0,2)},status=open}
        \end{polyhedron}
        \caption{The polytopes in the full strongly exceptional collection for $\Stell_2$.}
        \label{fig:collection for stell2}
    \end{figure}
    The corresponding tilting quiver is
    \begin{figure}[H]
    \centering
    \begin{tikzpicture}
        \node[shape = circle,draw=black] (A) at (0,0) {$1$};
        \node[shape = circle,draw=black] (B) at (3,1.4) {$2$};
        \node[shape = circle,draw=black] (C) at (3,-1.4) {$3$};
        \node[shape = circle,draw=black] (D) at (4.5,0) {$4$};
        \node[shape = circle,draw=black] (E) at (7,-0) {$5$};
        \begin{scope}[>={Stealth[black]},
              every node/.style={fill=white,circle,align=center,text width=5mm,inner sep = 0},
              every edge/.style={draw = black, very thick}]
    \path[->] (A)  edge[bend left = 20] node[pos=0.6]{1} (B);
    \path[->](A)  edge[bend right = 20] (B);
    \path[->] (A)  edge[bend left = 20]   (C);
    \path[->] (A)  edge[bend right=20]  node[pos=0.6]{2} (C);
    \path[->] (B) edge  (D);
    \path[->] (C) edge (D);
    \path[->] (B) edge[bend left=20] node {2} (E);
    \path[->] (C) edge[bend right=20] node {1} (E);
    \path[->] (D) edge (E);
\end{scope}
    \end{tikzpicture}
    \caption{Tilting Quiver for $\Stell_2$. The labels of the nodes are chosen in the order of the polytopes above.}
    \label{fig:Tilting Quiver for Stell2}
\end{figure}
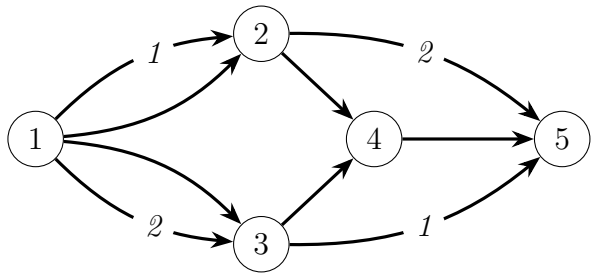

    The $S_2$ symmetry of the quiver is given by reflection along the horizontal line containing three points. In other words, it swaps node $2$ with node $3$ and leaves the rest fixed.
\end{example}

\section{Type $B_n$ Permutahedral Variety}\label{Sec: B_n Permutahedral Variety and Delta Matroids}
We move to our final calculation which is the type $B_n$-analogue of the permutahedral variety $\Perm_n^B$. 

\subsection{Root Systems and Associated Toric Varieties}
We set some notation in the theory of root systems. For all relevant definitions, see \cite{Humphrey12}.

\begin{definition}
Let
\begin{itemize}
    \item $\Phi$ be a finite crystallographic root system.
    \item $W$ be the corresponding Weyl group.
    \item $M$ denote the root lattice.
    \item $N$ denote the dual lattice to $N$ which is the coweight lattice of $\Phi$.
    \item $\omega_k$ denote the $k$th fundamental weight.
    \item $\mathcal{A}_{\Phi}$ be the Coxeter arrangement of $\Phi$.
    \item $\Sigma_{\Phi}$ be the fan generated by the maximal chambers of the Coxeter arrangement.
\end{itemize}
\end{definition}

We will primarily focus on the type $B_n$ case. Let $[\pm n]$ be the set $\{-1,-2,\ldots,-n\} \cup \{1,2,\ldots,n\}$. For this root system, we can choose the data:

\begin{itemize}
    \item The root system is 
     \[\Phi = \{ e_i - e_j \; \lvert \; i \not = j \in [n]\} \cup \{e_i + e_j \; \lvert \; i \not = j \in [n] \} \cup \{e_i \; \lvert \; i \in [n]\}\]
    \item The Weyl group is the group of signed permutations $W = S_n^{B}$ given by
     \[\{f: [\pm n] \to [ \pm n] \; \lvert \; \text{$f$ is a bijection with $f(-i) = -f(i)$}\}.\]
    \item The root lattice $M$ is just the integer lattice $\ZZ^n$.
    \item The fundamental weights are
     \[ \omega_k = e_1 + e_2 + \cdots + e_k,\]
     for $k \leq n-1$ and
      \[\omega_n = \frac{1}{2}(e_1 + e_2 + \cdots + e_n) \]

    \item The Coxeter arrangement $\mathcal{A}_{B_n}$ is the hyperplane arrangement given by the hyperplanes
     \[ \mathcal{H}_{ij} = \{ x \in \RR^n \; \lvert \; x_i = x_j\},\]
     for $i \not =j$ and
      \[ \mathcal{H}^{-}_{ij} = \{ x \in \RR^n \; \lvert \; x_i = -x_j\},\]
      also for $i \not = j$ and
       \[ \mathcal{H}_i = \{x \in \RR^n \; \lvert \; x_i = 0\},\]
       for all $i \in [n]$.
    \item We say that a subset $S \subseteq [\pm n]$ is \textbf{admissible} if $|S \cap \{i, -i\}| \leq 1$ for all $i \in [n]$. The rays of $\Sigma_{B_n}$ are indexed by admissible subsets and the primitive lattice point $v_S$ corresponding to the admissible subset $S$ is
     \[ v_S = \sum_{i \in S \cap [n]} e_i - \sum_{-i \in S \cap [-n]} e_i\]
    for all subsets $S$ with $|S| \leq n-1$ and
     \[ v_S = \frac{1}{2} \left (\sum_{i \in S \cap [n]} e_i - \sum_{-i \in S \cap [-n]} e_i \right ),\]
    when $|S| = n$.
\end{itemize}

In \cite{Batyrev2011}, Batryev and Blume defined a class of smooth, projective toric varieties associated to root systems that generalize the Losev-Manin compactification of $M_{0,n}$.
\begin{definition}
    For any crystallographic root system $\Phi$, the \textbf{type $\Phi$ permutahedral variety} $\Perm_{\Phi}$ is the toric variety associated to the fan $\Sigma_{\Phi}$ in the dual lattice $N = M^*$. 
\end{definition}
\begin{remark}
    The type $A_{n-1}$-permutahedral variety is the permutahedral variety $\Perm_{n}$
\end{remark}

We will use the notation $\Perm_n^B$ to denote the type $B_n$ permutahedral variety.

The (real) deformation cone of the fan $\Def(\Sigma_{\Phi})$ was studied in []. We will slightly adjust this definition by restricting to lattice polytopes.

\begin{definition}\cite{Ardila2020}
    A $\Phi$\textbf{-generalized permutahedron} is a lattice polytope in the root lattice of $\Phi$ whose normal fan coarsens $\Sigma_{\Phi}$. This is equivalently a lattice polytope whose edge directions are parallel to the roots of $\Phi$.
\end{definition}

Every variety $\Perm_{\Phi}$ comes equipped with an action of the Weyl group $W$ induced by the action on the Coxeter arrangement. This also gives an action of $W$ on the root lattice by the usual reflection action.

\begin{remark}
    In Batryev and Blume's construction, the polytopes that index the nef line bundles are polytopes in the root lattice. This is the convention we will take. In other sources, the type $\Phi$-generalized permutahedra are polytopes in the weight or coweight lattice. 
\end{remark}

We now give a description of the cones of $\Sigma_{B_n}$ in terms of signed set partitions as studied by Reiner \cite{Reiner1997}. A \textbf{signed set partition} of $[\pm n]$ is a set partition of a subset of $[\pm n]$ of the form $\{S_0, S_1, \overline{S}_1, \ldots, S_k,\overline{S}_k\}$ such that:
\begin{enumerate}
    \item If $i \in S_0$, then so is $-i$. We refer to this as the zero block of the partition. We allow $S_0$ to be empty.
    \item If $i \in S_{\ell}$, then $-i$ is not in $S_{\ell}$. Likewise for $S_{\ell}^{-}$. We also insist that each $S_{\ell}$ is non-empty.
    \item We have that $\{ -i \; \lvert \; i \in S_{\ell}\} = S_{\ell}^-$.
\end{enumerate}
An \textbf{ordered signed set partition} is a signed set partition alongside an ordering of the subsets such that $S_i$ precedes $S_j$ if and only if $S_j^{-}$ precedes $S_i^{-}$ and $S_{0}$ is in-between $S_i$ and $S_i^{-1}$ for all $i$.

\begin{proposition}\label{prop: cone description of ordered set partitions}
    The cones of $\Sigma_{B_n}$ are in bijection with ordered sign set partitions of $[n]$ where the cone $\sigma_A$ of the ordered signed set partition $A$ is given by 
    
    \begin{enumerate}
        \item $x_i = x_j$ if $i$ and $j$ are in the same part for any $i,j \in [\pm n]$.
        \item $x_i = 0$ if $i \in S_0$.
        \item $x_i > x_j$ if $i$ appears in a block that preceeds the block containing $j$ for $i,j \in [\pm n]$.
    \end{enumerate}
    where we use the convention that $x_{-j} = - x_j$ when $j \in [n]$.
\end{proposition}

\subsection{Coxeter Matroids and Delta Matroids}\label{Sec: Coxeter Generalized Permutahedra}

Our construction of the full strongly exceptional collection for type $B_n$ will be done in analogy with the type $A_n$ case. Our main objects will come from the theory of Coxeter matroids and delta matroids.

\begin{definition}
    Let $W$ be the Weyl group of $\Phi$ and let $I$ be a subset of the simple roots with corresponding parabolic subgroup $W_I$. A \textbf{Coxeter matroid} of type $(\Phi,I)$ is a subset $\mathcal{M} \subseteq W/W_I$ such that for any $w \in W$, there is a unique $A \in \mathcal{M}$ such that for all $B \in \mathcal{M}$ 
     \[ w^{-1}B \leq w^{-1} A\]
where $\leq$ is the Bruhat order on the quotient $W/W_I$.
    
\end{definition}

There have been various different constructions of a Coxeter matroid polytope. In the main modern definition, one usually takes the Coxeter matroid polytope to be the convex hull of the orbit of the sum of fundamental weights $\sum_{i \in I}\omega_i$ under $\mathcal{M}$ \cite{Ardila2020}. Unfortunately, this gives a polytope in the weight lattice (of the root system) while we are looking for polytopes in the root lattice. Instead, we will study a special class of Coxeter matroids coming from the theory of delta matroids and use their associated polytopes.

\begin{definition}
    A \textbf{delta matroid} on ground set $E$ is a collection of subsets $\mathcal{F}(\mathcal{M})$ of $E$, called feasible sets, such that for any two $A_1,A_2 \in \mathcal{F}(\mathcal{M})$ and any $e \in A_1 \Delta A_2$ there exists an $f \in A_1 \Delta A_2$ such that $A_1 \Delta \{e,f\}$ is in $\mathcal{F}(\mathcal{M})$. Here $A_1 \Delta A_2$ represents the symmetric difference of the two sets. The \textbf{feasible polytope} of $\mathcal{M}$ is the polytope
     \[ \FP(\mathcal{M}) = \operatorname{conv}\left (e_S \; \lvert \; S \subseteq \mathcal{F}(\mathcal{M}) \right ).\]
\end{definition}

\begin{proposition}\cite{Borovik2003}\label{prop: delta Matroids in bijectoin with Coxeter matroids wn}
    The set of delta matroids on $[n]$ is in bijection with the set of Coxeter matroids of type $(B_n, \omega_n)$. The bijection is the map induced by the bijection
     \[ 2^{[n]} \to W/W_{I} \]
     given by $S \mapsto [\pi]$ where $\pi$ is the signed permutation
      \[ \pi(i) = \begin{cases}
          -i & \text{if $i \in S$} \\
          i & \text{if $i \not \in S$.}
      \end{cases}.\]
    Further $\FP(\mathcal{M})$ is a type $B_n$ Coxeter matroid polytope and hence a $B_n$-generalized permutahedron.
\end{proposition}

As with type $A_n$ matroids, we have the following classification of which $B_n$-generalized permutahedra come from delta matroids.
\begin{theorem}\cite{Borovik2003}\label{thm: type Bn ggms}\label{thm: Bn coxeter matroids in cube}
    A polytope in $\RR^n$ is the base polytope of a delta matroid if and only if it is a lattice $B_n$-generalized permutahedra and every vertex is a vertex of the cube $[0,1]^n$.
\end{theorem}

A \textbf{loop} of a delta matroid is an element $i \in [n]$ that is in no feasible set and a \textbf{coloop} is an element that is in every feasible set. As with matroids, the feasible polytope of a delta matroid with a loop and the feasible polytope of the delta matroid obtained by changing the loop to a coloop are translations of each other.

The \textbf{type $B_n$ Gale order} on $2^{[n]}$ is the poset given by $A \preceq B$ for subsets $A,B \subseteq [n]$ where $A = \{a_1 < \cdots < a_j\}$ and $B = \{b_1 < \cdots < b_k\}$ if and only if $j \leq k$ and $a_{j-i} \leq b_{k - i}$ for all $i \in \{0,\ldots,j-1\}$. 

\begin{remark}
     By Exercise 8.5 of \cite{Bjo1984}, this order is isomorphic to the induced Bruhat order on the quotient maximal parabolic quotient of $W$ given by the weight $w_{n}$ of type $B_n$. 
\end{remark}

\begin{definition}
    The \textbf{standard Schubert delta matroid} $\Delta_S$ is the delta matroid given by
     \[\mathcal{F}(\Delta_S) = \{ T \subseteq [n] \; \lvert \; \text{$T \preceq S$} \}. \]

    The \textbf{Schubert Delta matroid} $\Delta_S^w$ for $w \in S_n^B$ is the delta matroid defined by the polytope $w \cdot \BP(\Delta_S) + e_T$ where $T = \{ w(i) \in [n] \; \lvert \; w(i) < 0\}$. This translation ensures that this polytope is contained in the cube $[0,1]^n$.
\end{definition}

Let $\operatorname{DSch}_n$ denote the set of delta Schubert matroids on $[n]$ and $\overline{\operatorname{DSch}}_n$ denote those which are loopless.

As with matroids, we have various operations on delta matroids. We only use the following:

\begin{definition}
Let $\mathcal{M}$ be a delta matroid on $[n]$.
\begin{itemize}
    \item The \textbf{restriction} of $\mathcal{M}$ to $S \subseteq [n]$ is the delta matroid on $S$ given by 
     \[\mathcal{F}(\mathcal{M}|S) = \{ B \in \mathcal{F}(\mathcal{M}) \; \lvert \; B \subseteq S\}.\]
     \item If $\mathcal{M}$ is a delta matroid on $S$ and $S \subseteq T$, we say that the \textbf{trivial extension} of $\mathcal{M}$ to $T$ is the delta matroid with the same feasible sets but where we changed the ground set from $S$ to $T$. This amounts to viewing all $i \in T \backslash S$ as loops.
     \item If $i \in [n]$ is a loop of $\mathcal{M}$, we say that the matroid obtained by \textbf{flipping the loop $i$ to a coloop} is the matroid $\mathcal{M}'$
      \[\mathcal{F}(\mathcal{M}') = \{ S \cup \{i\} \; \lvert \; S \in \mathcal{F}(\mathcal{M})\}. \]
\end{itemize}
\end{definition}

\begin{proposition}\label{prop: operations preserve delta matroids}
    The class of Schubert delta matroids is closed under restriction, trivial extensions, and flipping loops.
\end{proposition}

\subsection{Feasible Polytopes of Delta Schubert Matroids are Exceptional}\label{sec: delta matroids are exceptional} In this section we show the set feasible polytopes of loopless Schubert delta matroids is strongly exceptional.

Recall that for a polytope $P$ and a vertex $v$, the tangent cone $C_v$ of $v$ is the cone
 \[C_v = \{x \in \RR^n \; \lvert \; v- \lambda x \in P \text{ for some $\lambda \in \RR^n_{\geq 0}$}\}. \]
It is related to the normal cone $\sigma_v$ of the vertex $v$ by $-\sigma^{\vee} = C_v$.

\begin{lemma}\label{lem: cone difference contraction}
    Let $P \in \Def(\Sigma)$ and $v$ be a vertex of $P$. Let $C$ be a cone which contains the tangent cone of $v$. Then, for any $m \in M$, the set difference $P \backslash C +m$ is contractible or empty.
\end{lemma}

\begin{proof}
    Consider the straight-line deformation retract $F: P \times [0,1] \to \{v\}$ given by
     \[ F(x,t) = v - t(v-x).\]
    This is well-defined since if $x \in P$ then so is the line segment between $x$ and $v$ since $P$ is convex. Suppose that $P \backslash C + m$ is non-empty which happens exactly when $v  \not \in C + m$.

    For any point $x \in P$, we have that $v-x \in C$ since $C$ contains the tangent cone at $v$. Therefore, if $x - t(v-x) \in C+m$ for any $t \in [0,1]$, we must have that $x$ is also in $C+m$ since $C + m$ is closed under addition. In the contrapositive, this means that if $x \in P \backslash C+m$, then $x - t(v-x) \in P \backslash C+m$ for all $t \in [0,1]$.

    This implies that the straight-line deformation retract $F$ restricts to a strong deformation retract which shows that $P \backslash C+m$ is contractible.
\end{proof}

This lemma is useful since the feasible polytopes of delta Schubert matroids polytopes can be understood as intersections of cones with the cube.

\begin{lemma}\label{lem: delta schubert cone intersection}\cite{EHL22}
    Let $S$ be an admissible subset of $[\pm n]$ of cardinality $n$ and $w \in S_n^B$. Then, we have the intersection
     \[ \Delta_S^w = -\sigma_w^{\vee}+ e_{S \cap [n]} \cap [0,1]^n\]
\end{lemma}

\begin{proposition}\label{prop: Delta matroids are exceptional}
    The set of feasible polytopes of loopless Schubert delta matroids on $[n]$ forms a strongly exceptional collection.
\end{proposition}

\begin{proof}
    Let $\Delta_1$ and $\Delta_2$ be any two loopless Schubert delta matroids. We show that $\FP(\Delta_2) \backslash \FP(\Delta_1) +m$ is contractible or empty for all $m \in M$.

     Since $\Delta_2$ is a Schubert delta matroid, we have that $\FP(\Delta_2)$ is contained in the cube $[0,1]^n$ by Theorem \ref{thm: type Bn ggms}. Thus,
      \[ \FP(\Delta_2) \backslash (\FP(\Delta_1) + m) = \FP(\Delta_2) \backslash \left( (\FP(\Delta_1) +m) \cap [0,1]^n \right).\]
      We have that $(\FP(\Delta_1) +m)$ is also contained in the cube $[0,1]^n$. So, in order for $(\FP(\Delta_1) +m) \cap [0,1]^n$ to be non-empty, we must have $m = e_S - e_T$ for disjoint subsets $S$ and $T$ of $[n]$. With this $m$, we can describe the intersection as
       \[(\FP(\Delta_1) + m)\cap[0,1]^n = \{x \in \FP(\Delta_1) \;\lvert \; x_i = 0 \text{ for $i \in S$ and $x_i = 1$ for $i \in T$}\}. \]

     Since every feasible polytope in a cube is the feasible polytope of a delta matroid, it suffices to specify which vectors $e_I$ are in the polytope. By this, the set on the right-hand side is the feasible polytope of the delta matroid $\FP(\Delta)$ where $\Delta$ is the delta matroid obtained by the following steps: first, restrict $\Delta_1$ to $[n] \backslash (S \cup T)$; second take the trivial extension of this restriction by $S \cup T$; and finally, flip the elements in $T$ to coloops. By Proposition \ref{prop: operations preserve delta matroids}, $\Delta$ is again a Schubert delta matroid. Hence, it suffices to consider the case where $m=0$.

     Now consider $\FP(\Delta_2) \backslash \FP(\Delta_1)$. Say that $\Delta_1 = \Delta_S^w$, We then have
      \[\FP(\Delta_2) \backslash \FP(\Delta_1) = \FP(\Delta_2) \backslash \left ( -\sigma_w^{\vee} + e_S \cap [0,1]^n \right ). \]
      Since $\FP(\Delta_2)$ is also contained in $[0,1]^n$ and $e_S$ is a vertex of $[0,1]^n$, we get
       \[ \FP(\Delta_2) \backslash \FP(\Delta_1) = \FP(\Delta_2) \backslash \left ( -\sigma_w^{\vee} + e_S \right ).\]
     Since the normal fan of $\FP(\Delta_2)$ coarsens the type $B_n$ fan $\Sigma_{B_n}$, we see that $\sigma_w^\vee$ contains a tangent cone of $\FP(\Delta_2)$. Therefore, Lemma \ref{lem: cone difference contraction} applies and shows that this set difference is either empty or contractible.
\end{proof}

This result and Proposition \ref{prop: exceptionality gives basis} gives a new proof of the following result of Eur, Fink, Larson, and Spink:

\begin{theorem}\cite{Eur2022}
    The classes $[\FP(\Delta)]$ for $\Delta \in \overline{\operatorname{DSch}}_n$ form a basis for $\McM(\Sigma_{B_n})$ and so the corresponding line bundles give a basis for $K_0(\Perm_n^B) \cong \McM(\Sigma_{B_n})$. Under the Chern character, this also gives a basis of $H^*(\Perm_n^B)$.
\end{theorem}

\subsection{Fully Strongly Exceptional Sequence for $\Perm_n^B$}\label{Sec: FSES of Delta Schubert Matroids}

We now show that the nef line bundles of $X_{B_n}$ corresponding to loopless\footnote{As with the type $A_n$ case, the loopless condition ensures that we only choose each line bundle once since changing loops to coloops gives isomorphic line bundles.} Schubert delta matroids form a full strongly exceptional collection. The arguments here are a type $B_n$ version of the arguments of Sections \ref{Sec: Schubert matroids are Exceptional} and \ref{Sec: Schubert Matroids Generate}.

We have the following Coxeter extensions of the results.

\begin{proposition}\label{prop: Schuberts delta closed under intersection cube}\cite{EHL22}
    The intersection of a $B_n$-generalized permutahedron $P$ with the unit cube $[0,1]^n$ is a translation of the feasible polytope of a delta matroid. Further, this gives a strict subdivision of $P$ into these feasible polytopes.
\end{proposition}

\begin{proposition}\label{prop: Schubert delta matroids are full}
    The set $\overline{\operatorname{DSch}_n}$ of base polytopes of coloopless Schubert delta matroids on ground set $[n]$ polytopally generate $\Def(\Sigma_{B_n})$.
\end{proposition}

\begin{proof}
    If a Schubert delta matroid $\Delta$ contains a loop $i$, then changing the loop to a coloop gives another Schubert delta matroid. From the polytope perspective, this amounts to translating the base polytope of $\Delta$ by $e_i$. Hence, every translation of a Schubert delta matroid is contained in $\langle \overline{\operatorname{DSch}_n} \rangle_0$.

    By Lemma \ref{lem: delta schubert cone intersection}, we see that for any type $B_n$ generalized permutahedron $P$ contained in the cube $[0,1]^n$, we obtain a truncated Brianchon-Gram complex for $P$ where every term besides $P$ consists of intersections of $[0,1]^n$ with cones of $\Sigma_n$. These intersections are exactly base polytopes of delta Schubert matroids by Proposition \ref{thm: type Bn ggms}. Therefore, this exact sequence shows that all $P$ contained in $[0,1]^n$ are in $\langle \overline{\operatorname{DSch}_n} \rangle_1$.

    By Proposition \ref{prop: Schuberts delta closed under intersection cube}, the intersection of any lattice $B_n$-generalized permutahedron with a translation of a cube is again in $\Def(\Sigma_{B_n})$. By tiling $M_{\RR}$ by the lattice translations of $[0,1]^n$, we have a strict subdivision of $P \in \Def(\Sigma_{B_n})$ into translations of polytopes contained in $[0,1]^n$. Using the exact inclusion-exclusion complex given by this subdivision, we see that $P \in \langle \overline{\operatorname{DSch}_n} \rangle_2$ for all $P \in \Def(\Sigma)$.
\end{proof}

Combining this with our previous section, we see that the set of feasible polytopes of loopless Schubert delta matroids on $[n]$ forms a full strongly exceptional collection of polytopes for $\Poly(\Sigma_{B_n})$. By the polytopal criterion, we have
\begin{theorem}[Theorem \ref{mainthm: FSEC for PermB}]\label{thm: full strongly exceptional collection of type Bn}
    The sequence of nef line bundles $( \mathcal{L}_{\FP(\Delta)})$ for $\Delta \in \overline{\operatorname{DSch}_n}$ form a full strongly exceptional collection for $\Perm_n^B$ when ordered by non-decreasing lattice points.
\end{theorem}

\subsection{Invariance of the Collection}

The invariance of our collection follows immediately from root theory. In particular, the Weyl group $W$ acts on the root system and the corresponding Coxeter complex.

\begin{proposition}
    The full strongly exceptional collection $(\mathcal{L}_{\FP(\Delta)})$ for $\Delta \in \overline{\operatorname{DSch}}_n$ is invariant under the action of $S_n^B$.
\end{proposition}

\begin{proof}
    It again suffices to check that the class of polytopes indexing our line bundles is invariant up to translation under the action of $S_n^B$ on $M$. This is the reflection action and we have defined our class of polytopes as the polytopes in the $S_n^B$-orbits of feasible polytopes of standard Schubert delta matroids, so the result follows.
\end{proof}

\subsection{Augmented Inclusion Maps of Delta Matroids and the Tilting Algebra}

As in the type $A_n$ case, the tilting quiver of our collection can be described by augmented inclusion. We will again focus on coloopless Schubert delta matroids which up to translation of feasible polytopes is the same set as loopless Schubert delta matroids.

\begin{definition}
    Let $\Delta_1$ and $\Delta_2$ be two delta matroids on ground set $[n]$ and $S$ be a subset of loops of $\Delta_1$. We say that we have an \textbf{$S$-augmented inclusion} of $\Delta_1$ into $\Delta_2$ if for every $B \in \mathcal{B}(\Delta_1)$ we have $B \cup S \in \mathcal{B}(\Delta_2)$. We denote the augmented inclusion by
     \[\Delta_1 \xlongrightarrow{S} \Delta_2. \]
     The \textbf{composition} of two augmented inclusions 
     \[ \Delta_1 \xlongrightarrow{S} \Delta_2 \xlongrightarrow{T} \Delta_3\]
    is the augmented inclusion
     \[ \Delta_1 \xlongrightarrow{S \cup T} \Delta_3.\]
\end{definition}

We say that an augmented inclusion of delta matroids is \textbf{indecomposable} if it is not the composition of two other augmented inclusions.

Consider the labeled quiver $Q_n$ where the vertices of $Q_n$ are the coloopless Schubert delta matroids on ground set $[n]$ and the arrows are the indecomposable augmented inclusions. The arrow corresponding to $\Delta_1 \xlongrightarrow{S} \Delta_2$ is labeled by $S$. For any path $p:v \to u$, let $\operatorname{lab}(p)$ be the union of the labels appearing in the arrows of the path $p$. Define the relations $I$ on the path algebra $\kk Q_n$ by
 \[I = \big \langle p - q \; \lvert \; p,q: v \to u \text{ paths with $\operatorname{lab}(p) = \operatorname{lab}(q)$} \big \rangle. \]

\begin{proposition}\label{prop: tilting quiver for Perm_nB}
    The tilting quiver correpsonding the full strongly exceptional collection $(\mathcal{L}_{\FP(\Omega)})$ as $\Omega$ ranges over coloopess Schubert delta matroids is $(Q_n, I)$.
\end{proposition}

\begin{proof}
    This follows from the same arguments as Proposition \ref{prop: tilting quiver for Perm_n}. We chose our definition of Schubert delta matroids so that their feasible polytopes are all contained in $[0,1]^n$. So the polytope $\Delta_1 + e_S$ is contained in $\Delta_2$ if and only if $i$ is a loop of $\Delta_1$ for all $i \in S$ and after flipping the loop to a coloop, every feasible set of $\Delta_1$ is a feasible set of $\Delta_2$.
\end{proof}

\begin{example}
The following polytopes index the full strongly exceptional collection of line bundles on $X_{B_2}$
\begin{figure}[H]
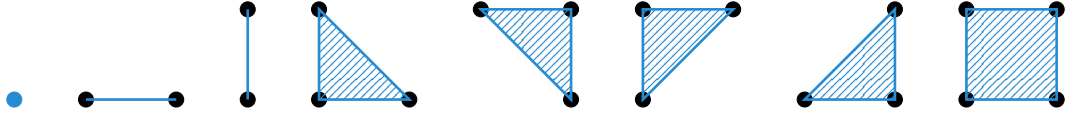

    \centering
        \begin{polyhedron}{}
        \vertex{point ={0,0}}
    \end{polyhedron}
            \begin{polyhedron}{}
        \vertex{point ={0,0},color=black}
        \vertex{point ={2,0},color=black}
        \polygon{points = {(0,0),(2,0)},status = open}
    \end{polyhedron}
        \begin{polyhedron}{}
        \vertex{point ={0,0},color=black}
        \vertex{point ={0,2},color=black}
        \polygon{points = {(0,0),(0,2)},status = open}
    \end{polyhedron}
    \begin{polyhedron}{}
        \vertex{point ={0,0},color=black}
        \vertex{point ={2,0},color=black}
        \vertex{point ={0,2},color=black}
        \polygon{points = {(0,0),(2,0),(0,2)},status = open}
    \end{polyhedron}
        \begin{polyhedron}{}
        \vertex{point ={2,2},color=black}
        \vertex{point ={0,2},color=black}
        \vertex{point ={2,0},color=black}
        \polygon{points = {(2,2),(2,0),(0,2)},status = open}
    \end{polyhedron}
        \begin{polyhedron}{}
        \vertex{point ={0,0},color=black}
        \vertex{point ={2,2},color=black}
        \vertex{point ={0,2},color=black}
        \polygon{points = {(0,0),(2,2),(0,2)},status = open}
    \end{polyhedron}
        \begin{polyhedron}{}
        \vertex{point ={0,0},color=black}
        \vertex{point ={2,0},color=black}
        \vertex{point ={2,2},color=black}
        \polygon{points = {(0,0),(2,0),(2,2)},status = open}
    \end{polyhedron}
        \begin{polyhedron}{}
        \vertex{point ={0,0},color=black}
        \vertex{point ={2,0},color=black}
        \vertex{point ={0,2},color=black}
        \vertex{point ={2,2},color=black}
        \polygon{points = {(0,0),(2,0),(2,2),(0,2)},status = open}
    \end{polyhedron}
    \caption{Polytopes indexing the collection for $X_{B_2}$.}
    \label{fig:polytopes indexing collection for B2}
\end{figure}
    The corresponding quiver is below
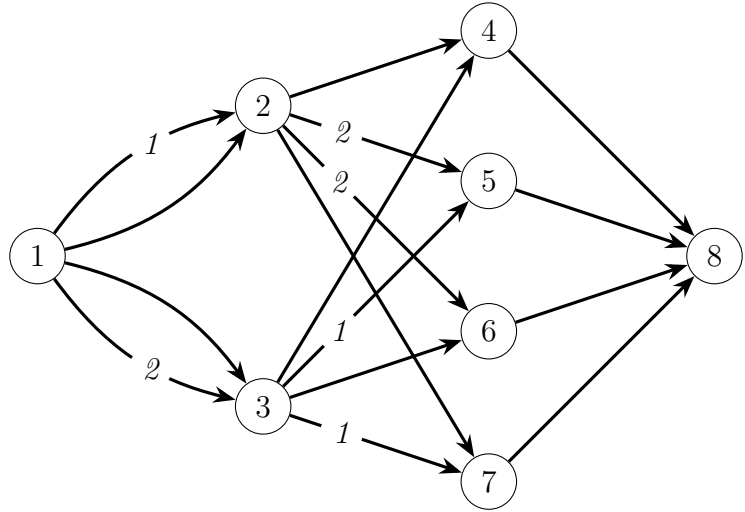
\begin{figure}[H]
    \centering
    \begin{tikzpicture}
    \begin{scope}[>={Stealth[black]},
              every node/.style={draw = black,shape = circle,align=center},
              every edge/.style={draw = black, very thick}]
        \node (A) at (0,0) {$1$};
        \node (B) at (3,2) {$2$};
        \node (C) at (3,-2) {$3$};
        \node (D) at (6,3) {$4$};
        \node (E) at (6,1) {$5$};
        \node (F) at (6,-1) {$6$};
        \node (G) at (6,-3) {$7$};
        \node (H) at (9,0) {$8$};
        \end{scope}

        \begin{scope}[>={Stealth[black]},
              every node/.style={fill=white,circle,align=center,text width=5mm,inner sep = 0},
              every edge/.style={draw = black, very thick}]
    \path[->] (A)  edge[bend left = 20] node[pos=0.6]{1} (B);
    \path[->](A)  edge[bend right = 20] (B);
    \path[->] (A)  edge[bend left = 20]   (C);
    \path[->] (A)  edge[bend right=20]  node[pos=0.6]{2} (C);
    \path[->] (B) edge (D);
    \path[->] (B) edge node[pos=0.3] {2}(E);
    \path[->] (B) edge node[pos=0.3] {2}(F); 
    \path[->] (B) edge (G);
    \path[->] (C) edge (D);
    \path[->] (C) edge node[pos=0.3] {1}(E);
    \path[->] (C) edge (F); 
    \path[->] (C) edge node[pos=0.3] {1}(G);
    \path[->] (D) edge (H);
    \path[->] (E) edge (H);
    \path[->] (F) edge (H);
    \path[->] (G) edge (H);
\end{scope}
    \end{tikzpicture}
    \caption{Tilting Quiver for $\Perm_2^B$. The labels of the nodes are chosen in the order of the polytopes above.}
    \label{fig:Tilting Quiver for Perm B2}
\end{figure}

The group $S_2^B$ is generated by the cycle $(12)$ and the permutation $1 \mapsto -1$ and $2 \mapsto 2$. The first permutation swaps nodes $2$ and $3$ as well as nodes $6$ and $7$. The second permutation swaps nodes $4$ and $7$ as well as nodes $5$ and $6$.
\end{example}

\subsection{What about the other types?}\label{Sec: What about other types?}

Naturally, one would hope that these results extend to the toric varieties of other root systems. However, some difficulties arise with types other than $A$ and $B$. Problems present themselves even in the type $C_n$ case. 

To see this, let $\Perm_n^C$ be the toric variety corresponding to the type $C_n$ root system. The fan of type $C_{n}$ is the same as the fan of type $B_{n}$, but they have different underlying lattice. In particular, the weight lattice $M$ of $\Perm_n^C$ is the root lattice of type $C_n$ which, after choosing a standard description of the root system, consists of lattice points in $\ZZ^n$ whose coordinate sums are even.

Since the theory of Coxeter matroids is the same for type $B_n$ and type $C_n$ root systems, one could hope that the collection of Coxeter matroid polytopes corresponding to Schubert delta matroids continues to be a full strongly exceptional collection. The hope is ill-founded for two reasons.

First, the nef line bundles corresponding to these Coxeter matroids are not all acyclic. Namely, the Schubert matroids in type $C_n$ include the class of the point $\{\operatorname{pt}\}$ and the polytope below

\begin{figure}[H]
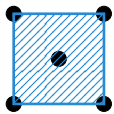

    \centering
    \begin{polyhedron}{}
        \vertex{point = {0,0},color=black}
        \vertex{point = {2,0},color=black}
        \vertex{point = {0,2},color=black}
        \vertex{point = {2,2},color=black}
        \vertex{point = {1,1},color=black}
        \polygon{points = {(0,0),(2,0),(2,2),(0,2)},status = open}
    \end{polyhedron}
    \caption{Type $B_n$ Schubert Coxeter matroid polytope corresponding to the uniform delta matroid.}
    \label{fig:Bn Schubert Coxeter}
\end{figure}

In order for a collection of line bundles that contains the structure sheaf (the line bundle associated to the class of a point) to be strongly exceptional, every sheaf must be acyclic. By standard toric geometry arguments, we know that $H^{n-1}(X,\mathcal{L}_P^{-1})$ has dimension equal to the number of interior lattice points of $P$. The polytope above has an interior lattice point and so has nontrivial $H^1(\mathcal{L}_{P}^{-1})$. Therefore, this collection is not exceptional with respect to any order.

Second, even if it were exceptional, our argument that the type $B_n$ Coxeter matroid polytopes corresponding to Schubert delta matroids is full no longer works. In particular, we no longer have a truncated Brianchon-Gram complex to reduce the arbitrary Coxeter matroid polytope case to the Schubert case. See Remark 3.15 in \cite{Eur2021} for more discussion.

As discussed in \cite{Batyrev2011}, the type $C_n$ permutahedral variety is a quotient of the type $B_n$ permutahedral variety. It might be possible to give a description of the derived category of $\Perm_n^C$ through this quotient but we do not pursue that in this paper.

\section{Compatability of our Collections with the Cuspidal Decomposition}

We now move to our final section where we show that our full strongly exceptional collections and their induced semi-orthogonal decompositions of the derived category are compatible with a recursively defined semi-orthogonal decomposition of varieties studied by Castatev and Tevelev.

In the process of studying the moduli space $\overline{M}_{0,n}$, they defined the following category and the related semi-orthogonal decomposition:

\begin{definition}
    Given a collection of smooth maps $\pi_i: X \to Y_i$ for $i \in I$, we say that a sheaf $\mathcal{E}$ of $X$ is \textbf{cuspidal} if $R\pi_{i*}(\mathcal{E}) = 0$ for all $i \in I$. The \textbf{cusipdal derived category} $\operatorname{D}^b_{cusp}(X)$ is the subcategory of $\D(X)$ consisting of all cuspidal coherent sheaves.
\end{definition}

\begin{theorem}\cite{Castravet2020}\label{thm: cuspidal decomposition}
    Let $X$ be a contravariant functor from the category of subsets of $[n]$ with morphisms given by inclusions to smooth projective varieties. For $S \subseteq T$, let $\pi_{S/T}$ be the map assigned to $X_T \to X_S$. We call these the \textbf{forgetful maps}. Suppose the following three conditions hold:

    \begin{enumerate}
        \item We have
         \[ R\pi_{i*}(\mathcal{O}_{X_S}) = \mathcal{O}_{X_{S \backslash i}},\]
         for all $i \in S$.
         \item For all $i,k \in S$, with $i \not = k$, the morphisms
          \[\pi_i:X_{S\backslash k} \to X_{S \backslash \{i,k\}} \quad \quad \text{and} \quad \quad \pi_{k}: X_{S \backslash i} \to X_{S \backslash \{i,k\}} \]
          are Tor-independent.
          \item If we let
           \[Y := X_{S \backslash i} \times_{X_{S \backslash \{i,k\}}} X_{S \backslash k}\]
           and $\alpha_{i,k}: X_S \to Y$ be the map induced by $\pi_i$ and $\pi_k$, then
            \[R\alpha_{i,k*} \mathcal{O}_{X_S} = \mathcal{O}_Y. \]
    \end{enumerate}
    Under these assumptions, we have a semi-orthogonal decomposition
     \[\operatorname{D}^b(X_S) = \left \langle \D_{cusp}(X_S), \left \{ \pi_{K}^* \D_{cusp}(X_{S \backslash K}) \right \}_{K \subset S}, \pi_S^*\D_{cusp}(X_{\emptyset}) \right \rangle, \]
     as $K$ runs over all proper subsets of $S$ in order of increasing cardinality. We refer to this as the \textbf{cuspidal decomposition} of the category.
\end{theorem}

\subsection{Forgetful Species of Toric Varieties and their Cuspidal Decomposition}

We begin by introducing a special class of contravariant functors for which we always have a cuspidal decomposition. This class is based on Joyal's theory of species \cite{Joyal1981} and Aguiar and Mahajan's theory of Hopf monoids \cite{Aguiar2010} as applied to generalized permutahedra by Ardila and Aguiar \cite{Aguiar2017}.

       For any finite set $S$, let $\RR^S$ denote the vector space of real-valued functions on $S$ and let $e_i$ be the standard basis where $i \in S$. Let $M_S$ be the integer lattice in $\RR^S$ and $N_S$ be the dual lattice.
       
\begin{definition}
        Let $\mathbf{F}$ be a contravariant functor $\mathbf{F}$ from the category of finite sets with inclusions to the category of toric varieties with toric morphisms such that the weight lattice of $\mathbf{F}[S]$ is $M_S$ and the map between $\mathbf{F}[T]$ and $\mathbf{F}[S]$ for $S \subseteq T$ is induced by the inclusion $M_S \hookrightarrow M_T$.  Let $\pi_{S,T}$ denote both the dual map $N_T \mapsto N_S$ and the induced map of toric varieties. We denote the fan of $\mathbf{F}[S]$ by $\Sigma_S$.

        We say that $\mathbf{F}$ is a \textbf{forgetful species of toric varieties} if for every $S \subsetneq T$, we have that
         \[\Sigma_S = \{ \pi_{S,T}(\sigma) \; \lvert \; \sigma \in \Sigma_T\}.\]
\end{definition}

The forgetful maps of a forgetful species satisfy the following nice property:

\begin{definition}[\cite{Abramovich2000,Molcho2021}]
    Let $f: X_{\Sigma_1} \to X_{\Sigma_2}$ be a toric morphism of toric varieties. We say that $f$ is \textbf{weakly semistable} if
    \begin{enumerate}
        \item Every cone of $\Sigma_1$ surjects onto a cone of $\Sigma_2$ under the induced map $N_1 \to N_2$.
        \item Whenever $f(\sigma_1) = \sigma_2$, then we have an equality of monoids $f(\sigma_1 \cap N_1) = \sigma_2 \cap N_2$.
    \end{enumerate}
    We say that $f$ is \textbf{semistable} if the varieties $X_{\Sigma_1}$ and $X_{\Sigma_2}$ are smooth.
\end{definition}

\begin{lemma}
    Let $\mathbf{F}$ be a forgetful species of toric varieties. The forgetful maps $\pi_{S,T}: \mathbf{F}[T] \to \mathbf{F}[S]$ are weakly semistable.
\end{lemma}

The main consequences of this property that we will need are the following two results:

\begin{theorem}\cite{Molcho2021}\label{thm: semistable implies flat}
    If $f: X_{\Sigma_1} \to X_{\Sigma_2}$ is weakly semistable, then $f$ is flat and has reduced fibers.
\end{theorem}

\begin{theorem}\cite{Molcho2021}\label{def: toric fiber product}
    If $f_1: X_{\Sigma_1} \to X_{\Sigma_3}$ is a semistable map, then for any toric morphism $f_2: X_{\Sigma_2} \to X_{\Sigma_3}$, we have that the fiber product $X_{\Sigma_1}\times_{X_{\Sigma_3}}X_{\Sigma_2}$ is a toric variety and hence is the fiber product in the category of toric varieties with toric morphisms.
\end{theorem}

\begin{proposition}
    Let $\mathbf{F}$ be a forgetful species of smooth projective toric varieties. Then, the conditions of Theorem \ref{thm: cuspidal decomposition} hold. Therefore, every $\mathbf{F}[S]$ has a cuspidal decomposition.
\end{proposition}

\begin{proof}
    The conditions follow from the flatness and properness of our forgetful maps.
    \begin{enumerate}
        \item For this, we only have to verify that $\operatorname{R}^p\pi_{i*}(\mathcal{O}_{\mathbf{F}[S]})$ vanishes for $p \geq 1$. It is proven in \cite{Cox2011} that if $f: X \to Y$ is a proper toric morphism then the higher direct images of the structure sheaves vanish. A map between projective varieties is projective and hence proper which implies the condition.
        \item Two maps are Tor-independent when one of them is flat. So it suffices to show that the maps $\pi_{i}$ are flat. This follows from Theorem \ref{thm: semistable implies flat} since our forgetful maps are semistable.
        \item By Theorem \ref{def: toric fiber product}, we know that $Y$ is a toric variety since our forgetful maps are semistable. Since each $\pi_i$ is proper, we know that $\alpha_{i,k}$ is also proper by base change and now the argument for part (1) works again.
    \end{enumerate}
\end{proof}

\begin{definition}
 Let $\mathcal{P}_E$ be a full strongly exceptional collection of lattice polytopes for the fan of $\mathbf{F}[E]$. We say that the collections $\{\mathcal{P}_E\}$ are \textbf{compatible with the forgetful maps} if
\begin{itemize}
     \item For any $S \subseteq E$, the inclusion map $i: M_S \to M_E$ gives a bijection between $\mathcal{P}_S$ and the polytopes $P \in \mathcal{P}_E$ such that some translation of $P$ is contained in $M_S \subseteq M_E$.
 \end{itemize}
\end{definition}

We say that a polytope $P \subseteq \RR^T$ is \textbf{generic} if $P+m \not \subseteq \RR^S$ for any $m \in M$ and any $S \subset T$. For instance, if $P$ is full-dimensional then it is generic.

\begin{theorem}\label{thm: full strongly exceptional collection for forgetful species}
    Let $\mathbf{F}$ be a forgetful species of smooth varieties and $P_S$ be full strongly exceptional collections of polytopes for the fan of the variety $\mathbf{F}[S]$ such that $\{P_S\}$ is compatible with the forgetful maps. Let $\mathcal{P}_T^{gen}$ be the generic polytopes in $\mathcal{P}_T$.
    
    Then the line bundles $\mathcal{L}_{P_i}^{-1}$ for $P_i \in \mathcal{P}_T^{gen}$ form a full strongly exceptional collection for $\D_{cusp}(\mathbf{F}[T])$ when they are ordered by non-increasing number of lattice points.
\end{theorem}

\begin{proof}
    First, we show that these line bundles are in the cuspidal category. Consider a map $\pi_{S,T}: \mathbf{F}[T] \to \mathbf{F}[S]$ for $S \subseteq T$. Let $P \in \mathcal{P}_T^{gen}$. 
    To show that $R\pi_{S,T*}(\mathcal{L}_P^{-1}) = 0$, it suffices to show that $\operatorname{RHom}^p_{\mathbf{F}[S]}(\mathcal{L}_Q^{-1}, R\pi_{S,T*}(\mathcal{L}_P^{-1})) = 0$ for all $p \geq 0$ and for all $Q \in \mathcal{P}_S$. This is true since the collection $\mathcal{L}_{Q}^{-1}$ is a full strongly exceptional collection for $\mathbf{F}[S]$.
    
    Consider $Q \in \mathcal{P}_S$. Let $Q'$ denote the polytope in $\mathcal{P}_T$ that indexes the line bundle $\pi_{S,T}^*(\mathcal{L}_Q)$. This is just the same polytope as $Q \subseteq M_S$ viewed as a polytope in $M_T$. Then, using the push-pull adjunction, exactness of pullbacks for locally free sheaves, and the fact that duality commutes with pullbacks for locally free sheaves, we have
    \begin{align*}
        \operatorname{RHom}^p_{\mathbf{F}[S]}( \mathcal{L}_Q^{-1}, R\pi_{S,T*}(\mathcal{L}_P^{-1})) 
            &= \operatorname{RHom}^p_{\mathbf{F}[T]}( \pi_{S,T}^*(\mathcal{L}_Q)^{-1}, \mathcal{L}_P^{-1})\\
            &= \operatorname{RHom}^p_{\mathbf{F}[T]}(\mathcal{L}_{Q'}^{-1}, \mathcal{L}_{P}^{-1}) \\
            &= \operatorname{RHom}^p_{\mathbf{F}[T]}(\mathcal{L}_P, \mathcal{L}_{Q'}).
    \end{align*}
    Since $P$ and $Q'$ are both in the full strongly exceptional collection $\mathcal{P}_T$, we know that this expression vanishes for $p \geq 1$. All that remains is to compute this for $p = 0$. For this, we have the usual toric-geometric interpretation of Hom-spaces given by
     \[\operatorname{Hom}_{\mathbf{F}[T]}(\mathcal{L}_P, \mathcal{L}_{Q'}) \cong \ZZ[m \in M \; \lvert \; P + m \subseteq Q]. \]
    By construction, $Q'$ is contained in $\RR^S$ and since $P$ is generic there is no $m$ such that $P+m \subseteq Q' \subseteq \RR^S$. This shows that $\operatorname{RHom}^p_{\mathbf{F}[S]}(\mathcal{L}_{Q'}^{-1}, R\pi_{S,T*}(\mathcal{L}_P^{-1})) = 0$ for all $p \geq 0 $ and for all $Q \in \mathcal{P_S}$. Thus $R\pi_{S,T*}(\mathcal{L}_P^{-1}) = 0$. This argument holds for all $S \subset T$ and so $\mathcal{L}_P^{-1}$ is in $\D_{cusp}(\mathbf{F}[T])$.

    Next, we show that $\{\mathcal{L}_{P}^{-1}\}_{P \in \mathcal{P}_T^{gen}}$ generates the cupsidal category. To see this, we show that if $\mathcal{E} \in \D_{cusp}(\mathbf{F}[T])$ and $\mathcal{E} \in \langle \mathcal{L}_{P}^{-1} \rangle_{P \in \mathcal{P}_T^{gen}}^{\perp}$, then it must be $0$. Let $Q'$ be a non-generic polytope in $\mathcal{P}_T$. Suppose $S$ is a subset such that $Q' \subseteq \RR^S$. Then by the compatibility condition, we have that there is some $Q \in \mathcal{P}_S$ such that $\mathcal{L}_{Q'} = \pi_{S,T}^*(\mathcal{L}_Q)$. Since $\mathcal{E}$ is in $\D_{cusp}(\mathbf{F}[T])$, this implies
     \[\operatorname{RHom}_{\mathbf{F}[S]}^p(\mathcal{L}_{Q'}^{-1}, R\pi_{S,T*}(\mathcal{E})) = \operatorname{RHom}_{\mathbf{F}[T]}^{p}(\mathcal{L}_{Q}^{-1}, \mathcal{E}) = 0, \]
     following the same manipulations as earlier in the proof. Therefore $\mathcal{E}$ is orthogonal to all $\mathcal{L}_{P}^{-1}$ with $P \in \mathcal{P}_T$ which means $\mathcal{E} = 0$.

     Finally, there are no extensions between the line bundles $\mathcal{L}_{Q}^{-1}$ for $Q \in \mathcal{P}_{T}^{gen}$ since they were part of a full strongly exceptional collection for $\D(\mathbf{F}[T])$.
\end{proof}
   This implies the semi-orthogonal decomposition for $\D(\mathbf{F}[T])$ induced by the full strongly exceptional collection $\{\mathcal{L}_P^{-1}\}_{P \in \mathcal{P}_T}$ refines the cuspidal decomposition of $\mathbf{F}[T]$.

\subsection{The Permutahedral Varieties}

We now construct forgetful species of smooth varieties out of the permutahedral varieties and show that the full strongly exceptional collections we constructed for them are compatible with the forgetful maps.

For any finite set $E$, let $\Perm_E$ be the variety defined on the fan $\Sigma_E \subseteq \{x \in \RR^E\; \lvert \; \sum x_i = 0\}$ whose top-dimensional cones are the chambers of the hyperplane arrangement given by the hyperplanes
 \[ A_{i,j} = \{ x \in \RR^E \; \lvert \; x_i = x_j\},\]
 for $i \not = j \in E$ all intersected with $\{x \in \RR^E\; \lvert \; \sum x_i = 0\}$. Let $\mathbf{Perm}[E] = \Perm_E$. It is clear that $\Perm_E$ is isomorphic to $\Perm_n$ where $n = |E|$. We can similarly extend the definitions of $\Stell_n$ and $\Perm_n^{B}$ to get varieties $\mathbf{Stell}[E]$ and $\mathbf{Perm}^B[E]$.

For the following proof, we will use the cone descriptions of the fans of our varieties from Propositions \ref{prop: cone descrpition of permutahedral fan}, \ref{prop: cone description of stellahedral fan}, and \ref{prop: cone description of ordered set partitions}. Those statements were only made when $E = [n]$ but they naturally extend for any finite set $E$.

\begin{proposition}\label{prop: forgetful maps for permutahedral varieties}
    For the functors $\mathbf{Perm},\mathbf{Stell}$, and $\mathbf{Perm}^B$, the maps $\pi_S: N_E\to N_S$ for $S \subseteq E$ given by projecting onto the coordinates labeled by $S$ make these into forgetful species.
\end{proposition}

\begin{proof}
    We do this case by case. All we need to show is that if we have a cone $\sigma \in \Sigma_{E}$, then the cone obtained by projecting onto the coordinates indexed by $S$ is a cone in $\Sigma_S$ and that all cones of $\Sigma_S$ arise in this way. This is not too hard since we have exact descriptions of all the cones in these cases.

    \begin{enumerate}
        \item We begin with $\mathbf{Perm}[E]$. Let $\sigma_A$ be the cone indexed by the ordered set partition $(A_1,\ldots, A_k)$. This is the cone where $x_i = x_j$ for two coordinates in the same part and $x_i < x_j$ when $i$ appears in a part that comes before the part containing $j$. When we remove the coordinates indexed by $i \in S$, then we will get the cone indexed by the set partition $A|S$ which is obtained by removing all elements in $E \backslash S$ and then removing all empty parts of the ordered set partition.

        It is clear that every ordered set partition of $S$ is obtained in this way from an ordered set partition of $E$.
        \item Now consider $\mathbf{Stell}[E]$. Let $I \leq F_{\bullet}$ be a compatible pair and let $\sigma_{I \leq F_{\bullet}}$ be the cone indexed by it. That is,
         \[\sigma_{I \leq F_{\bullet}} = \operatorname{cone}(e_i \; \lvert \; i \in I) + \operatorname{cone}(-e_{E\backslash F} \; \lvert \; F \in F_{\bullet}). \]
        If we project onto $S$, we get the cone
         \[\operatorname{cone}(e_i \; \lvert \; i \in I \cap S) + \operatorname{cone}(-e_{S\backslash F} \; \lvert \; F \in F_{\bullet}).  \]
        This is the cone of $\Sigma_S$ indexed by the pair $S \cap I \leq F_{\bullet}'$ where $F_{k}' = F_k \cap S$. If $I \leq F_{\bullet}$ is compatible then so is $I \cap S \leq F_{\bullet}'$. 
        
        Further, every compatible pair of subsets and flags of subsets of $S$ is a compatible pair for $E$ by just considering every subset of $S$ as a subset of $E$. So, every cone arises is this way.

        \item Finally, consider $\mathbf{Perm}^B[E]$. Let $A$ be an ordered signed set partition and let $\sigma_A$ be the corresponding cone of $\Sigma_E$ as in Proposition \ref{prop: cone description of ordered set partitions}. Then, by a similar argument to the one in the permutahedral case, the projection onto the coordinates indexed by $S$ gives the cone indexed by the ordered signed set partition $A|S$ where we delete all indices not in $S$ and delete all non-empty subsets except the zero block. Further, any ordered signed set partition of $E$ arises in this way from a signed set partition of $S$.
    \end{enumerate}
    This shows that all of these functors are forgetful species with these maps. Notice that in the process, we showed that these maps are compatible with the fan and hence give toric morphisms.
\end{proof}

 Let $\overline{\Sch}_E, \Sch_E$ and $\overline{\operatorname{DSch}}_E$ denote the sets of loopless Schubert matroids on ground set $E$, all Schubert matroids on ground set $E$, and loopless Schubert delta matroids on ground set $E$, respectively. It is clear that our calculation of full strongly exceptional collections never depended on the ground set. So, we have full strongly exceptional collection of nef line bundles for $\D(\mathbf{Perm}[E]), \D(\mathbf{Stell}[E]),$ and $\D(\mathbf{Perm}^B[E])$ labeled by the polytopes in $\overline{\Sch}_E, \Sch_E$ and $\overline{\operatorname{DSch}}_E$, respectively.

\begin{proposition}
    The collections $\{\BP(\Omega)\}_{\overline{\Sch}_E}, \{\IP(\Omega)\}_{\Sch_E}$ and $\{\FP(\Omega)\}_{\overline{\operatorname{DSch}}_E}$ are compatible with the forgetful maps of the forgetful species $\mathbf{Perm}[E]$, $\mathbf{Stell}[E]$, and $\mathbf{Perm}^B[E]$, respectively.
\end{proposition}

\begin{proof}
    We only do this for $\mathbf{Perm}[E]$ since all the cases are essentially identical. We need to show that base polytopes of loopless Schubert matroids on ground set $E$ where some translation is contained in $M_S \subseteq M_E$ are in bijection with the polytopes obtained by embedding base polytopes of loopless Schubert matroids on ground set $S$ under the map $M_S \hookrightarrow M_E$.

    This follows immediately from the fact that embedding the base polytope of a loopless Schubert matroid on ground set $S$ under the map $M_S \hookrightarrow M_E$ is exactly the operation of trivially extending from $S$ to $E$ and flipping loops to coloops which corresponds to translations. We also have the observation that if the base polytope of a matroid on ground set $[E]$ is contained in a translation of this hyperplane, then each $i \in S$ must be a loop or coloop. Therefore the polytopes in $\mathcal{P}_S$ are in bijection with the base polytopes of loopless Schubert matroids on $E$ such that each $i \in E \backslash S$ is a coloop, which are exactly the polytopes in $\mathcal{P}_E$ which have a translation contained in $M_S$
\end{proof}

By noting that the generic polytopes in each of the families above are exactly given by loopless and coloopless Schubert matroids and Schubert delta matroids, we obtain a calculation for the cuspidal part of all of our varieties.

\begin{theorem}[Theorem \ref{mainthm: cuspidal part fsec}]\label{thm: fsec for cuspidal categories}
    Let $E$ be a finite set.

    \begin{enumerate}
        \item The line bundles $(\mathcal{L}_{\BP(\Omega)}^{-1})$ indexed by loopless and coloopless Schubert matroids give a full strongly exceptional collection for $\D_{cusp}(\mathbf{Perm}[E])$ when ordered by non-increasing number of lattice points.
        \item The line bundles $(\mathcal{L}_{\IP(\Omega)}^{-1})$ indexed by loopless and coloopless Schubert matroids give a full strongly exceptional collection for $\D_{cusp}(\mathbf{Stell}[E])$ when ordered by non-increasing number of lattice points.
        \item The line bundles $(\mathcal{L}_{\FP(\Delta)}^{-1})$ indexed by loopless and coloopless Schubert delta matroids give a full strongly exceptional collection for $\D_{cusp}(\mathbf{Perm}^B[E])$ when ordered by non-increasing number of lattice points.
    \end{enumerate}
\end{theorem}

\subsection{Relation to Castavet and Tevelev's Collection} Castravet and Tevelev constructed a full exceptional collection of sheaves for $\D(\Perm_n)$ by constructing one for $\D_{cusp}(\Perm[S])$, for all finite sets $S$, and then using the cuspidal decomposition from Theorem \ref{thm: cuspidal decomposition}. They described their sheaves in terms of external tensor products of line bundles coming from the wonderful compactification\footnote{This is the wonderful compactification of groups which is different than the wonderful compactification of hyperplane arrangements.} of $\operatorname{PGL}_n$. There is a different way to describe their collection that relates it to our collection.

    The permutahedral variety $\Perm_n$ embeds into the flag variety $G/B$ of type $A_n$ as the torus orbit closure of a generic point. Let $i: \Perm_n \hookrightarrow G/B$ be the inclusion map. By the Borel-Weil theorem, we know that the nef line bundles of $G/B$ are parameterized by the dominant weights. 

    For any general dominant weight $\lambda$, let $\mathcal{L}_{\lambda}$ be the line bundle associated to $\lambda$. Let $G_i = i^*(\mathcal{L}_{\omega_i})$ for $i =1, \ldots, n-1$ where $\omega_i = e_1 + e_2 + \cdots + e_i$ is the $i$th fundamental weight. In this language, their result is:

    \begin{theorem}\cite{Castravet2020}
    The derived category $\D_{cusp}(\Perm_n)$ has a full strongly exceptional collection of sheaves of cardinality equal to the number of derangements. It contains the $n-1$ line bundles $G_i$ for $i=1,\ldots,n-1$.
    \end{theorem}
    
    The nef line bundle  $i^*(\mathcal{L}_{\lambda})$ on $\Perm_n$ is the nef line bundle $\mathcal{L}_P$ given by the polytope
     \[ P = \operatorname{conv}( \sigma \cdot \lambda \; \lvert \; \sigma \in S_n).\]
    These are the weight polytopes studied for instance by Postnikov in \cite{Postnikov2009}. When applied to the fundamental weights, we see that these polytopes are exactly the base polytopes of the uniform matroid of rank $1, \ldots, n-1$. These are coloopless and loopless Schubert matroids. Therefore, our collection of line bundles includes the line bundles $G_i$ from the collection of Castravet and Tevelev. It would be interesting to find a relationship between the other sheaves in their collection and matroids.

    Another interesting aspect of this comparison is that Castravet and Tevelev's showed that their collection has cardinality equal to the number of derangements of $n$. These are the permutations of $[n]$ which have no fixed points. This gives a curious equicardinality.

    \begin{proposition}
        The number of loopless and coloopless Schubert matroids on $[n]$ is equal to the number of derangements on $[n]$.
    \end{proposition}

We have a type $B_n$-version of this calculation. A \textbf{signed derangement} is a signed permutation $\pi: [\pm n] \to [\pm n]$ such that $\pi(i) \not = i$ for all $i \in [\pm n]$.

\begin{proposition}
    The number of loopless and coloopless Schubert delta matroids on $[n]$ is equal to the number of signed derangements of $[\pm n]$. Hence, the rank of $K(\D_{cusp}(\Perm_n^{B}))$ is the number of signed derangements of $[\pm n]$.
\end{proposition}

\begin{proof}
    Let $Sch(n)$ be the number of loopless and coloopless Schubert delta matroids on $[n]$. By the cuspidal decomposition of $\D(X_{B_n})$ and Theorem \ref{thm: fsec for cuspidal categories}, we get the following recurrence for $|S_n^{B}|$
     \[|S_n^{B}| = \sum_{A \subseteq [n]} Sch(n-|A|), \]
     where $Sch(0) = 1$. 
     
     Let $Der_B(n)$ be the number of signed derangements of size $n$. Then, by keeping track of the subset $A$ of fixed points, we have the recurrence
      \[ |S_n^{B}| = \sum_{A \subseteq [n]} Der_B(n -|A|).\]
      Since $Sch(n)$ and $Der_B(n)$ satisfy the same recurrence, they are the same numbers. 
\end{proof}

\bibliographystyle{plain} 
\bibliography{ref.bib} 

\end{document}